\documentclass[11pt]{article}

\usepackage{amssymb,amsmath,bm}
\usepackage{amsthm}
\usepackage{empheq}
\usepackage{textcomp}
\usepackage{enumerate}      
\usepackage{graphicx}        
\usepackage{caption}

\usepackage{mathrsfs}

\usepackage{url}

\renewcommand{\setminus}{{\smallsetminus}}

\newcommand{\cp}[1]{\vcenter{\hbox{#1}}}

\usepackage{amssymb,amsmath,bm}
\usepackage{amsthm}
\usepackage{hyperref}
\usepackage{mathrsfs}
\usepackage{textcomp}
\usepackage{enumerate}
\usepackage{graphicx}
\usepackage{tikz}
\usepackage{verbatim}

\newtheorem{theorem}{Theorem}[section]
\newtheorem{lemma}[theorem]{Lemma}
\newtheorem{proposition}[theorem]{Proposition}
\newtheorem{definition}[theorem]{Definition}

\newtheorem{conjecture}[theorem]{Conjecture}

\theoremstyle{remark}
\newtheorem{remark}[theorem]{Remark}

\theoremstyle{remark}

\numberwithin{equation}{section}

\newcommand{\lt}{\left}
\newcommand{\rt}{\right}

\def\Vol{\operatorname{Vol}}
\def \CC{\mathbb C}
\def \Re{\operatorname{Re}}
\def \arg{\operatorname{arg}}

\def\SS{{\mathbb S}}
\def\CC{{\mathbb C}}
\def\RR{{\mathbb R}}
\def\ZZ{{\mathbb Z}}
\def\NN{{\mathbb N}}

\def\Vol{\operatorname{Vol}}
\def\CS{\operatorname{CS}}
\def \Re{\operatorname{Re}}
\def \im{\operatorname{Im}}
\def \arg{\operatorname{arg}}
\def \Li{\operatorname{Li_2}}
\def \im{\operatorname{Im}}

\def \det{\operatorname{det}}
\def \Hess{\operatorname{Hess}}

\title{On the asymptotic expansion for the relative Reshetikhin-Turaev invariants of  fundamental shadow link pairs}
\author{Tushar Pandey and Ka Ho Wong}
\date{}

\begin{document}

\maketitle

\begin{abstract}
We study the asymptotic expansion conjecture of the relative Reshetikhin-Turaev invariants proposed in \cite{WY4} for all pairs $(M,L)$ satisfying the property that $M\setminus L$ is homeomorphic to some fundamental shadow link complement. The hyperbolic cone structure of such $(M,L)$ can be described by using the logarithmic holonomies of the meridians of some fundamental shadow link. We show that when the logarithmic holonomies are sufficiently small and all cone angles are less than $\pi$, the asymptotic expansion conjecture of $(M,L)$ is true. Especially, we verify the  asymptotic expansion conjecture of the relative Reshetikhin-Turaev invariants  for all pairs $(M,L)$ satisfying the property that $M\setminus L$ is homeomorphic to some fundamental shadow link complement, with cone angles sufficiently small. Furthermore, we show that if $M$ is obtained by doing rational surgery on a fundamental shadow link complement with sufficiently large surgery coefficients, then the cone angles can be pushed to any value less than $\pi$. \end{abstract}

\tableofcontents

\section{Introduction}

In \cite{WY1}, T. Yang and the second author proposed the volume conjecture for the relative Reshetikhin-Turaev invariants of a pair $(M,L)$, where $M$ is a closed oriented 3-manifold and $L$ is a framed link inside $M$. The conjecture suggests that the asymptotics of the invariants capture the hyperbolic volume and the Chern-Simons invariants of the cone manifold $M$ with the singular locus $L$ and cone angles $\theta$ determined by the sequence of colorings of the framed link.

\begin{conjecture} \label{conj}(\cite[Conjecture 1.1]{WY1})
Let M be a closed oriented 3-manifold and let L be a framed hyperbolic link in M with n components. For an odd integer, $r\geq 3$, let $\mathrm{\bf m}$ = $(\mathrm{ m}_1, \dots, \mathrm{ m}_n)$ and let $RT_r(M,L,\mathrm{\bf m})$ be the $r$-th relative Reshetikhin-Turaev invariant of M with L colored by $\mathrm{\bf m}$ and evaluated at the root of unity $q = e^{\frac{2\pi \sqrt{-1}}{r}}$. For a sequence $\mathrm{\bf m}^{(r)} = (\mathrm{ m}_1^{(r)}, \dots, \mathrm{ m}_n^{(r)}) $, let
\begin{align*}
    \theta_k = \left \lvert 2\pi - \lim_{r \rightarrow \infty} \frac{4\pi m_k^{(r)}}{r} \right\rvert
\end{align*}
and let $\boldsymbol\theta = (\theta_1, .., \theta_n)$. If $M_{L_{\boldsymbol\theta}}$ is a hyperbolic cone manifold consisting of M and a hyperbolic cone metric on M with singular locus L and cone angles $\boldsymbol\theta$, then 
\begin{align*}
    \lim_{r \rightarrow \infty} \frac{4\pi}{r} \log{RT_r(M, L,\mathrm{\bf m}^{(r)})}
    = \Vol(M_{L_{\boldsymbol\theta}}) + \sqrt{-1}\CS(M_{L_{\boldsymbol\theta}})
\end{align*}
where r varies over all positive odd integers.
\end{conjecture}

Notice that when $\mathrm{\bf m}^{(r)} = 0$, the relative Reshetikhin-Turaev invariants of the pair $(M,L)$ is equal to the Resehtikhin-Turaev invariants for the closed oriented 3-manifold $M$ with $\boldsymbol\theta=(2\pi,\dots,2\pi)$. In this case, $M_{L_{\boldsymbol\theta}}$ is the 3-manifold $M$ with the complete hyperbolic metric and Conjecture \ref{conj} recovers the Chen-Yang's volume conjecture for closed, oriented hyperbolic 3-manifolds proposed in \cite{CY}. As a result, Conjecture \ref{conj} not only relates the asymptotics of quantum invariants of pairs with hyperbolic geometry of the underlying cone manifold, but also provides a possible approach to solve the Chen-Yang's volume conjecture. 

In \cite{WY1}, T. Yang and the second author study Conjecture \ref{conj} for all pairs $(M,L)$ obtained by doing a change-of-pair operation from the pair $(M_c, L_{\text{FSL}})$, where $M_c = \#^{c+1} (\SS^2\times \SS^1)$ for some $c\in \NN$ and $L_{\text{FSL}}$ is a fundamental shadow link inside $M_c$. 
 Here we recall from \cite[Proposition 1.3 and 1.4]{WY1} that the change-of-pair operation is a topological move that changes a pair $(M,L)$ to another pair $(M^*,L^*)$ without changing the complement, i.e. $M\setminus L \simeq M^* \setminus L^*$. In particular, $M^*$ can be obtained by doing \emph{integral} Dehn fillings on the boundary components of $M\setminus L$. Moreover, if two pairs $(M,L)$ and $(M^*,L^*)$ share the same complement, i.e. $M\setminus L \simeq M^*\setminus L^*$, then they are related by a sequence of change-of-pair operations. In this case, $M^*$ can be obtained by doing \emph{rational} Dehn fillings on  the boundary components of $M\setminus L$. 
In \cite{WY1}, Conjecture \ref{conj} has been proved for all pairs $(M,L)$ obtained by doing a change-of-pair operation from the pair $(M_c, L_{\text{FSL}})$, where $M_c = \#^{c+1} (\SS^2\times \SS^1)$ for some $c\in \NN$ and $L_{\text{FSL}}$ is a fundamental shadow link inside $M_c$, with sufficiently small cone angles. 
Especially, since every closed, oriented 3-manifolds can be obtained by doing Dehn-filling on the boundary of some fundamental shadow link complement \cite{CT}, if the cone angle can be pushed from sufficiently close to 0 all the way to $2\pi$, then one can prove the Chen-Yang's volume conjecture. Besides, to show that it is possible to push the cone angle, in \cite{WY2}, T. Yang and the second author proved Conjecture \ref{conj} for all pairs $(M,L)$ with $M\setminus L$ homeomorphic to the figure eight knot complement in $\SS^3$, for all cone angle from $0$ to $2\pi$, except finitely many cases corresponding to the exceptional surgery of the figure eight knot. 
Furthermore, in \cite{WY4}, T. Yang and the second author refined Conjectue \ref{conj} by studying the asymptotic expansion formula of the relative Reshetikhin-Turaev invariants. Let $M$ be a closed oriented $3$-manifold and let $L$ be a framed hyperbolic link in $M$ with $n$ components. Let $\{\mathrm{\bf m}^{(r)}\}=\{(\mathrm{ m}^{(r)}_1,\dots,\mathrm{ m}^{(r)}_{n})\}$ be a sequence of colorings of the components of $L$ by the elements of $\{0,\dots,r-2\}$ such that for each $k\in\{1,\dots,n\},$ either $\mathrm{m}_k^{(r)}> \frac{r}{2}$ for all $r$ sufficiently large or $\mathrm{m}_k^{(r)}< \frac{r}{2}$ for all $r$ sufficiently large. In the former case we let $\mu_k=1$ and in the latter case we let $\mu_k=-1,$ and we let 
$$\theta^{(r)}_k=\mu_k\bigg(\frac{4\pi \mathrm{ m}^{(r)}_k}{r}-2\pi\bigg).$$
Let $\boldsymbol\theta^{(r)}=(\theta^{(r)}_1,\dots,\theta^{(r)}_{n}).$ Suppose for all $r$ sufficiently large, a hyperbolic cone metric on $M$ with singular locus $L$ and cone angles $\boldsymbol\theta^{(r)}$ exists. We denote $M$ with such a hyperbolic cone metric by $M^{(r)},$  let $\mathrm{Vol}(M^{(r)})$ and $\mathrm{CS}(M^{(r)})$ respectively be the volume and the Chern-Simons invariant of $M^{(r)},$ and let $\mathrm H^{(r)}(\gamma_1),\dots, \mathrm H^{(r)}(\gamma_{n})$ be the logarithmic holonomies in $M^{(r)}$ of the parallel copies $(\gamma_1,\dots,\gamma_{n})$ of the core curves of $L$ given by the framing. Let $\rho_{M^{(r)}}:\pi_1(M\setminus L)\to \mathrm{PSL}(2;\mathbb C)$ be the holonomy representation of  the restriction of $M^{(r)}$ to $M\setminus L,$ and let $\mathbb{T}_{(M\setminus L, \mathbf\Upsilon)}([\rho_{M^{(r)}}])$ be the Reidemeister torsion of $M\setminus L$  twisted by the adjoint action of  $\rho_{M^{(r)}}$ with respect to the system of meridians $\mathbf \Upsilon$ of a tubular neighborhood of the core curves of $L$ (see Section \ref{TRT} for more details).

\begin{conjecture}\label{CRRT}(\cite[Conjecture 1.1]{WY4}) Suppose $\{\theta^{(r)}\}$ converges as $r$ tends to infinity.
Then as $r$ varies over all positive odd integers and at $q=e^{\frac{2\pi \sqrt{-1}}{r}},$ the relative Reshetikhin-Turaev invariants 
$$ \mathrm{RT}_r(M,L,\mathbf m^{(r)})=C\frac{e^{\frac{1}{2}\sum_{k=1}^{n}\mu_k\mathrm H^{(r)}(\gamma_k)}}{\sqrt{\pm\mathbb{T}_{(M\setminus L, \mathrm{\mathbf\Upsilon})}([\rho_{M^{(r)}}])}}e^{\frac{r}{4\pi}\big(\mathrm{Vol}(M^{(r)})+\sqrt{-1}\mathrm{CS}(M^{(r)})\big)}\bigg(1+O\Big(\frac{1}{r}\Big)\bigg),$$
where $C$ is a quantity of norm $1$ independent of the geometric structure on $M.$
\end{conjecture}
In \cite{WY4}, Conjecture \ref{CRRT} has been proved for all pairs $(M,L)$ obtained by doing a change-of-pair operation from the pair $(M_c, L_{\text{FSL}})$ with sufficiently small cone angles. 
Similar to the relationship between Conjecture \ref{conj} and the Chen-Yang's volume conjecture, Conjecture \ref{CRRT} also provides a new approach to understand the asymptotic expansion of the Reshetikhin-Turaev invariants for closed, oriented 3-manifolds discussed in \cite{GRY} and \cite{OT}. 

In this paper, we combine the ideas in \cite{WY}, \cite{WY1}, \cite{WY2} and \cite{WY4} to study the asymptotic expansion conjecture for the relative Reshetikhin-Turaev invariants of any pair $(M,L)$ with $M\setminus L$ homeomorphic to some $M_c\setminus L_{\text{FSL}}$. Let $\mathrm{H}(u_1),\dots,\mathrm{H}(u_n)$ be the logarithmic holonomies of the meridians of $L_{\text{FSL}}\subset M_c$. Near the complete structure, for any pair $(M,L)$ with $M\setminus L$ homeomorphic to $M_c\setminus L_{\text{FSL}}$, the hyperbolic cone structure of $(M,L)$ can be described by using the parameters $\mathrm{H}(u_1),\dots,\mathrm{H}(u_n)$. 
\begin{theorem}\label{mainthm0}
Given a fundamental shadow link $L_{\text{FSL}} \subset M_c$ with $n$ components. There exists $\delta>0$ (dependent of $L_{\text{FSL}}$) such that if $(M,L)$ is a pair with $M\setminus L$ homeomorphic to $M_c \setminus L_{\text{FSL}}$ and with a hyperbolic cone structure satisfying $|\mathrm{H}(u_k)|<\delta$ and $\theta_k\in [0,\pi)$ for all $k=1,\dots,n$, then Conjecture \ref{CRRT} is true for $(M,L)$.
\end{theorem}

As a special case of Theorem \ref{mainthm0}, 
\begin{theorem}\label{mainthm}
Given a fundamental shadow link $L_{\text{FSL}} \subset M_c$, if $(M,L)$ is a pair with $M\setminus L$ homeomorphic to $M_c \setminus L_{\text{FSL}}$, 
then there exists $\epsilon>0$ such that Conjecture \ref{CRRT} is true for $(M,L)$ for any cone angles $\boldsymbol\theta \in [0,\epsilon)^n$.
\end{theorem}
Note that $M$ in Theorem \ref{mainthm} covers all 3-manifolds $M$ obtained by doing surgery on some fundamental shadow link complement. 
It is expected that in Theorem \ref{mainthm}. when $M$ is a closed, oriented hyperbolic 3-manifold, the cone angles can be pushed to $2\pi$ so that the Chen-Yang's volume conjecture for the Reshetikhin-Turaev invariants of $M$ is true. In this paper, we restrict our attention to the case where $M$ is a hyperbolic 3-manifolds obtained by doing rational surgery on some fundamental shadow link complement with sufficiently large surgery coefficients and all the cone angles are less than $\pi$.

\begin{theorem}\label{mainthmlarge}
Given a fundamental shadow link $L_{\text{FSL}}\subset M_c$ with $n$ components, there exists a constant $C>0$ (depending of $L_{\text{FSL}}$) such that if
\begin{itemize}
\item $M\setminus L$ is homeomorphic to $M_c \setminus L_{\text{FSL}}$; and
\item $M$ is obtained by doing a $\{(p_k,q_k)\}_{k=1}^n$ surgery on the boundaries of $M_c\setminus L_{\text{FSL}}$ with 
$$|p_k| + |q_k| > C$$
for all $k=1,\dots,n$,
\end{itemize}
then Conjecture \ref{CRRT} is true for $(M,L)$ for any cone angles $\boldsymbol\theta\in[0,\pi)^n$.
\end{theorem}

Since doing rational surgery on fundamental shadow link complements provides more possible way to obtain the same closed, oriented 3-manifold, we obtain more possible testing examples to solve the Chen-Yang's volume conjecture. The following result follows immediately from Theorem \ref{mainthm} and \ref{mainthmlarge}.

\begin{theorem}\label{Cor}
Conjecture \ref{conj} is true for all the pair $(M,L)$ with $M\setminus L$ homeomorphic to some fundamental shadow link complement, with small cone angles.
\end{theorem}

\begin{theorem}\label{Cor2}
Conjecture \ref{conj} is true for all the pair $(M,L)$ described in Theorem \ref{mainthmlarge}, with all cone angles less than $\pi$.
\end{theorem}

\subsection*{Plan of this paper}
In Section 2, we give a brief review for the preliminary knowledge required for the proof of Theorem \ref{mainthm}. The materials in this section can be found in \cite{WY, WY1, WY3, WY4}. In Section 3, we compute the relative Reshetikhin-Turaev invariants of $(M,L)$ and express it as a sum of the evaluation of certain holomorphic function at some integral points (Proposition \ref{computation}). An important step is to use a Gauss sum formula (Lemma \ref{sm}) to simplify the relative Reshetikhin-Turaev invariants. In Section 4, we apply the Poisson summation formula to write the invariants as the sum of the Fourier coefficients together with some error terms (Proposition \ref{formulaFC} and \ref{Fperror}). We also gives a simplified expression for the leading Fourier coefficients (Proposition \ref{lfcexpress}). 
In Section 5, we apply the saddle point approximation (Proposition \ref{saddle}) to study the asymptotic expansions of those Fourier coefficients. To do that, in Proposition \ref{prop52}, we show that certain critical values of the function in this Fourier coefficients give the complex volume of the cone manifold. The key observation is that the critical point equations of the function involved coincide with the cone angle equations of the cone manifold $M$ with singular locus $L$. Moreover, in Proposition \ref{lfc}, we verify that under certain technical assumptions, all the conditions required for applying the saddle point method are satisfied. In Proposition \ref{Rtorsion}, we show that the twisted Reideimester torsion appears in the asymptotics of the leading Fourier coefficient. The main idea is to apply the relationship between the torsion and the Gram matrix function studied in \cite{WY3} and \cite{WY4}. In Proposition \ref{0FCasym}, we obtain the asymptotic expansions for the leading Fourier coefficients, which capture the complex volume and the twisted Reidemeister torsion of the manifold with the cone structure determined by the sequence of colorings. 
Finally, in Proposition \ref{fail0}, \ref{fail1}, \ref{fail2} and \ref{errorest}, we show under certain technical assumption, the sum of all the other Fourier coefficients and the error term in Proposition \ref{Fperror} are negligible. In Lemma \ref{assat0}, \ref{assat1} and \ref{assat2}, we show respectively that in the contexts of Theorem \ref{mainthm0}, \ref{mainthm} and \ref{mainthmlarge}, all the technical assumptions mentioned above are satisfied. This completes the proof of the main theorems.

\subsection*{Acknowledgement} The authors would like to thank their supervisor Tian Yang for his guidance.

\section{Preliminaries}
The materials in this section are from \cite{WY, WY1, WY3, WY4}. We include the materials here for the reader's convinence.

\subsection{Relative Reshetikhin-Turaev invariants}\label{RRT}

In this article we will follow the skein theoretical approach of the relative Reshetikhin-Turaev invariants\,\cite{BHMV, Li} and focus on the $SO(3)$-theory and the values at the root of unity  $q=e^{\frac{2\pi\sqrt{-1}}{r}}$ for odd integers $r\geqslant 3.$

A framed link in an oriented $3$-manifold $M$ is a smooth embedding $L$ of a disjoint union of finitely many thickened circles $\mathrm S^1\times [0,\epsilon],$ for some $\epsilon>0,$ into $M.$ The Kauffman bracket skein module $\mathrm K_r(M)$ of $M$ is the $\mathbb C$-module generated by the isotopic classes of framed links in $M$  modulo the follow two relations: 

\begin{enumerate}[(1)]
\item  \emph{Kauffman Bracket Skein Relation:} \ $\cp{\includegraphics[width=1cm]{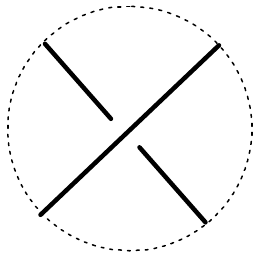}}\ =\ e^{\frac{\pi\sqrt{-1}}{r}}\ \cp{\includegraphics[width=1cm]{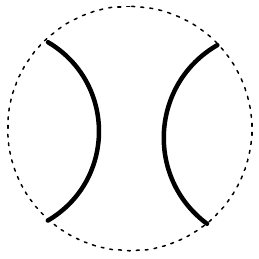}}\  +\ e^{-\frac{\pi\sqrt{-1}}{r}}\ \cp{\includegraphics[width=1cm]{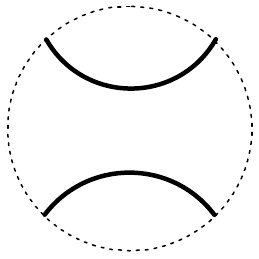}}.$ 

\item \emph{Framing Relation:} \ $L \cup \cp{\includegraphics[width=0.8cm]{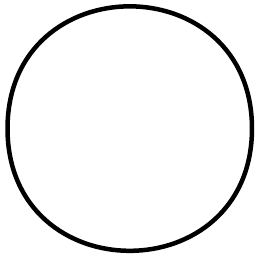}}=(-e^{\frac{2\pi\sqrt{-1}}{r}}-e^{-\frac{2\pi\sqrt{-1}}{r}})\ L.$ 
\end{enumerate}

There is a canonical isomorphism 
$$\langle\ \rangle:\mathrm K_r(\mathrm S^3)\to\mathbb C$$
defined by sending the empty link to $1.$ The image $\langle L\rangle$ of the framed link $L$ is called the Kauffman bracket of $L.$

Let $\mathrm K_r(A\times [0,1])$ be the Kauffman bracket skein module of the product of an annulus $A$ with a closed interval. For any link diagram $D$ in $\mathbb R^2$ with $k$ ordered components and $b_1, \dots, b_k\in \mathrm K_r(A\times [0,1]),$ let 
$$\langle b_1,\dots, b_k\rangle_D$$
be the complex number obtained by cabling $b_1,\dots, b_k$ along the components of $D$ considered as a element of $K_r(\mathrm S^3)$ then taking the Kauffman bracket $\langle\ \rangle.$

On $\mathrm K_r(A\times [0,1])$ there is a commutative multiplication induced by the juxtaposition of annuli, making it a $\mathbb C$-algebra; and as a $\mathbb C$-algebra $\mathrm K_r(A\times [0,1])  \cong \mathbb C[z],$ where $z$ is the core curve of $A.$ For an integer $n\geqslant 0,$ let $e_n(z)$ be the $n$-th Chebyshev polynomial defined recursively by
$e_0(z)=1,$ $e_1(z)=z$ and $e_n(z)=ze_{n-1}(z)-e_{n-2}(z).$ Let $\mathrm{I}_r=\{0,2,\dots,r-3\}$ be the set of even integers in between $0$ and $r-2.$ Then the Kirby coloring $\Omega_r\in\mathrm K_r(A\times [0,1])$ is defined by 
$$\Omega_r=\mu_r\sum_{n\in \mathrm{I}_r}[n+1]e_{n},$$
where
 $$\mu_r=\frac{2\sin\frac{2\pi}{r}}{\sqrt r}$$ 
 and $[n]$ is the quantum integer defined by
$$[n]=\frac{e^{\frac{2n\pi\sqrt{-1}}{r}}-e^{-\frac{2n\pi\sqrt{-1}}{r}}}{e^{\frac{2\pi\sqrt{-1}}{r}}-e^{-\frac{2\pi\sqrt{-1}}{r}}}.$$

Let $M$ be a closed oriented $3$-manifold and let $L$ be a framed link in $M$ with $n$ components. Suppose $M$ is obtained from $S^3$ by doing a surgery along a framed link $L',$ $D(L')$ is a standard diagram of $L'$ (ie, the blackboard framing of $D(L')$ coincides with the framing of $L'$). Then $L$ adds extra components to $D(L')$ forming a linking diagram $D(L\cup L')$ with $D(L)$ and $D(L')$ linking in possibly a complicated way. Let
$U_+$ be the diagram of the unknot with framing $1,$ $\sigma(L')$ be the signature of the linking matrix of $L'$ and $\mathbf m=(m_1,\dots,m_n)$ be a multi-elements of $I_r.$ Then the $r$-th \emph{relative Reshetikhin-Turaev invariant of $M$ with $L$ colored by $\mathbf m$} is defined as
\begin{equation}\label{RT}
\mathrm{RT}_r(M,L,\mathbf m)=\mu_r \langle e_{m_1},\dots,e_{m_n}, \Omega_r, \dots, \Omega_r\rangle_{D(L\cup L')}\langle \Omega_r\rangle _{U_+}^{-\sigma(L')}.
\end{equation}

Note that if $L=\emptyset$ or $m_1=\dots=m_n=0,$ then $\mathrm{RT}_r(M,L,\mathbf m)=\mathrm{RT}_r(M),$ the $r$-th Reshetikhin-Turaev invariant of $M;$ and if $M=S^3,$ then $\mathrm{RT}_r(M,L,\mathbf m)=\mu_r\mathrm J_{\mathbf m, L}(q^2),$ the value of the $\mathbf m$-th unnormalized colored Jones polynomial of $L$ at $t=q^2.$

\subsection{Hyperbolic cone manifolds}\label{CCS}

According to \cite{CHK}, a $3$-dimensional \emph{hyperbolic cone-manifold} is a $3$-manifold $M,$ which can be triangulated so that the link of each simplex is piecewise linear homeomorphic to a standard sphere and $M$ is equipped with a complete path metric such that the restriction of the metric to each simplex is isometric to a hyperbolic geodesic simplex. The \emph{singular locus} $L$ of a cone-manifold $M$ consists of the points with no neighborhood isometric to a ball in a Riemannian manifold. It follows that
\begin{enumerate}[(1)]
\item $L$ is a link in $M$ such that each component is a closed geodesic.

\item At each point of $L$ there is a \emph{cone angle} $\theta$ which is the sum of dihedral angles of $3$-simplices containing the point.

\item The restriction of the metric on $M\setminus L$ is a smooth hyperbolic metric, but is incomplete if $L\neq \emptyset.$
\end{enumerate}

Hodgson-Kerckhoff\,\cite{HK} proved that hyperbolic cone metrics on $M$ with singular locus $L$ are locally parametrized by the cone angles provided all the cone angles are less than or equal to $2\pi,$ and Kojima\,\cite{Ko} proved that hyperbolic cone manifolds $(M,L)$ are globally rigid provided all the cone angles are less than or equal to $\pi.$ It is expected to be  globally rigid if all the cone angles are less than or equal to $2\pi.$

Given a $3$-manifold $N$ with boundary a union of tori $T_1,\dots,T_n,$ a choice of generators $(u_i,v_i)$ for each $\pi_1(T_i)$ and pairs of relatively prime integers $(p_i, q_i),$ one can do the $(\frac{p_1}{q_1},\dots,\frac{p_n}{q_n})$-Dehn filling on $N$ by attaching a solid torus to each $T_i$ so that $p_iu_i+q_iv_i$ bounds a disk.  If $\mathrm H(u_i)$ and  $\mathrm H(v_i)$ are respectively the logarithmic holonomy for $u_i$ and $v_i,$ then a solution to  
\begin{equation}\label{DF}
p_i\mathrm H(u_i)+q_i\mathrm H(v_i)=\sqrt{-1}\theta_i
\end{equation}
 near the complete structure gives a cone-manifold structure on the resulting manifold $M$ with the cone angle $\theta_i$ along the core curve $L_i$ of the solid torus attached to $T_i;$ it is a smooth structure if $\theta_1=\dots=\theta_n= 2\pi.$

In this setting, the Chern-Simons invariant for a hyperbolic cone manifold $(M,L)$ can be defined by using the Neumann-Zagier potential function\,\cite{NZ}. To do this, we need a framing on each component, namely, a choice of a curve $\gamma_i$ on $T_i$ that is isotopic to the core curve $L_i$ of the solid torus attached to $T_i.$ We choose the orientation of $\gamma_i$ so that $(p_iu_i+q_iv_i)\cdot \gamma_i=1.$ Then we consider the following function 
$$\frac{\Phi(\mathrm{H}(u_1),\dots,\mathrm{H}(u_n))}{\sqrt{-1}}-\sum_{i=1}^n\frac{\mathrm H(u_i)\mathrm H(v_i)}{4\sqrt{-1}}+\sum_{i=1}^n\frac{\theta_i\mathrm H(\gamma_i)}{4},$$
where $\Phi$ is the Neumann-Zagier potential function (see \cite{NZ}) defined on the deformation space of hyperbolic structures on $M\setminus L$ parametrized by the holonomy of the meridians $\{\mathrm H(u_i)\},$ characterized by 
\begin{equation}\label{char}
\left \{\begin{array}{l}
\frac{\partial \Phi(\mathrm{H}(u_1),\dots,\mathrm{H}(u_n))}{\partial \mathrm{H}(u_i)}=\frac{\mathrm H(v_i)}{2},\\
\\
\Phi(0,\dots,0)=\sqrt{-1}\bigg(\mathrm{Vol}(M\setminus L)+\sqrt{-1}\mathrm{CS}(M\setminus L)\bigg)\quad \mod\ \pi^2\mathbb Z,
\end{array}\right.
\end{equation} 
where $M\setminus L$ is with the complete hyperbolic metric. Another important feature of $\Phi$ is that it is even in each of its variables $\mathrm H(u_i).$

Following the argument in \cite[Sections 4 \& 5]{NZ}, one can prove that if  the cone angles of components of $L$ are $\theta_1,\dots,\theta_n,$ then
\begin{equation}\label{VOL}
\mathrm{Vol}(M_{L_{\boldsymbol\theta}})=\mathrm{Re}\bigg(\frac{\Phi(\mathrm{H}(u_1),\dots,\mathrm{H}(u_n))}{\sqrt{-1}}-\sum_{i=1}^n\frac{\mathrm H(u_i)\mathrm H(v_i)}{4\sqrt{-1}}+\sum_{i=1}^n\frac{\theta_i\mathrm H(\gamma_i)}{4}\bigg).
\end{equation}
Indeed, in this case, one can replace the $2\pi$ in Equations (33) (34) and (35) of \cite{NZ} by $\theta_i,$ and as a consequence can replace the $\frac{\pi}{2}$ in Equations (45), (46) and (48) by $\frac{\theta_i}{4},$ proving the result.

In \cite{Y}, Yoshida proved that when $\theta_1=\dots=\theta_n=2\pi,$
$$\mathrm{Vol}(M)+\sqrt{-1}\mathrm{CS}(M)=\frac{\Phi(\mathrm{H}(u_1),\dots,\mathrm{H}(u_n))}{\sqrt{-1}}-\sum_{i=1}^n\frac{\mathrm H(u_i)\mathrm H(v_i)}{4\sqrt{-1}}+\sum_{i=1}^n\frac{\theta_i\mathrm H(\gamma_i)}{4}\quad\mod\sqrt{-1}\pi^2\mathbb Z.$$

Therefore, we can make the following 

\begin{definition}\label{CS} The \emph{Chern-Simons invariant} of a hyperbolic cone manifold $M_{L_{\boldsymbol\theta}}$ with a choice of the framing $(\gamma_1,\dots,\gamma_n)$ 
is defined as
$$\mathrm{CS}(M_{L_{\boldsymbol\theta}})=\mathrm{Im}\bigg(\frac{\Phi(\mathrm{H}(u_1),\dots,\mathrm{H}(u_n))}{\sqrt{-1}}-\sum_{i=1}^n\frac{\mathrm H(u_i)\mathrm H(v_i)}{4\sqrt{-1}}+\sum_{i=1}^n\frac{\theta_i\mathrm H(\gamma_i)}{4}\bigg)\quad\mod \pi^2\mathbb Z.$$
\end{definition}

Then together with (\ref{VOL}), we have
\begin{equation}\label{VCS}
\mathrm{Vol}(M_{L_{\boldsymbol\theta}})+\sqrt{-1}\mathrm{CS}(M_{L_{\boldsymbol\theta}})=\frac{\Phi(\mathrm{H}(u_1),\dots,\mathrm{H}(u_n))}{\sqrt{-1}}-\sum_{i=1}^n\frac{\mathrm H(u_i)\mathrm H(v_i)}{4\sqrt{-1}}+\sum_{i=1}^n\frac{\theta_i\mathrm H(\gamma_i)}{4}\quad\mod\sqrt{-1}\pi^2\mathbb Z.
\end{equation}


\subsection{Quantum $6j$-symbols} 

A triple $(m_1,m_2,m_3)$ of even integers in $\{0,2,\dots,r-3\}$ is \emph{$r$-admissible} if 
\begin{enumerate}[(1)]
\item  $m_i+m_j-m_k\geqslant 0$ for $\{i,j,k\}=\{1,2,3\},$
\item $m_1+m_2+m_3\leqslant 2(r-2).$
\end{enumerate}

Recall that for $n\in \ZZ_{\geq 0}$, the quantum factorial $[n]!$ is defined by $[0]!=1$ and 
$$ [n]! = \prod_{k=1}^n [k]$$
for $n\geq 0$. 
For an $r$-admissible triple $(m_1,m_2,m_3),$ define 
$$\Delta(m_1,m_2,m_3)=\sqrt{\frac{[\frac{m_1+m_2-m_3}{2}]![\frac{m_2+m_3-m_1}{2}]![\frac{m_3+m_1-m_2}{2}]!}{[\frac{m_1+m_2+m_3}{2}+1]!}}$$
with the convention that $\sqrt{x}=\sqrt{|x|}\sqrt{-1}$ when the real number $x$ is negative.

A  6-tuple $(m_1,\dots,m_6)$  is \emph{$r$-admissible} if the triples $(m_1,m_2,m_3),$ $(m_1,m_5,m_6),$ $(m_2,m_4,m_6)$ and $(m_3,m_4,m_5)$ are $r$-admissible

\begin{definition}
The \emph{quantum $6j$-symbol} of an $r$-admissible 6-tuple $(m_1,\dots,m_6)$ is 
\begin{multline*}
\bigg|\begin{matrix} m_1 & m_2 & m_3 \\ m_4 & m_5 & m_6 \end{matrix} \bigg|
= \sqrt{-1}^{-\sum_{i=1}^6m_i}\Delta(m_1,m_2,m_3)\Delta(m_1,m_5,m_6)\Delta(m_2,m_4,m_6)\Delta(m_3,m_4,m_5)\\
\sum_{k=\max \{T_1, T_2, T_3, T_4\}}^{\min\{ Q_1,Q_2,Q_3\}}\frac{(-1)^k[k+1]!}{[k-T_1]![k-T_2]![k-T_3]![k-T_4]![Q_1-k]![Q_2-k]![Q_3-k]!},
\end{multline*}
where $T_1=\frac{m_1+m_2+m_3}{2},$ $T_2=\frac{m_1+m_5+m_6}{2},$ $T_3=\frac{m_2+m_4+m_6}{2}$ and $T_4=\frac{m_3+m_4+m_5}{2},$ $Q_1=\frac{m_1+m_2+m_4+m_5}{2},$ $Q_2=\frac{m_1+m_3+m_4+m_6}{2}$ and $Q_3=\frac{m_2+m_3+m_5+m_6}{2}.$
\end{definition}

\begin{definition}\label{defhyperidealm} An $r$-admissible $6$-tuple $(m_1,\dots,m_6)$  is of the \emph{hyperideal type} if for $\{i,j,k\}=\{1,2,3\},$ $\{1,5,6\},$ $\{2,4,6\}$ and $\{3,4,5\},$
\begin{enumerate}[(1)]
\item $0\leqslant m_i+m_j-m_k<r-2,$ and
\item $r-2\leqslant m_i+m_j+m_k\leqslant 2(r-2).$
\end{enumerate}
\end{definition}

Here we recall a classical result of Costantino \cite{C1} which was originally stated at the root of unity $q=e^{\frac{\pi \sqrt{-1}}{r}}.$ At the root of unity $q=e^{\frac{2\pi\sqrt{-1}}{r}},$ see \cite[Appendix]{BDKY} for a detailed proof.

\begin{theorem}[\cite{C1}]\label{Vol}
Let $\{(m_1^{(r)},\dots,m_6^{(r)})\}$ be a sequence of $r$-admissible
$6$-tuples, and
let 
$$\theta_i=\Big|\pi-\lim_{r\rightarrow\infty}\frac{2\pi m_i^{(r)}}{r}\Big|.$$
If $\theta_1,\dots,\theta_6$ are the dihedral angles of a truncated hyperideal tetrahedron $\Delta,$ then
as $r$ varies over all the odd integers 
$$\lim_{r\to\infty}\frac{2\pi}{r}\log \bigg|\begin{array}{ccc}m_1^{(r)} & m_2^{(r)} & m_3^{(r)} \\m_4^{(r)} & m_5^{(r)} & m_6^{(r)} \\\end{array} \bigg|_{q=e^{\frac{2\pi \sqrt{-1}}{r}}}=\mathrm{Vol}(\Delta).$$
\end{theorem}

Closely related, a triple $(\alpha_1,\alpha_2,\alpha_3)\in [0,2\pi]^3$ is \emph{admissible} if 
\begin{enumerate}[(1)]
\item $\alpha_i+\alpha_j-\alpha_k\geqslant 0$ for $\{i,j,k\}=\{1,2,3\},$
\item $\alpha_i+\alpha_j+\alpha_k\leqslant 4\pi.$
\end{enumerate}
A $6$-tuple $(\alpha_1,\dots,\alpha_6)\in [0,2\pi]^6$ is \emph{admissible} if the triples $\{1,2,3\},$ $\{1,5,6\},$ $\{2,4,6\}$ and $\{3,4,5\}$ are admissible.

\begin{definition}\label{defhyperideal} A $6$-tuple $(\alpha_1,\dots,\alpha_6)\in [0,2\pi]^6$ is of the \emph{hyperideal type} if for $\{i,j,k\}=\{1,2,3\},$ $\{1,5,6\},$ $\{2,4,6\}$ and $\{3,4,5\},$
\begin{enumerate}[(1)]
\item $0\leqslant \alpha_i+\alpha_j-\alpha_k\leqslant 2\pi,$ and
\item $2\pi\leqslant \alpha_i+\alpha_j+\alpha_k\leqslant 4\pi.$
\end{enumerate}
\end{definition}


\subsection{Fundamental shadow links}\label{fsl}

In this section we recall the construction and basic properties of the fundamental shadow links. The building blocks for the fundamental shadow links are truncated tetrahedra as in the left of Figure \ref{bb}. If we take $c$ building blocks $\Delta_1,\dots, \Delta_c$ and glue them together along the triangles of truncation, we obtain a (possibly non-orientable) handlebody of genus $c+1$ with a link in its boundary consisting of the edges of the building blocks, such as in the right of Figure \ref{bb}. By taking the orientable double (the orientable double
covering with the boundary quotient out by the deck involution) of this handlebody, we obtain a link $L_{\text{FSL}}$ inside $M_c=\#^{c+1}(S^2\times S^1).$ We call a link obtained this way a \emph{fundamental shadow link}, and its complement in $M_c$ a \emph{fundamental shadow link complement}.

\begin{figure}[htbp]
\centering
\includegraphics[scale=0.25]{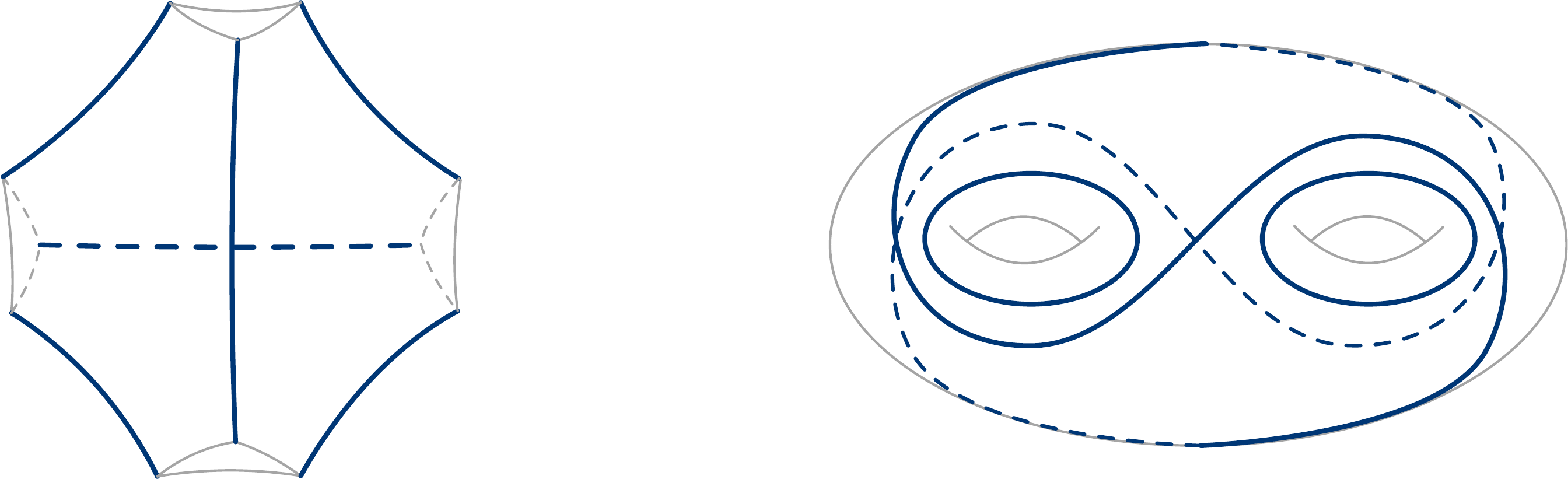}
\caption{The handlebody on the right is obtained from the truncated tetrahedron on the left by identifying the triangles on the top and the bottom by a horizontal reflection and the triangles on the left and the right by a vertical reflection.}
\label{bb}
\end{figure}

The fundamental importance of the family of the fundamental shadow links is the following.

\begin{theorem}[\cite{CT}]\label{CT} Any compact oriented $3$-manifold with toroidal or empty boundary can be obtained from a suitable fundamental shadow link complement by doing an integral Dehn-filling to some of the boundary components.
\end{theorem}

A hyperbolic cone metric on $M_c$ with singular locus $L_{\text{FSL}}$ and with sufficiently small cone angles $\theta_1,\dots,\theta_n$ can be constructed as follows. For each $s\in \{1,\dots, c\},$ let $e_{s_1},\dots,e_{s_6}$ be the edges of the building block $\Delta_s,$ and $\theta_{s_j}$ be the cone angle of the component of $L$ containing $e_{s_j}.$ If $\theta_i$'s are sufficiently small, then $\big\{\frac{\theta_{s_1}}{2},\dots,\frac{\theta_{s_6}}{2}\big\}$ form the set of dihedral angles of a truncated hyperideal tetrahedron, by abuse of notation still denoted by $\Delta_s.$ Then the hyperbolic cone manifold $M_c$ with singular locus $L_{\text{FSL}}$ and cone angles $\theta_1,\dots,\theta_n$ is obtained by glueing $\Delta_s$'s together along isometries of the triangles of truncation, and taking the double. In this metric, the logarithmic holonomy of the meridian $u_i$ of the tubular neighborhood $N(L_i)$ of $L_i$ satisfies 
\begin{equation}\label{m}
\mathrm{H}(u_i)=\sqrt{-1}\theta_i.
\end{equation}
A preferred longitude $v_i$ on the boundary of $N(L_i)$ can be chosen as follows. Recall that a fundamental shadow link is obtained from the double of a set of truncated tetrahedra (along the hexagonal faces) glued together by orientation preserving homeomorphisms between the trice-punctured spheres coming from the double of the triangles of truncation, and recall also that 
 the mapping class group of trice-punctured sphere is generated by mutations, which could be represented by the four $3$-braids in Figure \ref{mutation}. For each mutation, we assign an integer $\pm 1$ to each component of the braid as in Figure \ref{mutation}; and for a composition of a sequence of mutations, we assign the sum of the  $\pm 1$ assigned by the mutations to each component of the $3$-braid.

 \begin{figure}[htbp]
\centering
\includegraphics[scale=0.5]{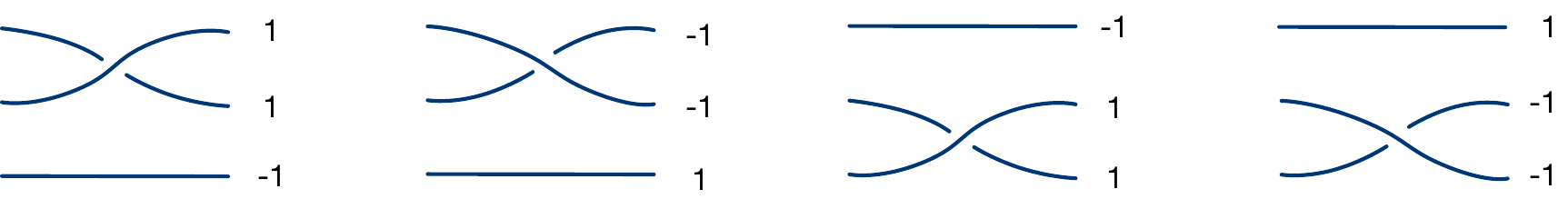}
\caption{ }
\label{mutation}
\end{figure}

  In this way, each  orientation preserving homeomorphisms between the trice-punctured spheres assigns three integers to three of the components of $L_{\text{FSL}},$ one for each. For each $i\in \{1,\dots, n\},$ let $\iota_i$ be the sum of all the integers on $L_i$ assigned by the homeomorphisms between the trice-punctured spheres. Then we can choose a preferred longitude $v_i$ such that $u_i\cdot v_i=1$ and the logarithmic holomony satisfies
\begin{equation}\label{l}
\mathrm{H}(v_i)=-l_i+\frac{\iota_i\sqrt{-1}\theta_i}{2},
\end{equation}
where $l_i$ is the length of the closed geodesic $L_i.$ In this way, a framing on $L_i$ gives an integer $p_i$ in the way that the parallel copy of $L_i$ on $N(L_i)$ is isotopic to the curve representing $p_iu_i+v_i.$

\begin{proposition}[\cite{C1,C2}]\label{FSL} If $L_{\text{FSL}}=L_1\cup\dots\cup L_n\subset M_c$ is a framed fundamental shadow link with framing $p_i$ on $L_i,$ and $\mathbf m=(m_1,\dots,m_n)$ is a coloring of its components with even integers in $\{0,2,\dots,r-3\},$ then
$$\mathrm{RT_r}(M_c,L_{\text{FSL}},\mathbf m)=\Bigg(\frac{2\sin\frac{2\pi}{r}}{\sqrt{r}}\Bigg)^{-c}\prod_{i=1}^n(-1)^{\frac{\iota_im_i}{2}}q^{(p_i+\frac{\iota_i}{2})\frac{m_i(m_i+2)}{2}}\prod_{s=1}^c\bigg|\begin{matrix}
        m_{s_1} & m_{s_2} & m_{s_3} \\
        m_{s_4} & m_{s_5} & m_{s_6} 
      \end{matrix} \bigg|,$$
where $m_{s_1},\dots,m_{s_6}$ are the colors of the edges of the building block $\Delta_s$ inherited  from the color $\mathbf m$ on $L_{\text{FSL}}.$    
\end{proposition}

Next, we talk about the volume and the Chern-Simons invariant of $M_c \setminus L_{\text{FSL}}$ at the complete hyperbolic structure. In the complete hyperbolic metric, since $M_c \setminus L_{\text{FSL}}$ is the union of $2c$ regular ideal octahedra, we have 
\begin{equation} \label{VolFSL}
\mathrm{Vol}(M_c \setminus L_{\text{FSL}}) = 2c v_8.
\end{equation} 

For the Chern-Simons invariant, in the case that the truncated tetrahedra $\Delta_1,\dots, \Delta_c$ are glued together along the triangles of truncation via orientation reversing maps, $M_c\setminus L_{\text{FSL}}$ is the ordinary double of the orientable handlebody, which admits an orientation reversing self-homeomorphism. Hence by \cite[Corollary 2.5]{MO}, \begin{equation*}
\mathrm {CS}(M_c \setminus L_{\text{FSL}}) = 0      \quad\quad\text{mod }\pi^2\mathbb Z
\end{equation*}
at the complete hyperbolic structure. In the general case, a fundamental shadow link complement $M_c\setminus L_{\text{FSL}}$ can be obtained from one from the previous case by doing a sequence of mutations along the thrice-punctured spheres coming from the double of the triangles of truncation. Therefore, by \cite[Theorem 2.4]{MR} that a mutation along an incompressible trice-punctured sphere in a hyperbolic three manifold changes the Chern-Simons invariant by $\frac{\pi^2}{2},$ we have
\begin{equation}\label{CSFSL}
\mathrm {CS}(M_c \setminus L_{\text{FSL}}) = \Big(\sum_{i=1}^n \frac{\iota_i}{2}\Big) \pi^2  \quad\quad\text{mod }\pi^2\mathbb Z.
\end{equation}

Together with Theorem \ref{Vol} and the construction of the hyperbolic cone structure, we see that Conjecture \ref{conj} is true for $(M_c,L_{\text{FSL}}).$ This was first proved by Costantino in \cite{C1} at the root of unity $q=e^{\frac{\pi \sqrt{-1}}{r}}.$

\subsection{Twisted Reidemeister torsion}\label{TRT}

Let $\mathrm C_*$ be a finite chain complex 
$$0\to \mathrm C_d\xrightarrow{\partial}\mathrm C_{d-1}\xrightarrow{\partial}\cdots\xrightarrow{\partial}\mathrm C_1\xrightarrow{\partial} \mathrm C_0\to 0$$
of $\mathbb C$-vector spaces, and for each $\mathrm C_k$ choose a basis $\mathbf c_k.$ Let $\mathrm H_*$ be the homology of $\mathrm C_*,$ and for each $\mathrm H_k$ choose a basis $\mathbf h_k$ and a lift $\widetilde{\mathbf h}_k\subset \mathrm C_k$ of $\mathbf h_k.$ We also choose a basis $\mathbf b_k$ for each image $\partial (\mathrm C_{k+1})$ and a lift $\widetilde{\mathbf b}_k\subset \mathrm C_{k+1}$ of $\mathbf b_k.$ Then $\mathbf b_k\sqcup \widetilde{\mathbf b}_{k-1}\sqcup \widetilde{\mathbf h}_k$ form a basis of $\mathrm C_k.$ Let $[\mathbf b_k\sqcup \widetilde{\mathbf b}_{k-1}\sqcup \widetilde{\mathbf h}_k;\mathbf c_k]$ be the determinant of the transition matrix from the standard basis $\mathbf c_k$ to the new basis $\mathbf b_k\sqcup \widetilde{\mathbf b}_{k-1}\sqcup \widetilde{\mathbf h}_k.$
 Then the Reidemeister torsion of the chain complex $\mathrm C_*$ with the chosen bases $\mathbf c_*$ and $\mathbf h_*$ is defined by 
\begin{equation*}
\mathrm{Tor}(\mathrm C_*, \{\mathbf c_k\}, \{\mathbf h_k\})=\pm\prod_{k=0}^d[\mathbf b_k\sqcup \widetilde{\mathbf b}_{k-1}\sqcup \widetilde{\mathbf h}_k;\mathrm c_k]^{(-1)^{k+1}}.
\end{equation*}
It is easy to check that $\mathrm{Tor}(\mathrm C_*, \{\mathbf c_k\}, \{\mathbf h_k\})$ depends only on the choice of $\{\mathbf c_k\}$ and $\{\mathbf h_k\},$ and does not depend on the choices of $\{\mathbf b_k\}$ and the lifts $ \{\widetilde{\mathbf b}_k\}$ and $\{\widetilde{\mathbf h}_k\}.$

We recall the twisted Reidemeister torsion of a CW-complex following the conventions in \cite{P2}. Let $K$ be a finite CW-complex and let $\rho:\pi_1(M)\to\mathrm{SL}(N;\mathbb C)$ be a representation of its fundamental group. Consider the twisted chain complex 
$$\mathrm C_*(K;\rho)= \mathbb C^N\otimes_\rho \mathrm C_*(\widetilde K;\mathbb Z)$$
where $\mathrm C_*(\widetilde K;\mathbb Z)$ is the simplicial complex of the universal covering of $K$ and $\otimes_\rho$ means the tensor product over $\mathbb Z$ modulo the relation
$$\mathbf v\otimes( \gamma\cdot\mathbf c)=\Big(\rho(\gamma)^T\cdot\mathbf v\Big)\otimes \mathbf c,$$
where $T$ is the transpose, $\mathbf v\in\mathbb C^N,$ $\gamma\in\pi_1(K)$ and $\mathbf c\in\mathrm C_*(\widetilde K;\mathbb Z).$ The boundary operator on $\mathrm C_*(K;\rho)$ is defined by
$$\partial(\mathbf v\otimes \mathbf c)=\mathbf v\otimes \partial(\mathbf c)$$
for $\mathbf v\in\mathbb C^N$ and $\mathbf c\in\mathrm C_*(\widetilde K;\mathbb Z).$ Let $\{\mathbf e_1,\dots,\mathbf e_N\}$ be the standard basis of $\mathbb C^N,$ and let $\{c_1^k,\dots,c_{d^k}^k\}$ denote the set of $k$-cells of $K.$ Then we call
$$\mathbf c_k=\big\{ \mathbf e_i\otimes c_j^k\ \big|\ i\in\{1,\dots,N\}, j\in\{1,\dots,d^k\}\big\}$$
the standard basis of $\mathrm C_k(K;\rho).$ Let $\mathrm H_*(K;\rho)$ be the homology of the chain complex $\mathrm C_*(K;\rho)$ and let $\mathbf h_k$ be a basis of $\mathrm H_k(K;\rho).$ Then the Reidemeister torsion of $K$ twisted by $\rho$ with basis $\{\mathbf h_k\}$ is 
$$\mathrm{Tor}(K, \{\mathbf h_k\}; \rho)=\mathrm{Tor}(\mathrm C_*(K;\rho),\{\mathbf c_k\}, \{\mathbf h_k\}).$$

By \cite{P}, $\mathrm{Tor}(K, \{\mathbf h_k\}; \rho)$ depends only on the conjugacy class of $\rho.$ By for e.g. \cite{T2}, the Reidemeister torsion is invariant under elementary expansions and elementary collapses of CW-complexes, and by \cite{M}  it is invariant under subdivisions, hence defines an invariant of PL-manifolds and of topological manifolds of dimension less than or equal to $3.$

We list some results by Porti\,\cite{P} for the Reidemeister torsions of hyperbolic $3$-manifolds twisted by the adjoint representation $\mathrm {Ad}_\rho=\mathrm {Ad}\circ\rho$ of the holonomy  $\rho$ of the hyperbolic structure. Here $\mathrm {Ad}$ is the adjoint acton of $\mathrm {PSL}(2;\mathbb C)$ on its Lie algebra $\mathbf{sl}(2;\mathbb C)\cong \mathbb C^3.$

For a closed oriented hyperbolic $3$-manifold  $M$ with the holonomy representation $\rho,$ by the Weil local rigidity theorem and the Mostow rigidity theorem,
$\mathrm H_k(M;\mathrm{Ad}_\rho)=0$ for all $k.$ Then the twisted Reidemeister torsion 
$$\mathrm{Tor}(M;\mathrm{Ad}_\rho)\in\mathbb C^*/\{\pm 1\}$$
 is defined without making any additional choice.

For a compact, orientable  $3$-manifold  $M$ with boundary consisting of $n$ disjoint tori $T_1 \dots,  T_n$ whose interior admits a complete hyperbolic structure with  finite volume, let $\mathrm X(M)$ be the $\mathrm{SL}(2; \CC)$-character variety of $M,$ let $\mathrm X_0(M)\subset\mathrm X(M)$ be the distinguished component containing the character of a chosen lifting of the holomony representation of the complete hyperbolic structure of $M,$ and let $\mathrm X^{\text{irr}}(M)\subset\mathrm X(M)$ be consisting of the irreducible characters. 
 
\begin{theorem}(\cite[Section 3.3.3]{P})\label{HM} For a generic character $[\rho]\in\mathrm X_0(M)\cap\mathrm X^{\text{irr}}(M)$  we have:
\begin{enumerate}[(1)]
\item For $k\neq 1,2,$ $\mathrm H_k(M;\mathrm{Ad}\rho)=0.$
\item  For $i\in\{1,\dots,n\},$ let $\mathbf I_i\in \mathbb C^3$ be up to scalar the unique invariant vector of $\mathrm Ad_\rho(\pi_1(T_i)).$ Then
$$\mathrm H_1(M;\mathrm{Ad}\rho)\cong\bigoplus_{i=1}^n\mathrm H_1(T_i;\mathrm{Ad}\rho)\cong \mathbb C^n,$$ 
and for each $\alpha=([\alpha_1],\dots,[\alpha_n])\in \mathrm H_1(\partial M;\mathbb Z)\cong 
\bigoplus_{i=1}^n\mathrm H_1(T_i;\mathbb Z)$ has a basis 
$$\mathbf h^1_{(M,\alpha)}=\{\mathbf I_1\otimes [\alpha_1],\dots, \mathbf I_n\otimes [\alpha_n]\}.$$
\item Let $([T_1],\dots,[T_n])\in \bigoplus_{i=1}^n\mathrm H_2(T_i;\mathbb Z)$ be the fundamental classes of $T_1,\dots, T_n.$ Then 
 $$\mathrm H_2(M;\mathrm{Ad}\rho)\cong\bigoplus_{i=1}^n\mathrm H_2(T_i;\mathrm{Ad}\rho)\cong \mathbb C^n,$$ 
and has a basis 
$$\mathbf h^2_M=\{\mathbf I_1\otimes [T_1],\dots, \mathbf I_n\otimes [T_n]\}.$$
\end{enumerate}
\end{theorem}

\begin{remark}[\cite{P}] Important examples of the generic characters in Theorem \ref{HM} include the characters of the lifting in $\mathrm{SL}(2;\mathbb C)$ of the holonomy of the complete hyperbolic structure on the interior of $M,$
 the restriction of the holonomy of the closed $3$-manifold $M_\mu$ obtained from $M$ by doing the hyperbolic Dehn surgery along the system of simple closed curves $\mu$ on $\partial M,$
 and by \cite{HK} the holonomy of a hyperbolic structure on the interior of $M$ whose completion is a conical manifold with cone angles less than $2\pi.$
\end{remark}

For $\alpha\in\mathrm H_1(M;\mathbb Z),$ define $\mathbb T_{(M,\alpha)}$ on $\mathrm X_0(M)$ by
$$\mathbb T_{(M,\alpha)}([\rho])=\mathrm{Tor}(M, \{\mathbf h^1_{(M,\alpha)},\mathbf h^2_M\};\mathrm{Ad}_\rho)$$
for the generic $[\rho]\in \mathrm X_0(M)\cap\mathrm X^{\text{irr}}(M)$ in Theorem \ref{HM}, and equals $0$ otherwise.

\begin{theorem}(\cite[Theorem 4.1]{P})\label{funT}
Let $M$ be a compact, orientable  $3$-manifold with boundary consisting of $n$ disjoint tori $T_1 \dots,  T_n$ whose interior admits a complete hyperbolic structure with  finite volume. Let  $\mathbb C(\mathrm X_0(M))$ be the ring of rational functions over $\mathrm X_0(M).$ Then there is up to sign a unique function
\begin{equation*}
\begin{split}
\mathrm H_1(\partial M;\mathbb Z)&\to \mathbb C(\mathrm X_0(M))\\
\alpha\quad\quad &\mapsto \quad\mathbb T_{(M,\alpha)}
\end{split}
\end{equation*}
which is a $\mathbb Z$-multilinear homomorphism with respect to the direct sum $\mathrm H_1(\partial M;\mathbb Z)\cong 
\bigoplus_{i=1}^n\mathrm H_1(T_i;\mathbb Z)$ satisfying the following properties:
\begin{enumerate}[(i)]
\item For all $\alpha \in \mathrm H_1(\partial M;\mathbb Z),$ the domain of definition of $\mathbb T_{(M,\alpha)}$ contains an open subset $\mathrm X_0(M)\cap\mathrm X^{\text{irr}}(M).$

\item \emph{(Change of curves formula).} Let $\mu=\{\mu_1,\dots,\mu_n\}$ and $\gamma=\{\gamma_1,\dots,\gamma_n\}$ be two systems of simple closed curves on $\partial M.$ If $\mathrm H(\mu_1),\dots, \mathrm H(\mu_n)$ and $\mathrm H(\gamma_1),\dots,\mathrm H(\gamma_n)$ are respectively the logarithmic holonomies of the curves in $\mu$ and $\gamma,$ then we have the equality of rational functions
\begin{equation*}\label{coc}
\mathbb T_{(M,\mu)}
=\pm\det\bigg( \frac{\partial \mathrm H(\mu_i)}{\partial \mathrm H(\gamma_j)}\bigg)_{ij}\mathbb T_{(M,\gamma)}.
\end{equation*}
\item \emph{(Surgery formula).} Let $[\rho_\mu]\in \mathrm X_0(M)$ be the character induced by the holonomy of the closed $3$-manifold $M_\mu$ obtained from $M$ by doing the hyperbolic Dehn surgery along the system of simple closed curves $\mu$ on $\partial M.$ If $\mathrm H(\gamma_1),\dots,\mathrm H(\gamma_n)$ are the logarithmic holonomies of the core curves $\gamma_1,\dots,\gamma_n$ of the solid tori added. Then
\begin{equation*}\label{sf}
\mathrm{Tor}(M_\mu;\mathrm{Ad}_{\rho_\mu})=\pm\mathbb T_{(M,\mu)}([\rho_\mu])\prod_{i=1}^n\frac{1}{4\sinh^2\frac{\mathrm H(\gamma_i)}{2}}.
\end{equation*}
\end{enumerate}
\end{theorem}

Next, we list some results for the computation of twisted Reidemeister torsions from \cite{WY3}. We first recall that if $\mathrm M_{4\times 4}(\mathbb C)$ is the space of $4\times 4$ matrices with complex entries, then the \emph{Gram matrix function} 
$$\mathbb G:\mathbb C^6\to \mathrm M_{4\times 4}(\mathbb C)$$
 is defined by 
\begin{equation}\label{gram}
\begin{split}
\mathbb{G}(\mathbf z)=\left[
\begin{array}{cccc}
1 & -\cosh z_{1} & -\cosh z_{2} &-\cosh z_{6}\\
-\cosh z_{1}& 1 &-\cosh z_{3} & -\cosh z_{5}\\
-\cosh z_{2} & -\cosh z_{3} & 1 & -\cosh z_{4} \\
-\cosh z_{6} & -\cosh z_{5} & -\cosh z_{4}  & 1 \\
 \end{array}\right]
 \end{split}
\end{equation}
for $\mathbf z=(z_{1}, z_{2}, z_{3}, z_{4}, z_{5}, z_{6})\in\mathbb C^6.$ The value of $\mathbb G$ at different $\mathbf u$ recover the Gram matrices of deeply truncated tetrahedra of all the types. See \cite[Section 2.1]{BY} for more details.

\begin{theorem}(\cite[Theorem 1.1]{WY3})\label{Rtorthm}  Let $M=\#^{c+1}(S^2\times S^1)\setminus L_{\text{FSL}}$ be the complement of a fundamental shadow link $L_{\text{FSL}}$ with $n$ components $L_1,\dots,L_n,$ which is the orientable double of the union of truncated tetrahedra $\Delta_1,\dots, \Delta_c$ along pairs of the triangles of truncation, and let $\mathrm X_0(M)$ be the distinguished component of the $\mathrm{SL}(2;\mathbb C)$ character variety of $M$ containing a lifting of the holonomy representation of the complete hyperbolic structure.  
\begin{enumerate}[(1)] 
\item  Let $\mathbf u=(u_1,\dots, u_n)$ be the system of the meridians of a tubular neighborhood of the components of $L_{\text{FSL}}.$ For a generic irreducible character $[\rho]$ in $\mathrm X_0(M),$ let $\mathrm H(u_1),
\dots,\mathrm H(u_n)$ be the logarithmic holonomies of $\mathbf u.$ For each $s\in\{1,\dots,c\},$ let $L_{s_1},\dots,L_{s_6}$ be the components of $L_{\text{FSL}}$ intersecting $\Delta_s,$ and let $\mathbb G_s$ be the value of the Gram matrix function at $\Big(\frac{\mathrm H(u_{s_1})}{2},\dots,\frac{\mathrm H(u_{s_6})}{2}\Big).$ Then
 $$\mathbb T_{(M,\mathbf u)}([\rho])=\pm2^{3c}\prod_{s=1}^c \sqrt{\det\mathbb G_s}.$$

\item In addition to the assumptions and notations of (1), let $\mathbf \Upsilon=(\Upsilon_1,\dots,\Upsilon_n)$ be a system of simple closed curves on $\partial M,$ and let $(\mathrm H(\Upsilon_1),\dots, \mathrm H(\Upsilon_n))$ be their logarithmic holonomies which are functions of $(\mathrm H(u_1),
\dots,\mathrm H(u_n)).$ Then
 $$ \mathbb T_{(M,\mathbf \Upsilon)}[(\rho)]=\pm2^{3c}\det\bigg(\frac{\partial \mathrm H(\Upsilon_i)}{\partial \mathrm H(u_j)}\bigg|_{[\rho]}\bigg)_{ij}\prod_{s=1}^c \sqrt{\det\mathbb G_s}.$$

\item Suppose $M_{\mathbf \Upsilon}$ is the closed $3$-manifold obtained from $M$ by doing the hyperbolic Dehn surgery along a system of simple closed curves $\mathbf \Upsilon=(\Upsilon_1,\dots,\Upsilon_n)$ on $\partial M$ and $\rho_{\mathbf \Upsilon}$ is the restriction of the holonomy representation of $M_{\mathbf \Upsilon}$ to $M.$ Let $(\mathrm H(\Upsilon_1),\dots, \mathrm H(\Upsilon_n))$ be the logarithmic holonomies of $\mathbf \Upsilon$ which are functions of the logarithmic holonomies of the meridians $\mathbf u.$ Let $(\gamma_1,\dots,\gamma_n)$ be a system of simple closed curves on $\partial M$ that are isotopic to the core curves of the solid tori filled in and let $\mathrm H(\gamma_1),\dots,\mathrm H(\gamma_n)$ be their logarithmic holonomies in $[\rho_\mu].$ Let  $\mathrm H(u_1),
\dots,\mathrm H(u_n)$ be the logarithmic holonomies of the meridians $\mathbf u$ in $[\rho_\mu]$ and for each $s\in\{1,\dots,c\},$ let $L_{s_1},\dots,L_{s_6}$ be the components of $L_{\text{FSL}}$ intersection $\Delta_s$ and let $\mathbb G_s$ be the value of the Gram matrix function at $\Big(\frac{\mathrm H(u_{s_1})}{2},\dots,\frac{\mathrm H(u_{s_6})}{2}\Big).$
Then 
 $$ \mathrm{Tor}(M_{\mathbf \Upsilon};\mathrm{Ad}_{\rho_{\mathbf \Upsilon}})=\pm2^{3c-2n}\det\bigg(\frac{\partial \mathrm H(\Upsilon_i)}{\partial \mathrm H(u_j)}\bigg|_{[\rho_\mu]}\bigg)_{ij}\prod_{s=1}^c \sqrt{\det\mathbb G_s}\prod_{i=1}^n\frac{1}{\sinh^2\frac{\mathrm H(\gamma_i)}{2}}.$$
\end{enumerate} 
\end{theorem}

\subsection{Dilogarithm and quantum dilogarithm functions}

Let $\log:\mathbb C\setminus (-\infty, 0]\to\mathbb C$ be the standard logarithm function defined by
$$\log z=\log|z|+\sqrt{-1}\arg z$$
with $-\pi<\arg z<\pi.$
 
The dilogarithm function $\mathrm{Li}_2: \mathbb C\setminus (1,\infty)\to\mathbb C$ is defined by
$$\mathrm{Li}_2(z)=-\int_0^z\frac{\log (1-u)}{u}du$$
where the integral is along any path in $\mathbb C\setminus (1,\infty)$ connecting $0$ and $z,$ which is holomorphic in $\mathbb C\setminus [1,\infty)$ and continuous in $\mathbb C\setminus (1,\infty).$

The dilogarithm function satisfies the follow properties (see eg. Zagier\,\cite{Z}).
\begin{enumerate}[(1)]
\item \begin{equation}\label{Li2}
\mathrm{Li}_2\Big(\frac{1}{z}\Big)=-\mathrm{Li}_2(z)-\frac{\pi^2}{6}-\frac{1}{2}\big(\log(-z)\big)^2.
\end{equation} 
\item In the unit disk $\big\{z\in\mathbb C\,\big|\,|z|<1\big\},$ 
\begin{equation}\label{Li1}
\mathrm{Li}_2(z)=\sum_{n=1}^\infty\frac{z^n}{n^2}.
\end{equation}
\item On the unit circle $\big\{ z=e^{2\sqrt{-1}\theta}\,\big|\,0 \leqslant \theta\leqslant\pi\big\},$ 
\begin{equation}\label{dilogLob}
\mathrm{Li}_2\Big(e^{2\sqrt{-1}\theta}\Big)=\frac{\pi^2}{6}+\theta(\theta-\pi)+2\sqrt{-1}\Lambda(\theta).
\end{equation}
\end{enumerate}
Here  $\Lambda:\mathbb R\to\mathbb R$  is the Lobachevsky function defined by
\begin{align}\label{QDtoL}
\Lambda(\theta)=-\int_0^\theta\log|2\sin t|dt,
\end{align}
which is an odd function of period $\pi.$ See eg. Thurston's notes\,\cite[Chapter 7]{T}.


The following variant of Faddeev's quantum dilogarithm functions\,\cite{F, FKV} will play a key role in the proof of the main result. 
Let $r\geqslant 3$ be an odd integer. Then the following contour integral
\begin{equation}
\varphi_r(z)=\frac{4\pi\sqrt{-1}}{r}\int_{\Omega}\frac{e^{(2z-\pi)x}}{4x \sinh (\pi x)\sinh (\frac{2\pi x}{r})}\ dx
\end{equation}
defines a holomorphic function on the domain $$\Big\{z\in \mathbb C \ \Big|\ -\frac{\pi}{r}<\mathrm{Re}z <\pi+\frac{\pi}{r}\Big\},$$  
  where the contour is
$$\Omega=\big(-\infty, -\epsilon\big]\cup \big\{z\in \mathbb C\ \big||z|=\epsilon, \mathrm{Im}z>0\big\}\cup \big[\epsilon,\infty\big),$$
for some $\epsilon\in(0,1).$
Note that the integrand has poles at $n\sqrt{-1},$ $n\in\mathbb Z,$ and the choice of  $\Omega$ is to avoid the pole at $0.$
\\

The function $\varphi_r(z)$ satisfies the following fundamental properties, whose proof can be found in \cite[Section 2.3]{WY}. 
\begin{lemma}
\begin{enumerate}[(1)]
\item For $z\in\mathbb C$ with  $0<\mathrm{Re}z<\pi,$
\begin{equation}\label{fund}
1-e^{2 \sqrt{-1}z}=e^{\frac{r}{4\pi\sqrt{-1}}\Big(\varphi_r\big(z-\frac{\pi}{r}\big)-\varphi_r\big(z+\frac{\pi}{r}\big)\Big)}.
 \end{equation}
 
 \item For $z\in\mathbb C$ with  $-\frac{\pi}{r}<\mathrm{Re}z<\frac{\pi}{r},$
 \begin{equation}\label{f2}
1+e^{r\sqrt{-1}z}=e^{\frac{r}{4\pi\sqrt{-1}}\Big(\varphi_r(z)-\varphi_r\big(z+\pi\big)\Big)}.
\end{equation}
\end{enumerate}
\end{lemma}

Using (\ref{fund}) and (\ref{f2}), for $z\in\mathbb C$ with $\pi+\frac{2(n-1)\pi}{r}< \mathrm{Re}z< \pi+\frac{2n\pi}{r},$ we can define $\varphi_r(z)$  inductively by the relation
\begin{equation}\label{extension}
\prod_{k=1}^n\Big(1-e^{2 \sqrt{-1} \big(z-\frac{(2k-1)\pi}{r}\big)}\Big)=e^{\frac{r}{4\pi\sqrt{-1}}\Big(\varphi_r\big(z-\frac{2n\pi}{r}\big)-\varphi_r(z)\Big)},
\end{equation}
extending $\varphi_r(z)$ to a meromorphic function on $\mathbb C.$  The poles of $\varphi_r(z)$ have the form $(a+1)\pi+\frac{b\pi}{r}$ or $-a\pi-\frac{b\pi}{r}$ for all nonnegative integer $a$ and positive odd integer $b.$

Let $q=e^{\frac{2\pi\sqrt{-1}}{r}},$
and let $$(q)_n=\prod_{k=1}^n(1-q^{2k}).$$

\begin{lemma}\label{fact}
\begin{enumerate}[(1)]
\item For $0\leqslant n \leqslant r-2,$
\begin{equation}
(q)_n=e^{\frac{r}{4\pi\sqrt{-1}}\Big(\varphi_r\big(\frac{\pi}{r}\big)-\varphi_r\big(\frac{2\pi n}{r}+\frac{\pi}{r}\big)\Big)}.
\end{equation}
\item For $\frac{r-1}{2}\leqslant n \leqslant r-2,$
\begin{equation} \label{move}
(q)_n=2e^{\frac{r}{4\pi\sqrt{-1}}\Big(\varphi_r\big(\frac{\pi}{r}\big)-\varphi_r\big(\frac{2\pi n}{r}+\frac{\pi}{r}-\pi\big)\Big)}.
\end{equation}
\end{enumerate}
\end{lemma}

We consider (\ref{move}) because there are poles in $(\pi,2\pi),$ and to avoid the poles we move the variables to $(0,\pi)$ by subtracting $\pi.$

For $n\in \ZZ_{\geq 0}$, let $\{0\}=1$, $\{n\} = q^{n}-q^{-n}$, $\{0\}! = 1$ and 
$$\{n\}! = \prod_{k=1}^n \{k\}.$$ 
Since 
$$\{n\}!=(-1)^nq^{-\frac{n(n+1)}{2}}(q)_n,$$
as a consequence of Lemma \ref{fact}, we have

\begin{lemma}\label{factorial}
\begin{enumerate}[(1)]
\item For $0\leqslant n \leqslant r-2,$
\begin{equation}
\{n\}!=e^{\frac{r}{4\pi\sqrt{-1}}\Big(-2\pi\big(\frac{2\pi n}{r}\big)+\big(\frac{2\pi}{r}\big)^2(n^2+n)+\varphi_r\big(\frac{\pi}{r}\big)-\varphi_r\big(\frac{2\pi n}{r}+\frac{\pi}{r}\big)\Big)}.
\end{equation}
\item For $\frac{r-1}{2}\leqslant n \leqslant r-2,$
\begin{equation} 
\{n\}!=2e^{\frac{r}{4\pi\sqrt{-1}}\Big(-2\pi\big(\frac{2\pi n}{r}\big)+\big(\frac{2\pi }{r}\big)^2(n^2+n)+\varphi_r\big(\frac{\pi}{r}\big)-\varphi_r\big(\frac{2\pi n}{r}+\frac{\pi}{r}-\pi\big)\Big)}.
\end{equation}
\end{enumerate}
\end{lemma}

The function $\varphi_r(z)$ and the dilogarithm function are closely related as follows.

\begin{lemma}\label{converge}  \begin{enumerate}[(1)]
\item For every $z$ with $0<\mathrm{Re}z<\pi,$ 
\begin{equation}
\varphi_r(z)=\mathrm{Li}_2(e^{2\sqrt{-1}z})+\frac{2\pi^2e^{2\sqrt{-1}z}}{3(1-e^{2\sqrt{-1}z})}\frac{1}{r^2}+O\Big(\frac{1}{r^4}\Big).
\end{equation}
\item For every $z$ with $0<\mathrm{Re}z<\pi,$ 
\begin{equation}
\varphi_r'(z)=-2\sqrt{-1}\log(1-e^{2\sqrt{-1}z})+O\Big(\frac{1}{r^2}\Big).
\end{equation}
\item \cite[Formula (8)(9)]{O2}
$$\varphi_r\Big(\frac{\pi}{r}\Big)=\mathrm{Li}_2(1)+\frac{2\pi\sqrt{-1}}{r}\log\Big(\frac{r}{2}\Big)-\frac{\pi^2}{r}+O\Big(\frac{1}{r^2}\Big).$$
\end{enumerate}\end{lemma}

\subsection{Continued fractions}\label{CF}

We recall some notations related to the continued fraction of  rational numbers, which will be used in the computation of the Reshetikhin-Turaev invariants (Proposition \ref{computation}).
For a pair of relatively prime integers $(p,q),$ let $$\frac{p}{q}=a_k-\frac{1}{a_{k-1}-\frac{1}{\cdots-\frac{1}{a_1}}}$$ be a  continued fraction.
For each $l\in\{1,\dots,k\},$ consider the matrix
\begin{align} \begin{bmatrix}\label{defABCD}
A_l & B_l \\
C_l & D_l \end{bmatrix}= T^{a_l}S\cdots T^{a_1}S,
\end{align}
where $$S= \begin{bmatrix}
0 & -1 \\
1 & 0 \end{bmatrix} \quad\text{and}\quad T= \begin{bmatrix}
1 & 1 \\
0 & 1 \end{bmatrix},$$  
and as a convention let
\begin{equation}
 \begin{bmatrix}
A_0  \\
C_0  \end{bmatrix}= \begin{bmatrix}
1    \\
0  \end{bmatrix}.
\end{equation}

\begin{lemma}\cite[Proposition 2.5]{J}\label{cf} 
\begin{enumerate}[(1)]

\item For $l\in\{1,\dots,k\},$ $A_l=a_lA_{l-1}-C_{l-1}$ and $C_l=A_{l-1}.$ 
\item For $l\in\{1,\dots,k\},$ $B_l=a_lB_{l-1}-D_{l-1}$ and $D_l=B_{l-1}.$


\item We have
$$\frac{A_k}{C_k}=\frac{p}{q}.$$

\item For $l\in\{1,\dots,k\},$
$$\frac{B_l}{A_l} = - \Big( \frac{1}{A_1}+\frac{1}{A_2A_1}+\dots + \frac{1}{A_{l} A_{l-1}}\Big).$$
\end{enumerate}
\end{lemma}

We observe that $A_k$ and $C_k$ are relatively prime because $A_kD_k-B_kC_k=\det(T^{a_k}S\cdots T^{a_1}S)=1.$ By Lemma \ref{cf} (3), $ \begin{bmatrix}
A_k\\
C_k \end{bmatrix}= \pm\begin{bmatrix}
p\\
q\end{bmatrix}. $
Since a $(p,q)$ Dehn-surgery and a $(-p,-q)$ Dehn-surgery provide the same $3$-manifold $M,$ we may without loss of generality assume that 
\begin{equation}\label{AkCk}
 \begin{bmatrix}
A_k \\
C_k\end{bmatrix}= \begin{bmatrix}
p\\
q\end{bmatrix}.
\end{equation}
 As a consequence, by Lemma \ref{cf} (1), we have
\begin{equation}\label{Ak-1=q}
\begin{bmatrix}
A_{k-1} \\
C_{k-1} \end{bmatrix}= \begin{bmatrix}
q\\
-p+a_kq\end{bmatrix}. 
\end{equation}
We also let 
\begin{equation}\label{p'}
 \begin{bmatrix}
p' \\
q'\end{bmatrix}= \begin{bmatrix}
D_k\\
-B_k\end{bmatrix}
\end{equation}
so that $pp'+qq'=1.$ In particular, by Lemma \ref{cf} (1), (2) and (4) we have
\begin{align}\label{sump'q'}
\frac{1}{A_1}+\frac{1}{A_2A_1}+\dots + \frac{1}{A_{k-1} A_{k-2}}
= -\frac{B_{k-1}}{A_{k-1}} = - \frac{D_k}{C_k} = - \frac{p'}{q}.
\end{align}

For $l\in\{1,\dots, k\},$ we also consider the quantity
\begin{equation}\label{K}
K_l=\frac{(-1)^{l+1}\sum_{j=1}^la_jC_j}{C_l}.
\end{equation}
The following Lemma \ref{arith'} and \ref{arith} from \cite{WY} are crucial in the computation of the relative Reshetikhin-Turaev invariants and the study of their asymptotics. 
\begin{lemma}\label{arith'}\cite[Lemma 3.2]{WY}
 $C_{k-1}K_{k-1}+C_{k-1}q$ is an even integer.
 \end{lemma}
  
\begin{lemma}\label{arith}\cite[Lemma 3.3]{WY}
\begin{enumerate}[(1)] 

\item Let
$$I:\{0,\dots, |q|-1\}\to\{0,\dots,2|q|-1\}$$ be  the map defined by 
$$I(s)=-C_{k-1}(2s+1+K_{k-1})\quad(\text{mod }2|q|).$$
Then $I$ is injective with image the set of integers in $\{0,\dots,2|q|-1\}$ with parity that of $1-q.$ 

In particular, there exist a unique $s^+\in\{0,\dots, |q|-1\}$ and a unique integer $m^+$ such that 
$$I(s^+)=1-q+2m^+q,$$
 and  a unique $s^-\in\{0,\dots, |q|-1\}$  and a unique  integer $m^-$ such that
$$I(s^-)=-1-q+2m^-q.$$
Moreover, 
\begin{equation}\label{+--}
s^+-s^-\equiv p' \quad(\text{mod }q).
\end{equation}

\item Let
$$J:\{0,\dots,|q|-1\}\to\mathbb Q$$
be  the map defined by
$$J(s)=\frac{2s+1}{q}+(-1)^k\sum_{i=1}^{k-1}\frac{(-1)^{i+1}K_i}{C_{i+1}}.$$
Then for the $s^+$ and $s^-$ in (1), 
$$J(s^+)\equiv\frac{p'}{q}\quad(\text{mod }\mathbb Z)$$
and
$$J(s^-)\equiv -\frac{p'}{q}\quad(\text{mod }\mathbb Z).$$

Moverover, 
$$J(s^+)\equiv  -J(s^-)\quad(\text{mod }2\mathbb Z).$$

\item Let 
$$K:\{0,\dots,|q|-1\}\to\mathbb Q$$
be  the map defined by 
$$K(s)=\frac{C_{k-1}(2s+1+K_{k-1})^2}{q}+\sum_{i=1}^{k-2}\frac{C_iK_i^2}{C_{i+1}}.$$
Then for the $s^+$ and $s^-$ in (1), 
$$K(s^+)\equiv -\frac{p'}{q}\quad(\text{mod }\mathbb Z)$$
and
$$K(s^-)\equiv -\frac{p'}{q}\quad(\text{mod }\mathbb Z).$$
\end{enumerate}
\end{lemma}
 
\subsection{Rational Dehn surgery}
Given a link $K=K_1\cup \dots \cup K_n \subset\SS^3$ with $n$ component, let $I\subset \{1,2,\dots,n\}$, $J = \{1,2,\dots,n\}\setminus I$ and 
\begin{align*}
     \frac{p_i}{q_i} = a_{i,\zeta_i} - \frac{1}{a_{i,\zeta_i-1} - \frac{1}{\dots -\frac{1}{a_{i,1}}}},
\end{align*}
where $a_{i,1},\dots, a_{i,\zeta_i}$ are integers for all $i$. For each $i\in I$, we choose a pair of meridian and longitude $\{(u_i,v_i)\}_{i\in I}$ of the fundamental group of the boundary of the tubular neighborhood of $K_i$. 
Recall from \cite[ p.273]{R} that doing $(p_i,q_i)$ surgery on the $K_i$ is the same as doing $(a_{i,\zeta_i},a_{i,\zeta_{i}-1},\dots,a_{i,1})$ surgery on the framed link $\tilde{K}_i$ obtained by adding a chain of framed simple loops around $K_i$ as shown in Figure \ref{surgexample}. Let $L_i$ be a simple loop with framing $a_{i,0}$ as shown in Figure \ref{surgexample2}. 
\begin{figure}[h]
\centering
\includegraphics[scale=0.2]{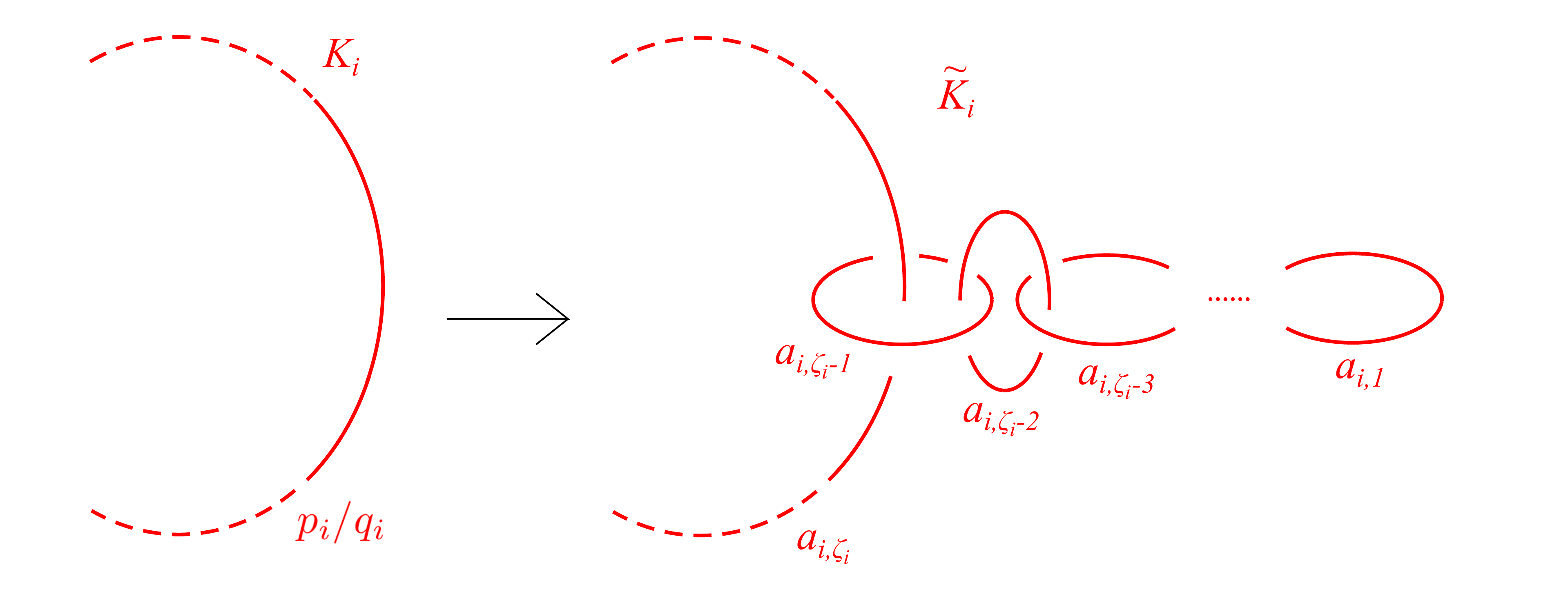}
\caption{doing $(p_i,q_i)$ surgery on $K_i$ is equivalent to doing $(a_{i,\zeta_i},a_{i,\zeta_{i-1}},\dots,a_{i,1})$ on $\tilde K_i$}\label{surgexample}
\end{figure}
\begin{figure}[h]
\centering
\includegraphics[scale=0.2]{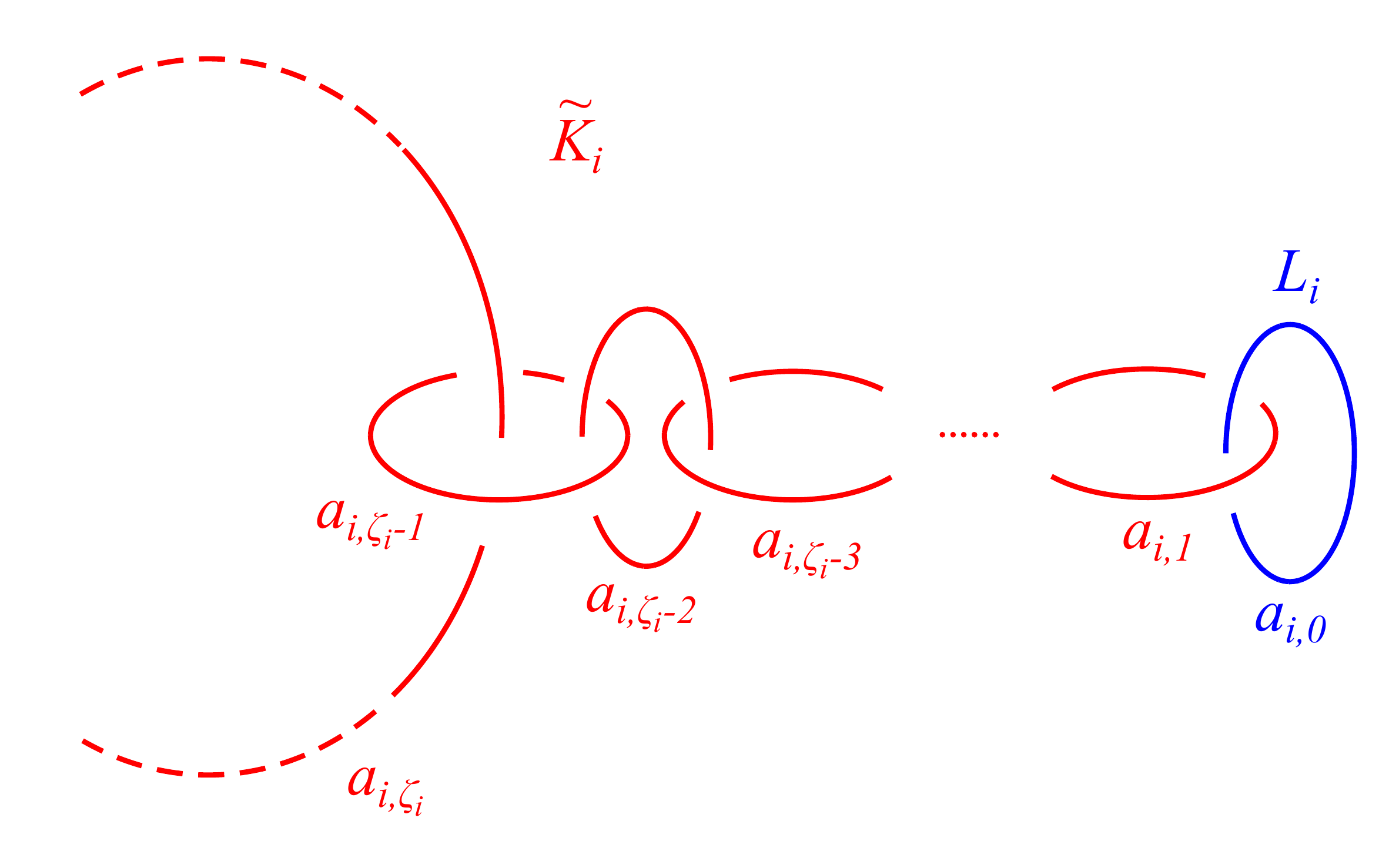}
\caption{Changing each $L_{\text{FSL},i}$ to $\tilde L_{\text{FSL},i}$}\label{surgexample2}
\end{figure}

Consider the continued fraction
$$
a_{i,\zeta_i} - \frac{1}{a_{i,\zeta_i-1} - \frac{1}{\dots -\frac{1}{a_{i,1}-\frac{1}{a_{i,0}}}}}.$$
Note that by (\ref{defABCD}), (\ref{p'}) and Lemma \ref{cf} (3), since
\begin{align*} 
\begin{bmatrix}
p_i & - q_i' \\
q_i &  p_i' 
\end{bmatrix}
= T^{a_{i,\zeta_i}}S\cdots T^{a_{i,1}}S,
\end{align*}
we have
\begin{align}
T^{a_{i,\zeta_i}}S\cdots T^{a_{i,1}}S T^{a_{i,0}}S
= \begin{bmatrix}
p_i & - q_i' \\
q_i &  p_i' 
\end{bmatrix}
\begin{bmatrix}
a_{i,0} & -1 \\
1 & 0 
\end{bmatrix}
=
\begin{bmatrix}
a_{i,0} p_i - q_i' & -p_i \\
a_{i,0} q_i + p_i' & q_i 
\end{bmatrix}.
\end{align}
This implies that the parallel copy $\gamma_i$ of $L_i$ given by the  framing $a_{i,0}$ is isotopic to curve $-q_i' u_i + p_i' v_i + a_{i,0}(p_i u_i + q_i v_i)$ on the boundary of the tubular neighborhood of $K_i$. Consider the hyperbolic cone structure on the closed oriented 3-manifold obtained by doing $(p_i,q_i)$ surgery on $\{K_i\}_{i\in I}$ and $(1,0)$ surgery on $\{K_j\}_{j \in J}$ with singular locus $\{L_i\}_{i\in I} \cup \{K_j\}_{j\in J}$ and cone angles $(\theta_1,\dots, \theta_n)$. Since 
$$ p_i\mathrm{H}(u_i) + q_i \mathrm{H}(v_i) = \theta_i \sqrt{-1}$$
for all $i\in I$, we have
\begin{align}\label{corecurveHol}
\mathrm{H}(\gamma_i) 
&= -q_i'\mathrm{H}(u_i) + p_i'\mathrm{H}(v_i) + a_{i,0}(p_i\mathrm{H}(u_i) + q_i\mathrm{H}(v_i) )\notag\\
&= -q_i'\mathrm{H}(u_i) + p_i'\mathrm{H}(v_i) + a_{i,0} \theta_i\sqrt{-1}.
\end{align}



\section{Computation of the relative Reshetikhin-Turaev invariants}\label{compRT}
Let $L_{\text{FSL}} = L_{\text{FSL},1} \cup \dots \cup L_{\text{FSL},n}$ be a fundamental shadow link in $M_c = \#^{c+1}(S^2 \times S^1)$ for some $c\in \NN$, and let $L' \subset S^3$ be the disjoint union of $c+1$ unknots with the 0-framings by doing surgery along which we get $M_c$. Let $M$ be a closed oriented $3$-manifold and $L=L_1\cup\dots\cup L_n$ be a framed link inside $M$ with $n$ components. Suppose $M\setminus L$ is homeomorphic to $M_c\setminus L_{\text{FSL}}$. Then, up to reordering if necessary, there exist a partition $\{I, J\}$ of $\{1, 2, .., n\}$ together with $p_i \in \ZZ$ and $q_i \in \ZZ\setminus\{0\}$ for each $i\in I$ such that 
\begin{enumerate}
\item $M$ is obtained by doing $(p_i/q_i)$ surgery along $L_{\text{FSL},i}$ and $(1,0)$ surgery along $L_{\text{FSL},j}$ in $M_c$,
\item the $i$-th component of $L$ in $M_c \setminus L_{\text{FSL}}$ is isotopic to a curve on the boundary of the tubular neighborhood of $L_{\text{FSL},i}$ that intersects the $(p_i, q_i)$-curve of the boundary at exactly one point, and
\item $L_j$ and $L_{\text{FSL},j}$ are isotopic in $M_c$ for all $j \in J$.
\end{enumerate}

For each $i\in I$, consider a continued fraction expansion
\begin{align*}
     \frac{p_i}{q_i} = a_{i,\zeta_i} - \frac{1}{a_{i,\zeta_i-1} - \frac{1}{\dots -\frac{1}{a_{i,1}}}},
\end{align*}
where $\zeta_i \in \NN$. We replace $L_{\text{FSL},i}$ by another framed link $\tilde{L}_{\text{FSL},i}$ of $\zeta_i$ many components with framings $a_{i,1},\dots, a_{i,\zeta_i}$ according to Figure~\ref{surg}. Let $\tilde{L}_{\text{FSL},I} = \bigcup_{i\in I} \tilde{L}_{\text{FSL},i}$. By eg \cite[ p.273]{R}, $M$ can also be obtained by doing surgery along the framed link $ \tilde{L}_{\text{FSL},I}  \cup L' \subset \SS^3$. 

\begin{figure}[h]
\centering
\includegraphics[scale=0.2]{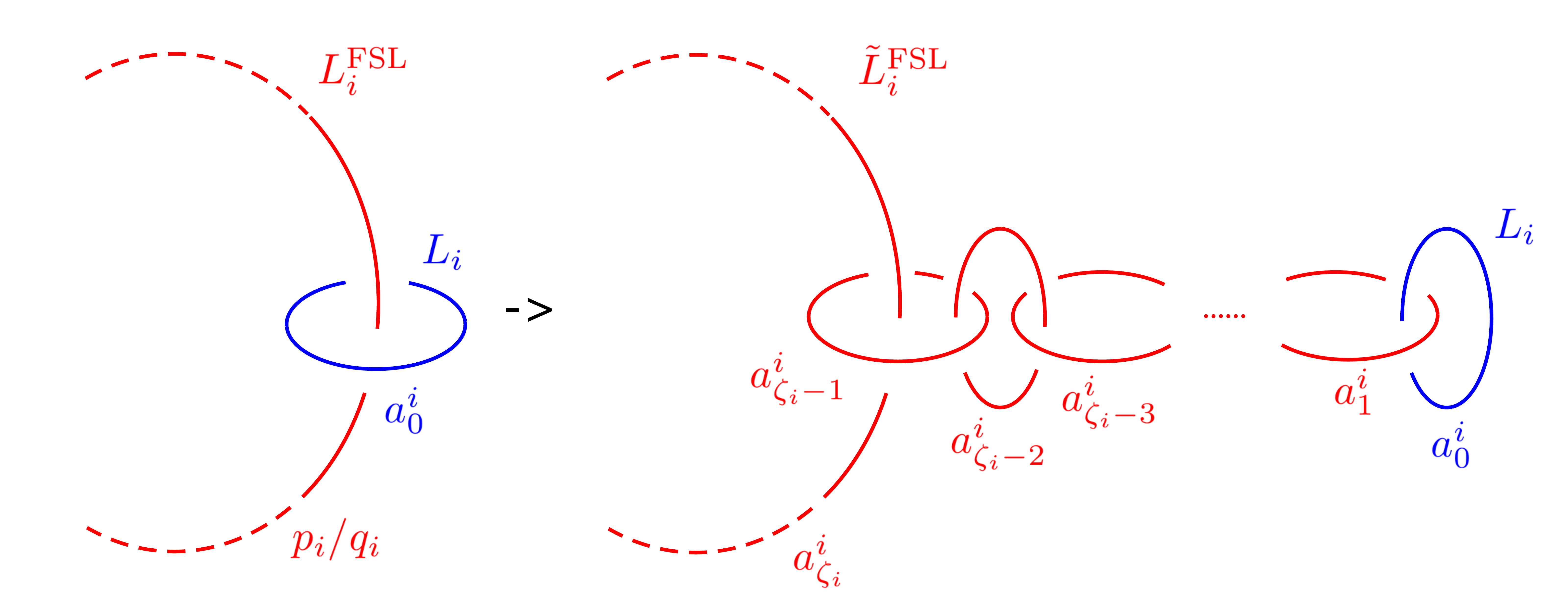}
\caption{Changing each $L_{\text{FSL},i}$ to $\tilde L_{\text{FSL},i}$}\label{surg}
\end{figure}

We let $\mathbf n_I = (n_i)_{i\in I} \in I_r^{|I|}$ and $\mathbf m_J=(m_j)_{j\in J}\in I_r^{|J|}$ be colors on the $I$ and $J$ components of $L$ respectively. Denote the framings of the $I$ and $J$ components by $a_{i,0}$ and $a_{j,0}$ respectively, where $i\in I$ and $j\in J$. First of all, we compute the $(\mathbf{n}_I, \mathbf{m}_J)$-th relative Reshetikhin-Turaev invariants of the pair $(M,L)$.

\begin{proposition}\label{RTformula}
\begin{align}\label{rt}
&RT_r(M,L,(\mathbf{n}_I, \mathbf{m}_J)) \notag\\ 
&= 
    \frac{\mu_r^{ \sum_{i \in I} {\zeta_i} - c}}{ \{1\}^{\sum_{i\in I} \zeta_i}} e^{-\sigma(\tilde{L}_{\text{FSL},I} \cup L')(-\frac{3}{r} - \frac{r+1}{4})\sqrt{-1}\pi} 
     \prod_{i\in I} q^{\frac{a_{i,0} n_i (n_i +2)}{2}}
     \prod_{j\in J}(-1)^{\frac{\iota_jm_j}{2}}q^{\big(a_{j,0}+\frac{\iota_j}{2}\big)\frac{m_j(m_j+2)}{2}}\notag \\
&  \qquad \times \sum_{ \mathbf m_I, \mathbf m_{\zeta_I}} \lt[ \lt(\prod_{i \in I} 
     \{(n_i + 1 )(m_{i,1} + 1)\}\{(m_{i,1} + 1)(m_{i,2} + 1)\}\dots\{(m_{i,\zeta_i-1} + 1)(m_{i,\zeta_i} + 1)\} \rt) \rt.\notag \\
&   \qquad\qquad\qquad\lt. \lt(  \prod_{i\in I} (-1)^{\frac{\iota_i m_{i,\zeta_i}}{2}}q^{\sum_{l=1}^{\zeta_i-1}\frac{a_{i,l} m_{i,l} (m_{i,l} + 2) }{2}
+ \lt(a_{i,\zeta_i} +\frac{\iota_i}{2}\rt) \frac{m_{i,\zeta_i} (m_{i,\zeta_i} + 2)}{2}
} \rt) 
    \prod_{s = 1}^{c}
    \begin{vmatrix}
    m_{s_1} & m_{s_2} & m_{s_3}\\
    m_{s_4} & m_{s_5} & m_{s_6}
    \end{vmatrix} \rt],
\end{align} 
where the sum is over multi-even integers $\mathbf m_I=(\mathbf m_i)_{i\in I} \in \{0,2,\dots,r-3\}^{\sum_{i\in I}\zeta_i - |I|}$ with each $(\mathbf m_i)=(m_{i,1},\dots, m_{i,\zeta_i -1}) \in \{0,2,\dots,r-3\}^{\zeta_i - 1}$ and multi-even integers $\mathbf m_{\zeta_I}=(m_{i,\zeta_i})_{i\in I}\in \{0,2,\dots,r-3\}^{|I|}$, and $m_{s_1} , \dots, m_{s_6}$ are the colors of the edges of the building block $\Delta_s$ inherited from the colors on $L_{\text{FSL}}$.
\end{proposition}
\begin{proof} The terms 
$$\prod_{j \in J} q^{\frac{a_{j,0} m_j (m_j +2)}{2}}, \quad \prod_{i\in I} q^{\frac{a_{i,0} n_i (n_i +2)}{2}}\quad \text{ and } \quad  q^{\frac{\sum_{l=1}^{\zeta_i}a_{i,l} m_{i,l} (m_{i,l} + 2) }{2}}$$ 
come from changing the framings of all the link components to zero. Besides, the term 
$$ \lt(\prod_{i \in I} 
     \{(n_i + 1 )(m_{i,1} + 1)\}\{(m_{i,1} + 1)(m_{i,2} + 1)\}\dots\{(m_{i,\zeta_i-1} + 1)(m_{i,\zeta_i} + 1)\} \rt) $$
comes from the skein computation
$$\Bigg\langle \cp{\includegraphics[width=1.6cm]{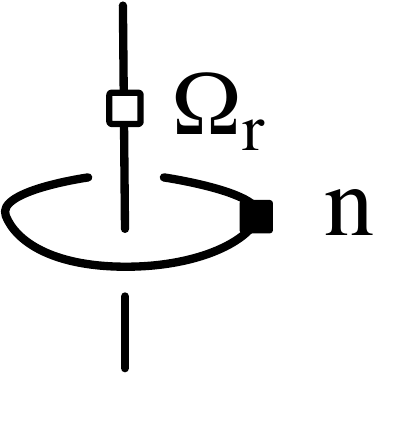}}\Bigg\rangle=\mu_r\sum_{m\in \mathrm{I}_r}\mathrm H(m,n)\Bigg\langle \cp{\includegraphics[width=0.7cm]{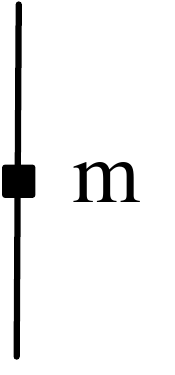}}\Bigg\rangle,$$
where 
$$ \mathrm{H}(m,n) = (-1)^{m+n} \frac{\{(m+1)(n+1)\}}{\{1\}} = \frac{\{(m+1)(n+1)\}}{\{1\}}  $$
for even integers $m,n$.
Together with Proposition \ref{FSL}, the result follows.
\end{proof}

Recall from \cite{WY1} that
\begin{proposition}\label{6jqd} (\cite[Proposition 3.1]{WY1}) The quantum $6j$-symbol at the root of unity $q=e^{\frac{2\pi\sqrt{-1}}{r}}$ can be computed as 
$$\bigg|
\begin{matrix}
        m_1 & m_2 & m_3 \\
        m_4 & m_5 & m_6 
      \end{matrix} \bigg|=\frac{\{1\}}{2}\sum_{k=\max\{T_1,T_2,T_3,T_4\}}^{\min\{Q_1,Q_2,Q_3,r-2\}}e^{\frac{r}{4\pi\sqrt{-1}}U_r\big(\frac{2\pi m_1}{r},\dots,\frac{2\pi m_6}{r},\frac{2\pi k}{r}\big)},$$
 where $U_r$ is defined as follows. If $(m_1,\dots,m_6)$ is of hyperideal type, then
\begin{equation}\label{term}
\begin{split}
U_r(\alpha_1,\dots,\alpha_6,\xi)=&\pi^2-\Big(\frac{2\pi}{r}\Big)^2+\frac{1}{2}\sum_{i=1}^4\sum_{j=1}^3(\eta_j-\tau_i)^2-\frac{1}{2}\sum_{i=1}^4\Big(\tau_i+\frac{2\pi}{r}-\pi\Big)^2\\
&+\Big(\xi+\frac{2\pi}{r}-\pi\Big)^2-\sum_{i=1}^4(\xi-\tau_i)^2-\sum_{j=1}^3(\eta_j-\xi)^2\\
&-2\varphi_r\Big(\frac{\pi}{r}\Big)-\frac{1}{2}\sum_{i=1}^4\sum_{j=1}^3\varphi_r\Big(\eta_j-\tau_i+\frac{\pi}{r}\Big)+\frac{1}{2}\sum_{i=1}^4\varphi_r\Big(\tau_i-\pi+\frac{3\pi}{r}\Big)\\
&-\varphi_r\Big(\xi-\pi+\frac{3\pi}{r}\Big)+\sum_{i=1}^4\varphi_r\Big(\xi-\tau_i+\frac{\pi}{r}\Big)+\sum_{j=1}^3\varphi_r\Big(\eta_j-\xi+\frac{\pi}{r}\Big),\\
\end{split}
\end{equation}
where $\tau_1=\frac{\alpha_1+\alpha_2+\alpha_3}{2},$ $\tau_2=\frac{\alpha_1+\alpha_5+\alpha_6}{2},$ $\tau_3=\frac{\alpha_2+\alpha_4+\alpha_6}{2}$ and $\tau_4=\frac{\alpha_3+\alpha_4+\alpha_5}{2},$ $\eta_1=\frac{\alpha_1+\alpha_2+\alpha_4+\alpha_5}{2},$ $\eta_2=\frac{\alpha_1+\alpha_3+\alpha_4+\alpha_6}{2}$ and $\eta_3=\frac{\alpha_2+\alpha_3+\alpha_5+\alpha_6}{2}.$
If $(m_1,\dots,m_6)$ is not of the hyperideal type, then $U_r$ will be changed according to Lemma \ref{factorial}.
\end{proposition}

We are going to apply the Gauss sum formula (Lemma \ref{sm} below) to write the relative Reshetikhin-Turaev invariants as a sum of the evaluation of certain holomorphic function (Proposition \ref{computation}). Recall that for each $i\in I$, we have
$$\frac{p_i}{q_i}=a_{i,\zeta_i}-\frac{1}{a_{i,\zeta_i-1}-\frac{1}{\cdots-\frac{1}{a_{i,1}}}}.$$
With respect to the continued fraction $[a_{i,1},\dots,a_{i,\zeta_i}]$, for each $i\in I$, let 
\begin{itemize}
\item $A_{i,l}, B_{i,l}, C_{i,l}, D_{i,l}$ be the integers defined in (\ref{defABCD}) for each $l = 1,\dots, \zeta_i$,
\item $p_i', q_i'$ and $K_{i,l}$ be the quantities defined in (\ref{p'}) and (\ref{K}) for each $l = 1,\dots, \zeta_i$,
\item $I_i(s_i), J_i(s_i)$ and $K_i(s_i)$ be the functions defined in Lemma \ref{arith}, where $s_i \in \{0,1,\dots,|q_i|-1\}$.
\end{itemize}
For any $n_{i}, m_{i,\zeta_i} \in \NN$, consider the sum
\begin{align}
&S_i(m_{i,\zeta_i},n_{i}) \notag\\
=& \sum_{m_{i,1},\dots,m_{i,\zeta_i-1}=0}^{r-1} (-1)^{\sum_{l=1}^{\zeta_i-1} a_{i,l} m_{i,l}} q^{\sum_{l=1}^{\zeta_i-1} \frac{a_{i,l} m_{i,l}^2}{2}}(q^{n_{i} m_{i,1}} - q^{-n_{i} m_{i,1}})q^{\sum_{l=1}^{\zeta_i-1} m_{i,l} m_{i,l+1}} \label{Sisum},
\end{align}
where $q=e^{\frac{2\pi \sqrt{-1}}{r}}$. This sum will appear later in Proposition \ref{computation}.

\begin{lemma}\label{sm} At $q=e^{\frac{2\pi \sqrt{-1}}{r}}$, we have
\begin{align*}
S_i(m_{i,\zeta_i},n_{i}) 
= \frac{(-1)^{\zeta_i+1}(\sqrt{-1}r)^{\frac{\zeta_i-1}{2}}}{\sqrt{q_i}} \sum_{s_i=0}^{|q_i|-1}\Bigg( e^{\frac{r}{4\pi\sqrt{-1}} Z_i^+\Big(s_i, \frac{2\pi m_{i,\zeta_i}}{r}, \frac{2\pi n_{i}}{r}\Big)} - e^{\frac{r}{4\pi\sqrt{-1}} Z_i^-\Big(s_i, \frac{2\pi m_{i,\zeta_i}}{r}, \frac{2\pi n_{i}}{r}\Big)} \Bigg),
\end{align*}
where
\begin{align*}
Z_i^\pm(s,\alpha,\beta)
=& \frac{C_{i,\zeta_i-1}}{q_i}(\alpha-\pi)^2 \mp \frac{2(\beta-\pi)(\alpha-\pi)}{q_i} - \frac{2\pi (I_i(s) \pm 1)}{q_i}(\alpha - \pi) \\
&+ K_i(s) \pi^2 - \frac{p_i'}{q_i} \beta^2 \mp 2\beta \pi J_i(s).
\end{align*}
\begin{proof}
First of all, we consider a closely related sum
\begin{align*}
\tilde{S}_i(m_{i,\zeta_i},n_{i}) 
= \sum_{m_{i,1},\dots,m_{i,\zeta_i-1}=0}^{r-1} & (-1)^{\sum_{l=1}^{\zeta_i-1} a_{i,l} m_{i,l}} q^{\sum_{l=1}^{\zeta_i-1} \frac{a_{i,l} m_{i,l}^2}{2}}(q^{n_{i} m_{i,1}} - q^{-n_{i} m_{i,1}}) \\
& \times\prod_{l=1}^{\zeta_i-1}\Big(q^{m_{i,l} m_{i,l+1}} - q^{-m_{i,l} m_{i,l+1}} \Big).
\end{align*}
By considering the transformation $m_{i,l} \mapsto r-m_{i,l}$, a direct computation shows that 
$$
\tilde{S}_i(m_{i,\zeta_i},n_{i})  = 2^{\zeta_i -1} S_i(m_{i,\zeta_i},n_{i}) .
$$
As a result, by Lemma 3.6 in \cite{WY} (note that our variable $q$ is the variable $t^{\frac{1}{2}}$ in \cite{WY}),
\begin{align*}
S_i(m_{i,\zeta_i},n_{i}) 
=&\frac{1}{2^{\zeta_i -1}}\tilde{S}_i(m_{i,\zeta_i},n_{i})  \\
=& \tau_i^+ \sum_{s_i=0}^{2|q_i|-1} e^{-\frac{\pi\sqrt{-1}}{r} \frac{C_{i,\zeta_i-1}}{q_i} \Big(m_{i,\zeta_i} + s_ir + \frac{rK_{i,\zeta_i-1}}{2} - \frac{(-1)^{\zeta_i} n_{i}}{C_{i,\zeta_i-1}} \Big)^2}\\
&\qquad- \tau_i^- \sum_{s_i=0}^{2|q_i|-1} e^{-\frac{\pi\sqrt{-1}}{r} \frac{C_{i,\zeta_i-1}}{q_i} \Big(m_{i,\zeta_i} + s_ir + \frac{rK_{i,\zeta_i-1}}{2} + \frac{(-1)^{\zeta_i} n_{i}}{C_{i,\zeta_i-1}} \Big)^2},
\end{align*}
where
\begin{align*}
&\tau_i^\pm 
= \frac{(\sqrt{-1}r)^{\frac{\zeta_i-1}{2}}}{2\sqrt{q_i}} \\
&\qquad \times e^{-\frac{\pi\sqrt{-1}}{r} n_{i}^2 \Big(\sum_{l=1}^{\zeta_i-2} \frac{1}{C_{i,l}C_{i,l+1}}\Big) - \frac{\pi\sqrt{-1}r}{4} \sum_{l=1}^{\zeta_i-2} \frac{C_{i,l}K_{i,l}^2}{C_{i,l+1}} \mp \pi\sqrt{-1} n_{i} \Big(\sum_{l=1}^{\zeta_i-2} (-1)^{l+1} \frac{K_{i,l}}{C_{i,l+1}}\Big)}.
\end{align*}
Moreover, since all of $C_{i,\zeta_i},$ $m_{i,\zeta_i},$ $s_i,$ $q_i, n_i$ and $r$ are integers, a direct computation shows that
$$\frac{e^{-\frac{\pi\sqrt{-1}}{r} \frac{C_{i,\zeta_i-1}}{q_i} \Big(m_{i,\zeta_i} + (s_i+q_i)r + \frac{rK_{i,\zeta_i-1}}{2} \pm \frac{(-1)^{\zeta_i} n_{i}}{C_{i,\zeta_i-1}} \Big)^2}}{e^{-\frac{\pi\sqrt{-1}}{r} \frac{C_{i,\zeta_i-1}}{q_i} \Big(m_{i,\zeta_i} + s_ir + \frac{rK_{i,\zeta_i-1}}{2} \pm \frac{(-1)^{\zeta_i} n_{i}}{C_{i,\zeta_i-1}} \Big)^2}}=e^{-\pi\sqrt{-1}r(K_{i,\zeta_i-1}C_{i,\zeta_i-1}+C_{i,\zeta_i-1}q_i)}=1,$$
where the last equality comes from Lemma \ref{arith'}  that  $K_{i,\zeta_i-1}C_{i,\zeta_i-1}+C_{i,\zeta_i-1}q_i$ is an even integer. Thus, for each $s_i\in\{0,\dots, |q_i|-1\},$ 
$$e^{-\frac{\pi\sqrt{-1}}{r} \frac{C_{i,\zeta_i-1}}{q_i} \Big(m_{i,\zeta_i} + (s_i+q_i)r + \frac{rK_{i,\zeta_i-1}}{2} \pm \frac{(-1)^{\zeta_i} n_{i}}{C_{i,\zeta_i-1}} \Big)^2}=e^{-\frac{\pi\sqrt{-1}}{r} \frac{C_{i,\zeta_i-1}}{q_i} \Big(m_{i,\zeta_i} + s_ir + \frac{rK_{i,\zeta_i-1}}{2} \pm \frac{(-1)^{\zeta_i} n_{i}}{C_{i,\xi_i-1}} \Big)^2}.$$

In particular, we can write
\begin{align*}
&S_i(m_{i,\zeta_i},n_{i}) 
= \tau_i
\sum_{s_i=0}^{|q_i|-1} e^{-\frac{\pi\sqrt{-1}}{r} \frac{C_{i,\zeta_i-1}}{q_i} \Big(m_{i,\zeta_i} + s_ir + \frac{rK_{i,\zeta_i-1}}{2}\Big)^2} \\
&\times\Bigg(e^{-\pi\sqrt{-1}n_{i}\Big( \frac{(-1)^{\zeta}}{rq_i} ( 2m_{i,\zeta_i} + 2s_ir) + \sum_{l=1}^{\zeta_i-1} \frac{(-1)^{l+1}K_{i,l}}{C_{i,l+1}}\Big)}
- e^{\pi\sqrt{-1}n_{i}\Big( \frac{(-1)^{\zeta}}{rq_i} ( 2m_{i,\zeta_i} + 2s_ir) + \sum_{l=1}^{\zeta_i-1} \frac{(-1)^{l+1}K_{i,l}}{C_{i,l+1}}\Big)} \Bigg),
\end{align*}
where
$$
\tau_i
= \frac{(\sqrt{-1}r)^{\frac{\zeta_i-1}{2}}}{\sqrt{q_i}}
e^{-\frac{\pi\sqrt{-1}}{r} n_{i}^2 \Big(\sum_{l=1}^{\zeta_i-1} \frac{1}{C_{i,l}C_{i,l+1}}\Big) - \frac{\pi\sqrt{-1}r}{4} \sum_{l=1}^{\zeta_i-2} \frac{C_{i,l}K_{i,l}^2}{C_{i,l+1}}}.
$$
By a direct computation, 
\begin{align*}
&e^{-\frac{\pi\sqrt{-1}}{r} \frac{C_{i,\zeta_i-1}}{q_i} \Big(m_{i,\zeta_i} + s_ir + \frac{rK_{i,\zeta_i-1}}{2}\Big)^2 - \frac{\pi\sqrt{-1}r}{4} \sum_{l=1}^{\zeta_i-2} \frac{C_{i,l}K_{i,l}^2}{C_{i,l+1}}}\\
=& 
e^{\frac{r}{4\pi\sqrt{-1}} \Big(\frac{C_{i,\zeta_i-1}}{q_i}\Big(\frac{2\pi m_{i,\zeta_i}}{r} -\pi \Big)^2 - \frac{2\pi I_i(s_i)}{q_i} \Big(\frac{2\pi m_{i,\zeta_i}}{r} -\pi \Big) + K_i(s_i) \pi^2 \Big)}.
\end{align*}
Besides, by Lemma \ref{cf} and (\ref{sump'q'}),
$$
e^{-\frac{\pi\sqrt{-1}}{r} n_{i}^2 \Big(\sum_{l=1}^{\zeta_i-1} \frac{1}{C_{i,l}C_{i,l+1}}\Big)}
= e^{\frac{r}{4\pi\sqrt{-1}} \Big(-\frac{p_i'}{q_i}\Big)\Big(\frac{2\pi n_{i}}{r}\Big)^2}.
$$
Moreover,
\begin{align*}
& e^{\mp\pi\sqrt{-1}n_{i}\Big( \frac{(-1)^{\zeta_i}}{rq_i} ( 2m_{i,\zeta_i} + 2s_ir) + \sum_{l=1}^{\zeta_i-1} \frac{(-1)^{l+1}K_{i,l}}{C_{i,l+1}}\Big)}\\
&= 
e^{\pm(-1)^{\zeta_i} \frac{r}{4\pi\sqrt{-1}} \Big(2 \Big(\frac{2\pi n_{i}}{r}\Big( \Big(\frac{2\pi m_{i,\zeta_i}}{r} - \pi \Big) \frac{1}{q_i} + \pi J_i(s_i)\Big)\Big) \Big)}\\
&=  e^{\pm (-1)^{\zeta_i} \frac{r}{4\pi\sqrt{-1}} \Big(\frac{2}{q_i} \Big(\frac{2\pi n_i}{r}-\pi\Big)\Big(\frac{2\pi m_{i,\zeta_i}}{r}-\pi\Big) + \frac{2\pi}{q_i}\Big(\frac{2\pi m_{i,\zeta_i}}{r} - \pi\Big) + 2\pi \Big(\frac{2\pi n_{i}}{r} \Big) J_i(s_i)\Big)}.
\end{align*}
The result follows from the above computation.
\end{proof}
\end{lemma}

\begin{proposition}\label{computation}
$$\mathrm{RT}_r(M,L,(\mathbf n_I,\mathbf m_J))= Z_r\sum_{\mathbf s_I, \mathbf m_{\zeta_I}, \mathbf k, \mathbf E_I} g_r^{\mathbf E_I}(\mathbf s_I, \mathbf m_{\zeta_I}, \mathbf k) ,$$
where 
\begin{enumerate}
\item $Z_r$ is given by
\begin{align*}
Z_r &= \frac{ (-1)^{\sum_{i\in I}(\zeta_i + 1 + \sum_{l=1}^{\zeta_i}a_{i,l})} (\sqrt{-1}r)^{\sum_{i\in I}\frac{\zeta_i -1}{2}}\mu_r^{\sum_{i\in I}\zeta_i - c}}{2^c\{ 1 \}^{\sum_{i \in I} \zeta_i - c}\sqrt{\prod_{i\in I} q_i}} \\
&\qquad  e^{\frac{\pi\sqrt{-1}}{r} \sum_{i\in I} \sum_{l=1}^{\zeta_i-1} a_{i,l} -\frac{r\pi \sqrt{-1}}{4} ( \sum_{i\in I} (a_{i,0} +a_{i,\zeta_i})  + \sum_{j\in J} a_{j,0}) +\sigma(\tilde{L}_{\text{FSL},I} \cup L')(\frac{3}{r} + \frac{r+1}{4})\sqrt{-1}\pi},
\end{align*}
\item $\mathbf s_I = (s_i)_{i \in I}$ where each $s_i$ runs over all integer in $\{0,\dots,|q_i|-1\}$,
\item $\mathbf m_{\zeta_I}=(\mathbf m_{i,\zeta_i})_{i\in I}$ runs over all multi-even integers in $\{0,2,\dots,r-3\}$ so that for each $s\in\{1,\dots,c\}$, the triples $(m_{s_1},m_{s_2},m_{s_3}),$  $(m_{s_1},m_{s_5},m_{s_6}),$  $(m_{s_2},m_{s_4},m_{s_6})$ and  $(m_{s_3},m_{s_4},m_{s_5})$ are $r$-admissible,
\item 
$\mathbf k=(k_1,\dots,k_c)$ runs over all multi-integers with each $k_s$ lying in between $\max\{T_{s_i}\}$ and $\min\{Q_{s_j},r-2\},$
\item $\mathbf E_I = (E_i)_{i\in I} \in \{-1, 1\}^{|I|}$ runs over all multi-sign, 
\item the function $g_r^{\mathbf E_I}({\mathbf s_I}, {\mathbf m_{\zeta_I}},{\mathbf k}) $ is defined by
\begin{align*}
g_r^{\mathbf E_I}({\mathbf s_I},{\mathbf m_{\zeta_I}},{\mathbf k}) 
&= \Bigg(\prod_{i \in I} E_i\Bigg) e^{\sqrt{-1} P_r^{\mathbf E_I}\lt(\mathbf s_I,\frac{2\pi \mathbf m_{\zeta_I}}{r} \rt) +  \frac{r}{4\pi \sqrt{-1}} W_r\lt(\mathbf s_I,\frac{2\pi \mathbf m_{\zeta_I}}{r}, \frac{2\pi \mathbf k}{r} \rt)},
\end{align*}
where $\frac{2\pi \mathbf k}{r} = (\frac{2\pi k_1}{r}, \dots , \frac{2\pi k_c}{r}),$ $\frac{2\pi \mathbf m_{\zeta_I}}{r} = \lt(\frac{2\pi m_{i,\zeta_i}}{r}\rt)_{i\in I}$, $\mathbf s_I=(s_i)_{i\in I}$,
\begin{align*}
P_r^{\mathbf E_I}(\mathbf s_I, \boldsymbol{\alpha}_{\zeta_I}) 
=& \sum_{i\in I}\Bigg(\frac{p_i'}{q_i}(\beta_i - \pi) + \frac{p_i}{q_i}(\alpha_{i,\zeta_i} - \pi) + \frac{E_i(\alpha_{i,\zeta_i}+\beta_i-2\pi)}{q_i}\Bigg)\\
&+ \pi\sum_{i\in I}\Bigg( \frac{I_i(s_i) + E_i}{q_i} + \frac{p_i'}{q_i} + E_i J_i(s_i)\Bigg)\\
&+ \sum_{i\in I}a_{i,0}\beta_i + \sum_{i\in I} \Big(\frac{\iota_i}{2}\Big)\alpha_{i,\zeta_i} + \sum_{j\in J}\Big(a_{j,0}+\frac{\iota_j}{2}\Big)\alpha_j
\end{align*}
and 
\begin{align*}
W_r(\mathbf s_I, \boldsymbol{\alpha}_{\zeta_I}, \boldsymbol \xi)  
&=  - \sum_{i \in I}  \Big(a_{i,0} + \frac{p_i'}{q_i}\Big)(\beta_i - \pi)^2 - \sum_{j \in J} \Big(a_{j,0}+\frac{\iota_j}{2}\Big)(\alpha_j - \pi)^2  \\
&\qquad -\sum_{i=1}^n \Big(\frac{p_i}{q_i}+\frac{\iota_i}{2}\Big)(\alpha_{i,\zeta_i}-\pi)^2 - \sum_{i\in I}\frac{2\pi(I_i(s_i) - E_i)}{q_i}(\alpha_{i,\zeta_i}-\pi)\\
&\qquad - \sum_{i\in I}\frac{2E_i (\alpha_{i,\zeta_i}-\pi)(\beta_i-\pi)}{q_i}
+ \sum_{i\in I}2\pi\beta_i\Big(-E_iJ_i(s_i) - \frac{p_i}{q_i} \Big)\\
&\qquad +\sum_{s=1}^c U_r(\alpha_{s_1},\dots,\alpha_{s_6},\xi_s)
+\sum_{i\in I}\pi^2\Bigg(K_i(s_i)+\frac{p_i'}{q_i}\Bigg)
+\Big(\sum_{i=1}^n\frac{\iota_i}{2}\Big)\pi^2 \\
&\qquad +\frac{4\pi^2}{r^2} h_I
\end{align*}
with $\beta_i = \frac{2\pi n_i}{r}$ for $i\in I$, $\alpha_j = \frac{2\pi m_j}{r}$ for $j\in J$ and $h_I = \sum_{i \in I} \frac{C_{i,\zeta_i-1} +2E_i - p_i' }{q_i}$.
\end{enumerate}
\end{proposition}

\begin{proof}

By Proposition \ref{RTformula}, we have
\begin{align*}
&RT_r(M,L,(\mathbf{n}_I, \mathbf{m}_J)) \notag\\ 
=& 
    \frac{\mu_r^{ \sum_{i \in I} {\zeta_i} - c}}{ \{1\}^{\sum_{i\in I} \zeta_i}} e^{-\sigma(\tilde{L}_{\text{FSL},I} \cup L')(-\frac{3}{r} - \frac{r+1}{4})\sqrt{-1}\pi} \prod_{i\in I} q^{\frac{a_{i,0} n_i (n_i +2)}{2}}
     \prod_{j\in J}(-1)^{\frac{\iota_jm_j}{2}}q^{\big(a_{j,0}+\frac{\iota_j}{2}\big)\frac{m_j(m_j+2)}{2}}\notag \\
&  \qquad \times 
\sum_{ \mathbf m_I } 
\sum_{ \boldsymbol \epsilon_I}
\sum_{ \mathbf{m_{\zeta_i}}}
\lt(\prod_{i\in I}
S_i^{(\epsilon_{i,1}, \dots, \epsilon_{i,\zeta_i - 1})}(m_{i,1},\dots, m_{i,\zeta_i}) \rt)
\lt(  \prod_{i\in I} (-1)^{\frac{\iota_i m_{i,\zeta_i}}{2}}q^{\lt(a_{i,\zeta_i}+\frac{\iota_i}{2}\rt) \frac{m_{i,\zeta_i} (m_{i,\zeta_i} + 2)}{2}
} \rt)  \\
&  \qquad \times \prod_{s = 1}^{c}
    \begin{vmatrix}
    m_{s_1} & m_{s_2} & m_{s_3}\\
    m_{s_4} & m_{s_5} & m_{s_6}
    \end{vmatrix}
,
\end{align*}
where 
\begin{itemize}
\item 
$\mathbf m_I=(\mathbf m_i)_{i\in I}$ with $\mathbf m_i=(m_{i,1},\dots, m_{i,\zeta_i -1})$ runs over all multi-even integers in $\{0,2,\dots,r-3\}$,
\item
$\mathbf m_{\zeta_I}  = (m_{i,\zeta_i})_{i\in I}$ runs over all multi-even integers in $\{0,2,\dots,r-3\}$ so that for each $s\in\{1,\dots,c\}$, the triples $(m_{s_1},m_{s_2},m_{s_3}),$  $(m_{s_1},m_{s_5},m_{s_6}),$  $(m_{s_2},m_{s_4},m_{s_6})$ and  $(m_{s_3},m_{s_4},m_{s_5})$ are $r$-admissible, 
\item 
${\boldsymbol \epsilon_I} = ({\boldsymbol \epsilon_i})_{i\in I}$ with each
${\boldsymbol \epsilon_i}=(\epsilon_{i,1}, \dots, \epsilon_{i,\zeta_i-1})\in \{0,1\}^{\zeta_i-1}$ and
\item $S_i^{(\epsilon_{i,1}, \dots, \epsilon_{i,\zeta_i -1})}(m_{i,1},\dots, m_{i,\zeta_i-1})$ is given by
\begin{align*}
&S_i^{(\epsilon_{i,1}, \dots, \epsilon_{i,\zeta_i -1 })}(m_{i,1},\dots, m_{i,\zeta_i-1}) \\
=&  \lt( q^{(n_i+1)(m_{i,1}+1)} - q^{-(n_i+1)(m_{i,1}+1)}  \rt)
(-1)^{\sum_{l=1}^{\zeta_i-1} \epsilon_{i,l}} (-1)^{\sum_{l=1}^{\zeta_i-1} a_{i,l} m_{i,l}} \\ &\qquad \times q^{\sum_{l=1}^{\zeta_i-1} \frac{a_{i,l} m_{i,l} (m_{i,l} + 2)}{2} + \sum_{l=1}^{\zeta_i - 1} (-1)^{\epsilon_{i,l}} (m_{i,l} + 1)(m_{i,l+1} + 1)
}.
\end{align*}
\end{itemize}
Note that in the formula of $S_i^{(\epsilon_l^i, \dots, \epsilon_{\zeta_i -1 }^i)}$, the term $(-1)^{\sum_{l=1}^{\zeta_i} a_{i,l} m_{i,l}}$ is equal to $1$ when $m_{i,l}$ is even. Nevertheless, we need this term for the next computation.

A direct computation shows that for each $i\in I$, we have 
\begin{align*}
S_i^{(1, \epsilon_{i,2}, \dots, \epsilon_{i,\zeta_i -1})}
(m_{i,1}, m_{i,2}, \dots, m_{i,\zeta_i-1})
&= S_i^{(0, \epsilon_{i,2}, \dots, \epsilon_{i,\zeta_i -1})}
(r-2-m_{i,1}, m_{i,2}, \dots , m_{i,\zeta_i-1}).
\end{align*}
Note that since $r$ is odd, $r-2-m_{i,1}$ runs through all odd integers from $0$ to $r-2$. More generally, we have 
\begin{align}\label{eventoodd}
&S_i^{(0, \dots, 0, 1, \epsilon_{i,l+1}, \dots, \epsilon_{i,\zeta_i -1})}
(m_{i,1}, \dots, m_{i,l-1}, m_{i,l}, m_{i,l+1}, \dots, m_{i,\zeta_i-1})\notag\\
&= S_i^{(0, \dots, 0, 0, \epsilon_{i,l+1}, \dots, \epsilon_{i,\zeta_i -1})}
(r-2-m_{i,1}, \dots, r-2- m_{i,l-1}, r-2-m_{i,l}, m_{i,l+1}, \dots, m_{i,\zeta_i-1}).
\end{align}
Originally, $\mathbf m_i = (m_{i,1},\dots,m_{i,\zeta_i-1})$ run through all multi-even integers in $\{0,2,\dots,r-3\}^{\zeta_i-1}$. By (\ref{eventoodd}), we can change the sum $\mathbf m_I$ to be over all integers in $\{0,1,\dots,r-2\}^{\sum_{i\in I}\zeta_i-|I|}$ and write
\begin{align*}
&RT_r(M,L,(\mathbf{n}_I, \mathbf{m}_J)) \notag\\ 
=&     \frac{\mu_r^{ \sum_{i \in I} {\zeta_i} - c}}{ \{1\}^{\sum_{i\in I} \zeta_i}} e^{-\sigma(\tilde{L}_{\text{FSL},I} \cup L')(-\frac{3}{r} - \frac{r+1}{4})\sqrt{-1}\pi} \prod_{i\in I} q^{\frac{a_{i,0} n_i (n_i +2)}{2}}
     \prod_{j\in J}(-1)^{\frac{\iota_jm_j}{2}}q^{\big(a_{j,0}+\frac{\iota_j}{2}\big)\frac{m_j(m_j+2)}{2}}\notag \\
&  \qquad  
\sum_{ \mathbf{m_{\zeta_i}}}
\lt( \prod_{i\in I} (-1)^{\frac{\iota_i m_{i,\zeta_i}}{2}}q^{\lt(a_{i,\zeta_i} +\frac{\iota_i}{2}\rt) \frac{m_{i,\zeta_i} (m_{i,\zeta_i} + 2)}{2}
} 
S_i^{(0,0,\dots,0)}(m_{i,1}, m_{i,2}, \dots, m_{i,\zeta_i-1})
\rt)\\
&\qquad \times \prod_{s = 1}^{c}
    \begin{vmatrix}
    m_{s_1} & m_{s_2} & m_{s_3}\\
    m_{s_4} & m_{s_5} & m_{s_6}
    \end{vmatrix},
    \end{align*}
where 
\begin{align*}
&S_i^{(0,0,\dots,0)}(m_{i,1}, m_{i,2}, \dots, m_{i,\zeta_i-1})\\
&= \sum_{ \mathbf m_I} 
\lt( q^{(n_i+1)(m_{i,1}+1)} - q^{-(n_i+1)(m_{i,1}+1)}  \rt)
(-1)^{\sum_{l=1}^{\zeta_i-1} a_{i,l} m_{i,l}} q^{\sum_{l=1}^{\zeta_i-1} \frac{a_{i,l} m_{i,l} (m_{i,l} + 2)}{2} + \sum_{l=1}^{\zeta_i - 1}  (m_{i,l} + 1)(m_{i,l+1} + 1)
}.
\end{align*}
Note that
\begin{align*}
S_i^{(0,0,\dots,0)}(m_{i,1}, m_{i,2}, \dots, m_{i,\zeta_i-1})
&= (-1)^{\sum_{l=1}^{\zeta_i-1} a_{i,l}} q^{-\sum_{l=1}^{\zeta_i-1} \frac{a_{i,l}}{2}} S_i(m_{i,\zeta_i}+1,n_{i}+1) ,
\end{align*}
where $S_i(n_{\zeta_i},m_{i,\zeta_i})$ is the sum introduced in (\ref{Sisum}). By Lemma \ref{sm}, we have
\begin{align*}
&S_i^{(0,0,\dots,0)}(m_{i,1}, m_{i,2}, \dots, m_{i,\zeta_i-1})
=(-1)^{\sum_{l=1}^{\zeta_i-1} a_{i,l}} q^{-\sum_{l=1}^{\zeta_i-1} \frac{a_{i,l}}{2}}  \frac{(-1)^{\zeta_i + 1}(\sqrt{-1}r)^{\frac{\zeta_i-1}{2}}}{\sqrt{q_i}} \\
&\qquad \qquad\sum_{s_i=0}^{|q_i|-1}\Bigg(e^{\frac{r}{4\pi\sqrt{-1}} Z_i^+ \Big(s_i, \frac{2\pi m_{i,\zeta_i}}{r} + \frac{2\pi}{r}, \frac{2\pi n_i}{r} + \frac{2\pi}{r}\Big)} - e^{\frac{r}{4\pi\sqrt{-1}} Z_i^- \Big(s_i, \frac{2\pi m_{i,\zeta_i}}{r} + \frac{2\pi}{r}, \frac{2\pi n_i}{r} + \frac{2\pi}{r}\Big)}\Bigg),
\end{align*}
where
\begin{align}\label{Zi}
Z_i^\pm(s,\alpha,\beta)
&= \frac{C_{i,\zeta_i-1}}{q_i}(\alpha-\pi)^2 \mp \frac{2(\beta-\pi)(\alpha-\pi)}{q_i} - \frac{2\pi (I_i(s) \pm 1)}{q_i}(\alpha - \pi) \\
&\qquad + K_i(s) \pi^2 - \frac{p_i'}{q_i} (\beta-\pi)^2 + 2\pi \beta\Big(-\frac{p_i'}{q_i} \mp J_i(s) \Big) + \frac{p_i'}{q_i}\pi^2 \notag.
\end{align}

By a direct computation,
\begin{align*}
&Z_i^\pm\Big(s,\alpha+\frac{2\pi}{r},\beta+\frac{2\pi}{r}\Big)\notag\\
=&Z_i^\pm(s,\alpha,\beta) + \frac{4\pi C_{i,\zeta_i-1}}{rq_i}(\alpha-\pi) + \frac{4\pi^2 C_{i,\zeta_i-1}}{q_i r^2} \mp \frac{4\pi}{rq_i}(\alpha+\beta-2\pi) \mp \frac{8\pi^2}{r^2 q_i} \\
&- \frac{4\pi^2(I_i(s)\pm 1)}{rq_i} - \frac{p_i'}{q_i}\Big(\frac{4\pi}{r} (\beta-\pi) + \frac{4\pi^2}{r^2}\Big) + \frac{4\pi^2}{r}\Big(-\frac{p_i'}{q_i} \mp J_i(s)\Big),
\end{align*}
which implies that
\begin{align}\label{Z2pior}
&e^{\frac{r}{4\pi\sqrt{-1}}Z_i^\pm\Big(s_i,\frac{2\pi m_{i,\zeta_i}}{r} +\frac{2\pi}{r},\frac{2\pi n_i}{r} +\frac{2\pi}{r}\Big)}\notag\\
=& e^{\sqrt{-1}\Big(-\frac{C_{i,\zeta_i-1}}{q_i}\Big(\frac{2\pi m_{i,\zeta_i}}{r} - \pi\Big) + \frac{pi'}{q_i}(\frac{2\pi n_i}{r} -\pi) \pm \frac{\big(\frac{2\pi m_{i,\zeta_i}}{r}+\frac{2\pi n_i}{r} -2\pi\big)}{q_i} + \pi \Big(\frac{I_i(s_i) \pm 1}{q_i} + \frac{p_i'}{q_i} \pm J_i(s_i)\Big) \Big)}\\
&\times e^{\frac{r}{4\pi\sqrt{-1}} \Big(Z_i^\pm(s_i,\frac{2\pi m_{i,\zeta_i}}{r},\frac{2\pi n_i}{r} ) + \frac{4\pi^2}{r^2}\Big(\frac{C_{i,\zeta_i-1} \mp 2 - p_i'}{q_i}\Big)\Big)} \notag.
\end{align}

Next, by a direct computation, for any even integer $n \in \NN$, we have
\begin{align}
q^{\frac{n(n+2)}{2}} 
&= \lt(e^{\frac{\pi\sqrt{-1}}{4}}\rt)^{-r} q^{\frac{1}{2}\lt(n-\frac{r}{2}\rt)^2+n} 
= \lt(e^{\frac{\pi\sqrt{-1}}{4}}\rt)^{-r} e^{\sqrt{-1} \lt(\frac{2\pi n}{r}\rt)} e^{\frac{r}{4\pi\sqrt{-1}}\lt(-\lt( \frac{2\pi n}{r} - \pi \rt)^2\rt)} \label{c1} 
\end{align}
By using Equation (\ref{c1}), we can write
\begin{align}
q^{\frac{a_{j,0} m_j(m_j+2)}{2}} 
&= \lt(e^{-\frac{r\pi\sqrt{-1}}{4}  a_{j,0}}\rt) \lt(e^{ a_{j,0} \sqrt{-1}\lt(\frac{2\pi m_j}{r}\rt) } \rt)
\lt(e^{\frac{r}{4\pi \sqrt{-1}}\lt( - a_{j,0} \lt(\frac{2\pi m_j}{r} - \pi\rt)^2\rt)} \rt),
\end{align}
\begin{align}
q^{\frac{a_{i,0} n_i(n_i+2)}{2}}
&= \lt(e^{-\frac{r\pi\sqrt{-1}}{4} a_{i,0}} \rt) \lt(e^{a_{i,0}\sqrt{-1} \lt(\frac{2\pi n_i}{r}\rt) } \rt)
\lt(e^{\frac{r}{4\pi \sqrt{-1}}\lt( - a_{i,0} \lt(\frac{2\pi n_i}{r} - \pi\rt)^2\rt)}  \rt) 
\end{align}
and
\begin{align}\label{pqterm}
&q^{\frac{a_{i,\zeta_i} m_{i,\zeta_i}(m_{i,\zeta_i}+2)}{2}}\notag\\
&= \lt(e^{-\frac{r\pi\sqrt{-1}}{4}  a_{i,\zeta_i}}\rt) \lt(e^{ a_{i,\zeta_i}\sqrt{-1}\lt(\frac{2\pi m_{i,\zeta_i}}{r}\rt) } \rt)
\lt(e^{\frac{r}{4\pi \sqrt{-1}}\lt( -a_{i,\zeta_i} \lt(\frac{2\pi m_{i,\zeta_i}}{r} - \pi\rt)^2\rt)} \rt) \notag\\
&= (-1)^{ a_{i,\zeta_i}} \lt(e^{-\frac{r\pi\sqrt{-1}}{4} a_{i,\zeta_i}}\rt) \lt(e^{ a_{i,\zeta_i}\sqrt{-1}\lt(\frac{2\pi m_{i,\zeta_i}}{r} - \pi\rt) } \rt)
\lt(e^{\frac{r}{4\pi \sqrt{-1}}\lt( - a_{i,\zeta_i} \lt(\frac{2\pi m_{i,\zeta_i}}{r} - \pi\rt)^2\rt)} \rt).
\end{align}
In particular, by Lemma \ref{cf}, the first term in (\ref{Zi}) and the last term of (\ref{pqterm}) can be grouped together to give
\begin{align*}
\frac{C_{i,\zeta_i-1}}{q_i}\Big(\frac{2\pi m_{i,\zeta_i}}{r} -\pi\Big)^2 - a_{i,\zeta_i}\Big(\frac{2\pi m_{i,\zeta_i}}{r} -\pi\Big)^2
= -\frac{p_i}{q_i}\Big(\frac{2\pi m_{i,\zeta_i}}{r} -\pi\Big)^2.
\end{align*}
Besides, for each $i\in I$, (\ref{Z2pior}) and the third term in (\ref{pqterm}) can be grouped together to give 
\begin{align*}
&e^{\frac{r}{4\pi\sqrt{-1}}Z_i^\pm\Big(s_i,\frac{2\pi m_{i,\zeta_i}}{r} +\frac{2\pi}{r},\frac{2\pi n_i}{r} +\frac{2\pi}{r}\Big)}e^{ a_{i,\zeta_i}\sqrt{-1}\lt(\frac{2\pi m_{i,\zeta_i}}{r} - \pi\rt) }\notag\\
=& e^{\sqrt{-1}\Big(\frac{p_i'}{q_i}\Big(\frac{2\pi m_{i,\zeta_i}}{r} - \pi\Big) + \frac{pi'}{q_i}(\frac{2\pi n_i}{r} -\pi) \pm \frac{\big(\frac{2\pi m_{i,\zeta_i}}{r}+\frac{2\pi n_i}{r} -2\pi\big)}{q_i} + \pi \Big(\frac{I_i(s_i) \pm 1}{q_i} + \frac{p_i'}{q_i} \pm J_i(s_i)\Big) \Big)}\\
&\times e^{\frac{r}{4\pi\sqrt{-1}} \Big(Z_i^\pm(s_i,\frac{2\pi m_{i,\zeta_i}}{r},\frac{2\pi n_i}{r} ) + \frac{4\pi^2}{r^2}\Big(\frac{C_{i,\zeta_i-1} \mp 2 - p_i'}{q_i}\Big)\Big)}.
\end{align*}
For any even integer $n\in \NN$, we have
\begin{align}
(-1)^{\frac{n}{2}}q^{\frac{n(n+2)}{4}} 
= \lt(e^{\frac{\pi\sqrt{-1}}{8}}\rt)^{-r} q^{\frac{1}{4}\lt(n-\frac{r}{2}\rt)^2+\frac{n}{2}} 
= e^{\frac{r}{4\pi \sqrt{-1}} \lt(\frac{\pi^2}{2}\rt)} e^{\frac{\sqrt{-1}}{2} \lt(\frac{2\pi n}{r}\rt)} e^{\frac{r}{4\pi\sqrt{-1}}\lt(- \frac{1}{2}\lt( \frac{2\pi n}{r} - \pi \rt)^2\rt)} . \label{c1.5} 
\end{align}
By using Equation (\ref{c1.5}), we can write
\begin{align}
&(-1)^{\frac{\iota_j m_j}{2}}q^{\frac{\iota_j m_j(m_j+2)}{4}} = \lt(e^{\frac{r}{4\pi \sqrt{-1}} \lt(\frac{\iota_j\pi^2}{2}\rt)}\rt) \lt( e^{\frac{\iota_j\sqrt{-1}}{2} \lt(\frac{2\pi m_j}{r}\rt)} \rt) \lt(e^{\frac{r}{4\pi\sqrt{-1}}\lt(-\frac{\iota_j}{2}\lt( \frac{2\pi m_j}{r} - \pi \rt)^2\rt)} \rt)
\end{align}
and
\begin{align}
& (-1)^{\frac{\iota_i m_{i,\zeta_i}}{2}}q^{\frac{\iota_i m_{i,\zeta_i}(m_{i,\zeta_i}+2)}{4}} 
= \lt(e^{\frac{r}{4\pi \sqrt{-1}} \lt(\frac{\iota_i\pi^2}{2}\rt)} \rt)\lt(e^{\frac{\iota_i\sqrt{-1}}{2} \lt(\frac{2\pi m_{i,\zeta_i}}{r}\rt)}\rt) \lt(e^{\frac{r}{4\pi\sqrt{-1}}\lt(- \frac{\iota_i}{2}\lt( \frac{2\pi m_{i,\zeta_i}}{r} - \pi \rt)^2\rt)}\rt).
\end{align}

By Proposition \ref{6jqd}, we have
\begin{align}
\prod_{s=1}^c
\begin{vmatrix}
        m_{s_1} & m_{s_2} & m_{s_3} \\
        m_{s_4} & m_{s_5} & m_{s_6} 
      \end{vmatrix} 
=\frac{\{1\}^c}{2^c}\sum_{\mathbf k}e^{\frac{r}{4\pi\sqrt{-1}} \sum_{s=1}^c U_r\big(\frac{2\pi m_{s_1}}{r},\dots,\frac{2\pi m_{s_6}}{r},\frac{2\pi k_s}{r}\big)},
\end{align}
where
$\mathbf k=(k_1,\dots,k_c)$ runs over all multi-integers with each $k_s$ lying in between $\max\{T_{s_i}\}$ and $\min\{Q_{s_j},r-2\}$. The result then follows from above computations.

\end{proof}

\section{Poisson summation formula}
To apply the Poisson Summation Formula to the summation in Proposition 
\ref{computation}, we consider the following regions and a bump function over them.

For a fixed $(\boldsymbol \alpha_j)_{j \in J} \in [0, 2\pi]^{|J|}$, we let $\boldsymbol{\alpha}_{\zeta_I} = (\alpha_{i,\zeta_i})_{i\in I} \in [0,2\pi]^{|I|}$, $\boldsymbol\xi = (\xi_1,\dots, \xi_c) \in \RR^c$,
$$D_A = \{ ( \boldsymbol{\alpha}_{\zeta_I}, \boldsymbol\xi) \in \mathbb{R}^{|I|  + c} \mid (\alpha_{s_1}, \dots , \alpha_{s_6}) \text{ is admissible, }
\max\{\tau_{s_i}\} \leq  \xi_s \leq \min \{ \eta_{s_j},2\pi \} \},$$
and
$$D_H = \{ (\boldsymbol{\alpha}_{\zeta_I}, \boldsymbol\xi) \in D_A \mid (\alpha_{s_1}, \dots , \alpha_{s_6})  \text{ is of hyperideal type}\}.$$
For sufficiently small $\delta>0$, we let
$$D_H^{\delta} = \{ (\boldsymbol{\alpha}_{\zeta_I}, \boldsymbol\xi) \in D_H \mid d((\boldsymbol{\alpha}_{\zeta_I}, \boldsymbol\xi), \partial D_H) > \delta \},$$
where $d$ is the Euclidean distance on $\mathbb{R}^{|I| + c}$. 
Moreover, we let $\psi: \mathbb{R}^{|I| + c} \rightarrow [0,1]$ be a $C^{\infty}$-smooth bump function satisfying
\begin{align*}
 \psi(\boldsymbol{\alpha}_{\zeta_I}, \boldsymbol\xi) 
 &= 1 \quad\text{ for } \quad (\boldsymbol{\alpha}_{\zeta_I}, \boldsymbol\xi) \in  \overline{D_H^{\delta}},\\
    \psi(\boldsymbol{\alpha}_{\zeta_I}, \boldsymbol\xi) 
    &= 0 \quad\text{ for } \quad   (\boldsymbol{\alpha}_{\zeta_I}, \boldsymbol\xi) \not \in D_H
\end{align*}
and $\psi \in (0,1)$ elsewhere.

Moreover, we let
\begin{equation*}
    f_r^{\mathbf E_I} (\mathbf s_I, {\mathbf m_{\zeta_I}},{\mathbf k}) = \psi\lt( \frac{2\pi \mathbf{m}_{\zeta_I}}{r}, \frac{2\pi \mathbf{k} }{r}\rt) g_r^{\mathbf E_I}(\mathbf s_I, \mathbf m_{\zeta_I}, \mathbf k).
\end{equation*}

To apply the Poisson summation formula, for each $i \in I$, we let $m_{i,\zeta_i} = 2 {\widehat m_{i,\zeta_i}}$ and ${\mathbf {\widehat m}_{\zeta_I}} = (\mathbf{\widehat m}_{i,\zeta_i})_{i\in I}$. Then from Proposition \ref{computation}, we have
\begin{align}
    RT_r (M,L,(\textbf{n}_I, \textbf{m}_J)) = Z_r  \sum_{(\mathbf {\widehat m}_{\zeta_I}, \mathbf k) \in \ZZ^{|I|+ c}} \lt(\sum_{\mathbf E_I,\mathbf s_I} f_r^{\boldsymbol {E}_I}(\mathbf s_I, 2 \widehat{\mathbf m}_{\zeta_I}, \mathbf k) \rt) + \text{error term}. \label{PSerr1}
\end{align}
Let
\begin{equation*}
    f_r (2 \widehat{\mathbf m}_{\zeta_I}, \mathbf k) = \sum_{\mathbf E_I, \mathbf s_I} f_r^{\mathbf E_I} (\mathbf s_I, 2 \widehat{\mathbf m}_{\zeta_I}, \mathbf k) .
\end{equation*}

Since $f_r$ is in the Schwartz space on $\mathbb{R}^{|I| + c}$, by the Poisson summation formula (see e.g. \cite[Theorem 3.1]{SS}),
\begin{equation*}
    \sum_{(\mathbf {\widehat m}_{\zeta_I},\mathbf k) \in \ZZ^{|I| + c}} f_r (2 \widehat{\mathbf m}_{\zeta_I}, \mathbf k) 
    = \sum_{(\mathbf A_{\zeta_I}, \mathbf B)\in \ZZ^{|I|+c}} \widehat{f_r}(\mathbf A_{\zeta_I}, \mathbf B),
\end{equation*}
where $\mathbf A_{\zeta_I} = (A_{i,\zeta_i})_{i\in I}\in \ZZ^{|I|}$, $\mathbf B = (B_1,\dots, B_c) \in \ZZ^c$ and $\widehat{f_r}(\mathbf A_{\zeta_I}, \mathbf B)$ is the $(\mathbf A_{\zeta_I}, \mathbf B)$-th Fourier coefficient of $f_r$ defined by 
\begin{align*}
 \widehat{f_r}(\mathbf A_{\zeta_I},  \mathbf B) 
= \int_{\mathbb{R}^{|I|+c}} f_r(2 \widehat{\mathbf m}_{\zeta_I}, \mathbf k) e^{\sum_{i \in I} 2\pi \sqrt{-1}A_{i,\zeta_i} \widehat{m}_{i,\zeta_i} + \sum_{s=1}^c 2\pi \sqrt{-1} B_l k_l} d{\mathbf{\widehat m}_{\zeta_I}}d{\textbf{k}}
\end{align*}
where $d{\mathbf{\widehat m}_{\zeta_I}} d \textbf{k} = \prod_{i\in I} d{\widehat m}_{i,\zeta_i} \prod_{s=1}^c dk_l$. 

By change of variables $ \widehat{m}_{i,\zeta_i}=\frac{r}{4\pi}\alpha_{i,\zeta_i}$ and $k_l = \frac{r}{2\pi}\xi_s$, the Fourier coefficients can be computed as follows.

\begin{proposition}\label{formulaFC}
\begin{equation*}
    \widehat{f_r}(\mathbf A_{\zeta_I}, \mathbf  B) = \sum_{\mathbf E_I, \mathbf s_I}  \widehat{f_r^{\mathbf E_I}}(\mathbf s_I, \mathbf A_{\zeta_I},  \boldsymbol  B)
\end{equation*}
with
\begin{align*}
   & \widehat{f_r^{\mathbf E_I}}(\mathbf s_I, \mathbf A_{\zeta_I},  \boldsymbol  B)  =
    \frac{r^{|I| + c}\big(\prod_{i \in I} E_i\big) }{2^{2|I| + c} \pi ^{|I| + c}}\\
&  \times  \int_{D_H} 
   (-1)^{\sum_{i\in I} A_{i,\zeta_i}} \phi_r\lt(\mathbf s_I, \boldsymbol{\alpha}_{\zeta_I}, \boldsymbol\xi \rt) 
    e^{\frac{r}{4\pi \sqrt{-1}}\lt( 
   W_r^{\mathbf E_I}\lt(\mathbf s_I, \boldsymbol \alpha_{\zeta_I}, \boldsymbol\xi\rt) 
    - \sum_{i \in I} 2\pi A_{i,\zeta_i} \alpha_{i,\zeta_i} - \sum_{s=1}^c 4\pi  B_s \xi_s \rt)} d\boldsymbol{\alpha}_{\zeta_I} d\boldsymbol\xi,
\end{align*}
where $d\boldsymbol{\alpha}_{\zeta_I} d\boldsymbol \xi = \prod_{i \in I}d\alpha_{i,\zeta_i} \prod_{s=1}^c d\xi_s$, 
$$
\phi_r(\mathbf s_I, \boldsymbol{\alpha}_{\zeta_I}, \boldsymbol\xi) 
= 
\psi(\boldsymbol{\alpha}_{\zeta_I}, \boldsymbol\xi) 
e^{ \sqrt{-1}P_r^{\mathbf E_I}(\mathbf s_I, \boldsymbol{\alpha_{\xi_I}})},
$$
\begin{align*}
P_r^{\mathbf E_I}(\mathbf s_I, \boldsymbol{\alpha}_{\zeta_I}) 
=& \sum_{i\in I}\Big(\frac{p_i'}{q_i}(\beta_i - \pi) + \frac{p_i}{q_i}(\alpha_{i,\zeta_i} - \pi) + \frac{E_i(\alpha_{i,\zeta_i}+\beta_i-2\pi)}{q_i}\Big)\\
&+ \pi\sum_{i\in I}\Big( \frac{I_i(s_i)+ E_i}{q_i} + \frac{p_i'}{q_i} + E_i J_i(s_i)\Big)\\
&+ \sum_{i\in I}a_{i,0}\beta_i + \sum_{i\in I} \Big(\frac{\iota_i}{2}\Big)\alpha_{i,\zeta_i} + \sum_{j\in J}\Big(a_{j,0}+\frac{\iota_j}{2}\Big)\alpha_j
\end{align*}
and 
\begin{align*}
W_r^{\mathbf E_I}(\mathbf s_I, \boldsymbol{\alpha}_{\zeta_I}, \boldsymbol \xi)  
&=  - \sum_{i \in I}  \Big(a_{i,0} + \frac{p_i'}{q_i}\Big)(\beta_i - \pi)^2 - \sum_{j \in J} \Big(a_{j,0}+\frac{\iota_j}{2}\Big)(\alpha_j - \pi)^2  \\
&\qquad -\sum_{i\in I} \Big(\frac{p_i}{q_i}+\frac{\iota_i}{2}\Big)(\alpha_{i,\zeta_i}-\pi)^2 - \sum_{i\in I}\frac{2\pi(I_i(s_i) + E_i)}{q_i}(\alpha_{\xi_i}-\pi)\\
&\\ 
&\qquad - \sum_{i\in I}\frac{2E_i (\alpha_{i,\zeta_i}-\pi)(\beta_i-\pi)}{q_i}
+ \sum_{i\in I}2\pi\beta_i\Big(-E_iJ_i(s_i) - \frac{p_i}{q_i} \Big)\\
&\qquad +\sum_{s=1}^c U_r(\alpha_{s_1},\dots,\alpha_{s_6},\xi_s)
+\sum_{i\in I}\pi^2\Bigg(K_i(s_i)+\frac{p_i'}{q_i}\Bigg)
+\Big(\sum_{i=1}^n\frac{\iota_i}{2}\Big)\pi^2 \\
&\qquad +\frac{4\pi^2}{r^2} h_I
\end{align*}
with $\beta_i = \frac{2\pi n_i}{r}$ for $i\in I$, $\alpha_j = \frac{2\pi m_j}{r}$ for $j\in J$ and $h_I = \sum_{i \in I} \frac{C_{i,\xi_i-1} -2E_i - p_i' }{q_i}$.
\end{proposition}

Together with Equation (\ref{PSerr1}), we have
\begin{proposition}\label{Fperror}
\begin{align*}
    RT_r(M,L,(\textbf{n}_I, \textbf{m}_J)) = Z_r \sum_{(\mathbf A_{\zeta_I}, \mathbf B)\in \ZZ^{|I|+c}}
    \widehat{f_r}(\mathbf A_{\zeta_I}, \boldsymbol  B) + \text{error term}.
\end{align*}
\end{proposition}

The error term in Proposition \ref{Fperror} will be estismated in Proposition \ref{errorest} of Section \ref{Sasymp}.

For each $i\in I$, with respect to the continued fraction 
$$\frac{p_i}{q_i}
=[a_{i,1},\dots,a_{i,\zeta_i}]
=a_{i,\zeta_i} - \frac{1}{a_{i,\zeta_i-1}- \frac{1}{\dots - \frac{1}{a_{i,1}}}},$$
let $s_i^\pm$ and $m_i^\pm$ be the integers $s^\pm$ and $m^\pm$ defined in Lemma \ref{arith} (1). For each multi-sign $\mathbf E_I=(E_i)_{i\in I}\in \{-1,1\}^{|I|}$, define $\mathbf s^{\mathbf E_I}=(s^{E_I}_i)_{i\in I}\in \ZZ^{|I|}$ and $\mathbf m^{\mathbf E_I}=(m^{E_I}_i)_{i\in I}\in \ZZ^{|I|}$ by
$$
s^{E_I}_i =
\begin{cases}
s_i^+ & \text{ if $E_i = -1$}\\
s_i^- & \text{ if $E_i = 1$}
\end{cases}
$$
and
$$
m^{E_I}_i =
\begin{cases}
m_i^+ & \text{ if $E_i = -1$}\\
m_i^- & \text{ if $E_i = 1$}
\end{cases}.
$$
In particular, by Lemma \ref{arith} (1), (2) and definitions of $s_i^{E_I}$ and $m_i^{E_I}$, we have 
\begin{align}\label{IJsi}
I_i(s_i^{E_I}) = -E_i - q_i + 2m_i^{E_I} q_i\quad\text{ and }\quad E_i J_i(s_i^{E_I}) \equiv -\frac{p_i}{q_i} \pmod{\ZZ}.
\end{align}

Let $\mathbf{1 - 2m^{E_I}} = (1-2m^{E_i}_i)_{i\in I} \in \ZZ^{|I|}$. In Section \ref{Sasymp}, we will show that $\widehat{f_r^{\mathbf E_I}}(\mathbf{s}^{\mathbf E_I},\mathbf{1-2m^{E_I}}, \mathbf{0})$ are the leading Fourier coefficients. The following proposition gives a simplified expression for the Fourier coefficients that will be used later.

\begin{proposition}\label{lfcexpress} We have
\begin{align*}
\widehat{f_r^{\mathbf E_I}}(\mathbf{s}^{\mathbf E_I}, \mathbf{1-2m^{E_I}}, \mathbf{0})&=\frac{Y(\mathbf E_I)r^{|I| + c} }{2^{|I| + c} \pi ^{|I| + c}} \int_{D_H} 
    \phi_r\lt(\mathbf{s}^{\mathbf E_I}, \boldsymbol{\alpha}_{\zeta_I}, \boldsymbol\xi \rt)     
   e^{\frac{r}{4\pi \sqrt{-1}}G_r^{E_I}(\boldsymbol\alpha_{\xi_I},\boldsymbol\zeta)} d\boldsymbol{\alpha}_{\zeta_I} d\boldsymbol\xi,
\end{align*}
where
\begin{align}\label{defYEI}
Y(\mathbf E_I) =
- (-1)^{\sum_{i\in I} \Big(\frac{p_i'}{q_i} + E_i J_i(s_i^{E_I})\Big) + |I|}\Bigg(\prod_{i \in I} E_i\Bigg)e^{\frac{r\pi}{4\sqrt{-1}}\sum_{i\in I}\Big(4m_i^{E_I}-2+ K_i(s_i^{E_I})+\frac{p_i'}{q_i} \Big)},
\end{align}
\begin{align*}
&\phi_r(\mathbf s_I, \boldsymbol{\alpha}_{\zeta_I}, \boldsymbol\xi)=
\psi(\boldsymbol{\alpha}_{\zeta_I}, \boldsymbol\xi)  \\
&\times
e^{ \sqrt{-1}\Big(\sum_{i\in I}\Big(\frac{p_i'}{q_i}(\beta_i - \pi) + \frac{p_i}{q_i}(\alpha_{i,\zeta_i} - \pi) + \frac{E_i(\alpha_{i,\zeta_i}+\beta_i-2\pi)}{q_i}\Big)+ \sum_{i\in I}a_{i,0}\beta_i + \sum_{i\in I} \Big(\frac{\iota_i}{2}\Big)\alpha_{i,\zeta_i} + \sum_{j\in J}\Big(a_{j,0}+\frac{\iota_j}{2}\Big)\alpha_j \Big)}
\end{align*}
and
\begin{align}\label{defG_r}
G_r^{\boldsymbol  E_I}(\boldsymbol \alpha_{\zeta_I}, \boldsymbol \xi) 
&=  \sum_{i\in I} 
\lt[- \lt(a_{i,0} + \frac{p_i'}{q_i}\rt) (\beta_i - \pi)^2 - \frac{p_i (\alpha_{i,\zeta_i} - \pi)^2 + 2 E_i(\beta_i - \pi)(\alpha_{i,\zeta_i} - \pi)}{q_i}
\rt] \notag
\\
&\qquad - \sum_{j \in J}\Big( a_{j,0} + \frac{\iota_j}{2}\Big)(\alpha_j - \pi)^2  -\sum_{i\in I}\frac{\iota_i}{2}(\alpha_{i,\zeta_i}-\pi)^2\notag\\
&\qquad +\sum_{s=1}^c U_r(\alpha_{s_1},\dots,\alpha_{s_6},\xi_s)+\Big(\sum_{i=1}^n\frac{\iota_i}{2}\Big)\pi^2 +\frac{4\pi^2}{r^2} h_I.
\end{align}
\end{proposition}
\begin{proof}
By definition of the Fourier coefficient and the bump function $\phi$, we can write
\begin{align*}
&\widehat{f_r^{\mathbf E_I}}(\mathbf{s}^{\mathbf E_I}, \mathbf{1-2m^{E_I}}, \mathbf{0}) = \frac{r^{|I| + c}}{2^{|I| + c} \pi ^{|I| + c}} \\
& \times \int_{D_H} 
   (-1)^{\sum_{i\in I} (1-2m_i^{E_i})} \phi_r\lt(\mathbf{s}^{\mathbf E_I}, \boldsymbol{\alpha}_{\zeta_I},\boldsymbol\xi \rt)     
   e^{\frac{r}{4\pi \sqrt{-1}}\Big(W_r^{\mathbf E_I}\lt(\mathbf{s}^{\mathbf E_I}, \boldsymbol \alpha_{\zeta_I}, \boldsymbol\xi\rt) 
    - \sum_{i \in I} 2\pi (1-2m^{E_I}_i)\alpha_{i,\zeta_i}\Big)} d\boldsymbol{\alpha}_{\zeta_I} d\boldsymbol\xi.
\end{align*}
By (\ref{IJsi}), 
\begin{align*}
e^{\sqrt{-1} \pi\sum_{i\in I}\Big( \frac{I_i(s_i)+ E_i}{q_i} + \frac{p_i'}{q_i} + E_i J_i(s_i)\Big)}
= - (-1)^{\sum_{i\in I} \Big(\frac{p_i'}{q_i} + E_i J_i(s_i)\Big) }.
\end{align*}
Besides, by (\ref{IJsi}), we can write
\begin{align*}
&W_r^{\mathbf E_I}\lt(\mathbf{s}^{\mathbf E_I}, \boldsymbol \alpha_{\zeta_I}, \boldsymbol\xi\rt) 
    - \sum_{i \in I} 2\pi (1-2m^{E_I}_i)\alpha_{i,\zeta_i}\\
=&  G_r^{\boldsymbol  E_I}(\boldsymbol \alpha_{\zeta_I}, \boldsymbol \xi) - 2\pi^2\sum_{i\in I}(1-2m_i^{E_i})
 + \sum_{i\in I}2\pi\beta_i\Big(-E_iJ_i(s_i^{E_I}) - \frac{p_i}{q_i} \Big) \\
&\qquad  +\sum_{i\in I}\pi^2\Bigg(K_i(s_i^{E_I})+\frac{p_i'}{q_i}\Bigg) +\frac{4\pi^2}{r^2} h_I .
\end{align*}
Furthermore, since $\beta_i = \frac{2\pi n_i}{r}$ and $n_i$ is even for all $i\in I$, by (\ref{IJsi}),
\begin{align*}
\frac{r}{4\pi \sqrt{-1}} \Big(2\pi\beta_i\Big(-E_iJ_i(s_i^{E_I}) - \frac{p_i}{q_i} \Big)\Big)
\in 2\pi\sqrt{-1}\ZZ \quad\text{and}\quad
e^{\frac{r}{4\pi \sqrt{-1}} \sum_{i\in I}2\pi\beta_i\Big(-E_iJ_i(s_i^{E_I}) - \frac{p_i}{q_i} \Big)}=1.
\end{align*}
The result follows from a direct computation.
\end{proof}

Define 
\begin{align}\label{defY}
Y=
- (-1)^{\sum_{i\in I} \Big(\frac{p_i'}{q_i} - J_i(s_i^+)\Big)}e^{\frac{r\pi}{4\sqrt{-1}}\sum_{i\in I}\Big(4m_i^{+}-2 + K_i(s_i^{+})+\frac{p_i'}{q_i} \Big)}.
\end{align}
The following lemma ensures that the leading Fourier coefficients in Proposition \ref{lfcexpress} do not cancel out with each other.
\begin{lemma}\label{YEIY} For any $\mathbf E_I\in\{1,-1\}^{|I|}$, we have
$Y(\mathbf E_I) = Y$.
\end{lemma}
\begin{proof}
Note that $Y$ is equal to $Y(\mathbf E_I)$ with $E_i = -1$ for all $i\in I$. We claim that
$$
(-1)^{\sum_{i\in I} \Big(\frac{p_i'}{q_i} - J_i(s_i^+)\Big)}e^{\frac{r\pi}{4\sqrt{-1}}\Big(4m_i^+ - 2 + K_i(s^+) + \frac{p_i'}{q_i}\Big)}
= - (-1)^{\sum_{i\in I} \Big(\frac{p_i'}{q_i} + J_i(s_i^-)\Big)}e^{\frac{r\pi}{4\sqrt{-1}}\Big(4m_i^- - 2 + K_i(s^-) + \frac{p_i'}{q_i}\Big)}
$$
for any $i\in I$. This shows that $Y(\mathbf E_I)$ is invariant when we change $E_i$ to $-E_i$ for any $i\in I$. By changing all $E_i$ to $-1$, we get the desired result.

To prove the claim, first, from Lemma \ref{arith} (2), since 
$$J_i(s_i^+)\equiv  -J_i(s_i^-)\quad(\text{mod }2\mathbb Z),$$
we have
$$
(-1)^{\sum_{i\in I} \Big(\frac{p_i'}{q_i} - J_i(s_i^+)\Big)} 
= 
(-1)^{\sum_{i\in I} \Big(\frac{p_i'}{q_i} + J_i(s_i^-)\Big)}.
$$
Moreover, from the definition of $K$ in Lemma \ref{arith} (3), we get
\begin{align}\label{K-K}
K_i(s_i^+) - K_i(s_i^-) + 4(m_i^+ - m_i^-)
= \frac{4C_{i,\zeta_i-1}}{q_i}(s_i^+ + s_i^-+1+K_{i,\zeta_i-1})(s_i^+-s_i^-) + 4(m_i^+ - m_i^-).
\end{align}
Besides, from the definition of $I$ and Lemma \ref{arith} (1),
\begin{align}\label{I+I}
I_i(s_i^+) + I_i(s_i^-)
= -2 C_{i,\zeta_i-1}(s_i^+ + s_i^- + 1 + K_{i,\zeta_i-1})
= 2q_i(m_i^+ + m_i^- -1 ).
\end{align}
From (\ref{K-K}) and (\ref{I+I}), we have
$$
K_i(s_i^+) - K_i(s_i^-) + 4(m_i^+ - m_i^-)
= 4((1 - m_i^+ - m_i^-)(s_i^+-s_i^-)+m_i^+ + m_i^-).
$$
In particular,
\begin{align*}
\frac{e^{\frac{r\pi}{4\sqrt{-1}}\Big(4m_i^+ - 2 + K_i(s^+) + \frac{p_i'}{q_i}\Big)}}{- e^{\frac{r\pi}{4\sqrt{-1}}\Big(4m_i^- - 2 + K_i(s^-) + \frac{p_i'}{q_i}\Big)}}
&= - e^{\frac{r\pi}{4\sqrt{-1}}(K_i(s_i^+) - K_i(s_i^-) + 4(m_i^+ - m_i^-))} \\
&= -e^{-\pi\sqrt{-1}((1 - m_i^+ - m_i^-)(s_i^++s_i^-)+m_i^+ + m_i^-)}\\
&= -(-1)^{(m_i^+ - m_i^-)(s_i^++s_i^-+1)+(s_i^++s_i^-)}.
\end{align*}
We claim that the integer $(m_i^+ - m_i^-)(s_i^++s_i^-+1)+(s_i^++s_i^-)$ is always odd. Note that by Lemma \ref{arith} (1) and the definition of $I$, we have
$$
-2C_{i,\zeta_i-1}(s_i^+ - s_i^-) = I(s^+) - I(s^-) = 2(m^+ - m^-)q_i + 2,
$$
which implies that
$$
(m^+ - m^-)q_i + C_{i,\zeta_i-1}(s_i^+ - s_i^-) = -1.
$$
In particular, at least one of $(m^+ - m^-)$ and $(s_i^+ - s_i^-)$ must be odd. Note that if $(s_i^+ - s_i^-)$ is even, then $(m^+ - m^-)$ is odd. In particular, $(m_i^+ - m_i^-)(s_i^++s_i^-+1)+(s_i^++s_i^-)$ is odd. If $(s_i^+ - s_i^-)$ is odd, then $(s_i^++s_i^-)$ is odd and $(s_i^++s_i^-+1)$ is even. In particular, $(m_i^+ - m_i^-)(s_i^++s_i^-+1)+(s_i^++s_i^-)$ is odd.

Altogether, 
\begin{align*}
\frac{e^{\frac{r\pi}{4\sqrt{-1}}\Big(4m_i^+ - 2 + K_i(s^+) + \frac{p_i'}{q_i}\Big)}}{- e^{\frac{r\pi}{4\sqrt{-1}}\Big(4m_i^- - 2 + K_i(s^-) + \frac{p_i'}{q_i}\Big)}}
&= -(-1)^{(m_i^+ - m_i^-)(s_i^++s_i^-+1)+(s_i^++s_i^-)} = 1.
\end{align*}
This completes the proof.
\end{proof}

For $\mathbf z = (z_1,\dots,z_n) \in \CC^n$, we write $\Re (\mathbf z) = (\Re (z_1),\dots, \Re (z_n))$, where $\Re z_i$ is the real part of $z_i$ for $i=1,\dots,n$. Let
$$D_{H,\CC} = \{ (\boldsymbol{\alpha}_{\zeta_I}, \boldsymbol\xi) \in \CC^{|I|+c} \mid (\Re(\boldsymbol{\alpha}_{\zeta_I}), \Re(\boldsymbol\xi)) \in D_H\}.$$
To end this section, we consider a closely related function $G^{\boldsymbol  E_I}(\boldsymbol \alpha_{\zeta_I}, \boldsymbol \xi): D_{H,\CC}\to \CC$ given by
\begin{align}\label{defG}
G^{\boldsymbol  E_I}(\boldsymbol \alpha_{\zeta_I}, \boldsymbol \xi) 
&=  \sum_{i\in I} 
\lt[- \lt(\frac{p_i'}{q_i} + a_{i,0}\rt) (\beta_i - \pi)^2 - \frac{p_i (\alpha_{i,\zeta_i} - \pi)^2 +   2 E_i(\beta_i - \pi)(\alpha_{i,\zeta_i} - \pi)}{q_i}
\rt] \notag
\\
&\qquad - \sum_{j \in J} a_{j,0}(\alpha_j - \pi)^2  -\sum_{i=1}^n\frac{\iota_i}{2}(\alpha_i-\pi)^2+\sum_{s=1}^c U(\alpha_{s_1},\dots,\alpha_{s_6},\xi_s)+\Big(\sum_{i=1}^n\frac{\iota_i}{2}\Big)\pi^2,
\end{align}
where $U$ is defined by
\begin{equation}\label{defU}
\begin{split}
U(\alpha_1,\dots,\alpha_6,\xi)=&\pi^2+\frac{1}{2}\sum_{i=1}^4\sum_{j=1}^3(\eta_j-\tau_i)^2-\frac{1}{2}\sum_{i=1}^4(\tau_i-\pi)^2\\
&+(\xi-\pi)^2-\sum_{i=1}^4(\xi-\tau_i)^2-\sum_{j=1}^3(\eta_j-\xi)^2\\
&-2\Li(1)-\frac{1}{2}\sum_{i=1}^4\sum_{j=1}^3\Li\big(e^{2i(\eta_j-\tau_i)}\big)+\frac{1}{2}\sum_{i=1}^4\Li\big(e^{2i(\tau_i-\pi)}\big)\\
&-\Li\big(e^{2i(\xi-\pi)}\big)+\sum_{i=1}^4\Li\big(e^{2i(\xi-\tau_i)}\big)+\sum_{j=1}^3\Li\big(e^{2i(\eta_j-\xi)}\big).\\
\end{split}
\end{equation}

Note that when both $G^{\boldsymbol  E_I}(\boldsymbol \alpha_{\zeta_I}, \boldsymbol \xi) $ and $G_r^{\boldsymbol  E_I}(\boldsymbol \alpha_{\zeta_I}, \boldsymbol \xi) $ are defined, they are related by
$$ \lim_{r\to \infty} G_r^{\boldsymbol  E_I}(\boldsymbol \alpha_{\zeta_I}, \boldsymbol \xi) 
= G^{\boldsymbol  E_I}(\boldsymbol \alpha_{\zeta_I}, \boldsymbol \xi)  .$$
More preciesly, by Lemma \ref{converge}, the differenece between $G_r^{\boldsymbol  E_I}(\boldsymbol \alpha_{\zeta_I}, \boldsymbol \xi)  $ and $G^{\boldsymbol  E_I}(\boldsymbol \alpha_{\zeta_I}, \boldsymbol \xi) $ is given by the following lemma.
\begin{lemma}\label{qdtod}
On any compact subset of $D_{H,\CC}$, we have
\begin{align}\label{kappadef}
G_r^{\boldsymbol  E_I}(\boldsymbol \alpha_{\zeta_I}, \boldsymbol \xi) 
&= G^{\boldsymbol  E_I}(\boldsymbol \alpha_{\zeta_I}, \boldsymbol \xi) 
- \frac{4c\pi\sqrt{-1}}{r} \log\lt(\frac{r}{2}\rt) + \frac{4\pi \sqrt{-1} \kappa(\boldsymbol \alpha_{\zeta_I}, \boldsymbol \xi) }{r} + \frac{\upsilon_r(\boldsymbol \alpha_{\zeta_I}, \boldsymbol \xi) }{r^2},
\end{align}
where
\begin{align*}
&\kappa(\boldsymbol \alpha_{\zeta_I}, \boldsymbol \xi) \\
=&\sum_{s=1}^c\Big(\frac{1}{2}\sum_{i=1}^4 \sqrt{-1}\tau_{s_i}- \sqrt{-1}\xi_s-\sqrt{-1}\pi-\frac{\sqrt{-1}\pi}{2}\\
&+\frac{1}{4}\sum_{i=1}^4\sum_{j=1}^3\log\big(1-e^{2\sqrt{-1}(\eta_{s_j}-\tau_{s_i})}\big)-\frac{3}{4}\sum_{i=1}^4\log\big(1-e^{2\sqrt{-1}(\tau_{s_i}-\pi)}\big)\\
&+\frac{3}{2}\log\big(1-e^{2\sqrt{-1}(\xi_s-\pi)}\big)-\frac{1}{2}\sum_{i=1}^4\log\big(1-e^{2\sqrt{-1}(\xi_s-\tau_{s_i})}\big)-\frac{1}{2}\sum_{j=1}^3\log\big(1-e^{2\sqrt{-1}(\eta_{s_j}-\xi_s)}\big)\Big)
\end{align*}
 and  $|\nu_r(\boldsymbol \alpha_{\zeta_I}, \boldsymbol \xi) |$ is bounded from above by a constant independent of $r.$

\end{lemma}
\begin{proof} Note that
\begin{align*}
G_r^{\boldsymbol  E_I}(\boldsymbol \alpha_{\zeta_I}, \boldsymbol \xi) 
- G^{\boldsymbol  E_I}(\boldsymbol \alpha_{\zeta_I}, \boldsymbol \xi) 
&=  \sum_{s=1}^c (U_r(\alpha_{s_1} ,\dots,\alpha_{s_6}, \xi_s) - 
U(\alpha_{s_1} ,\dots,\alpha_{s_6}, \xi_s)) + O\lt(\frac{1}{r^2}\rt).
\end{align*}
Thus, it suffices to study the difference between $U_r(\alpha_1,\dots, \alpha_6, \xi)$ and $U(\alpha_1,\dots, \alpha_6, \xi)$. 

By Lemma \ref{converge}(3), 
$$\varphi_r\Big(\frac{\pi}{r}\Big)=\mathrm{Li}_2(1)+\frac{2\pi\sqrt{-1}}{r}\log\Big(\frac{r}{2}\Big)-\frac{\pi^2}{r}+O\Big(\frac{1}{r^2}\Big).$$
Besides, by using Lemma \ref{converge}(1), we have
\begin{align*}
\varphi_r\Big(\eta_j-\tau_i+\frac{\pi}{r}\Big)
= \Li\lt( e^{2\sqrt{-1}(\eta_j - \tau_i +\frac{\pi}{r})}\rt) +
\frac{2\pi^2 e^{2\sqrt{-1}(\eta_j - \tau_i +\frac{\pi}{r})}}{3\lt(1-e^{2\sqrt{-1}(\eta_j - \tau_i +\frac{\pi}{r})}\rt)} \frac{1}{r^2} +  O\lt(\frac{1}{r^4}\rt) .
\end{align*}
In particular, on a given compact subset of $D_{H,\CC}$, by continuity, we have
\begin{align*}
\varphi_r\Big(\eta_j-\tau_i+\frac{\pi}{r}\Big)
= \Li\lt( e^{2\sqrt{-1}(\eta_j - \tau_i +\frac{\pi}{r})}\rt) + O\lt(\frac{1}{r^2}\rt).
\end{align*}
Next, by considering the Talyor series expansion of $\Li\lt( e^{2\sqrt{-1}(\eta_j - \tau_i + w)}\rt)$ at $w=0$, we have
\begin{align*}
\varphi_r\Big(\eta_j-\tau_i+\frac{\pi}{r}\Big)
= \Li\lt( e^{2\sqrt{-1}(\eta_j - \tau_i)}\rt) - 2\sqrt{-1} \log(1 - e^{2\sqrt{-1}(\eta_j - \tau_i)}) \lt(\frac{\pi}{r}\rt) + O\lt(\frac{1}{r^2}\rt).
\end{align*}
Similar computations show that
\begin{align*}
\varphi_r\Big(\tau_i-\pi+\frac{3\pi}{r}\Big)
&= \Li\Big(e^{2\sqrt{-1}(\tau_i-\pi)}\Big) - 2\sqrt{-1}\log(1-e^{2\sqrt{-1}(\tau_i-\pi)})\lt(\frac{3\pi}{r}\rt) + O\lt(\frac{1}{r^2}\rt),\\
\varphi_r\Big(\xi-\pi+\frac{3\pi}{r}\Big)
&= \Li\Big(e^{2\sqrt{-1}(\xi-\pi)}\Big) - 2\sqrt{-1}\log(1-e^{2\sqrt{-1}(\xi-\pi)})\lt(\frac{3\pi}{r}\rt) + O\lt(\frac{1}{r^2}\rt),\\
\varphi_r\Big(\xi-\tau_i+\frac{\pi}{r}\Big)
&= \Li\Big(e^{2\sqrt{-1}(\xi-\tau_i)}\Big) - 2\sqrt{-1}\log(1-e^{2\sqrt{-1}(\xi-\tau_i)})\lt(\frac{\pi}{r}\rt) + O\lt(\frac{1}{r^2}\rt),\\
\varphi_r\Big(\eta_j-\xi+\frac{\pi}{r}\Big)
&= \Li\Big(e^{2\sqrt{-1}(\eta_j-\xi)}\Big) - 2\sqrt{-1}\log(1-e^{2\sqrt{-1}(\eta_j-\xi)})\lt(\frac{\pi}{r}\rt) + O\lt(\frac{1}{r^2}\rt).
\end{align*}
Equation (\ref{kappadef}) then follows from a direct computation. 
\end{proof}

\section{Asymptotics of the invariants}\label{Sasymp}
In this section, we will find the asymptotics of the leading Fourier coefficients and estimate the other. 

\subsection{Preliminary}
\subsubsection{Saddle point approximation}
First, to obtain the asymptotic of the invariants, we recall the following proposition from \cite{WY1}.

\begin{proposition}\cite{WY1}\label{saddle}
Let $D_{\mathbf z}$ be a region in $\mathbb C^n$ and let $D_{\mathbf a}$ be a region in $\mathbb R^k.$ Let $f(\mathbf z,\mathbf a)$ and $g(\mathbf z,\mathbf a)$ be complex valued functions on $D_{\mathbf z}\times D_{\mathbf a}$  which are holomorphic in $\mathbf z$ and smooth in $\mathbf a.$ For each positive integer $r,$ let $f_r(\mathbf z,\mathbf a)$ be a complex valued function on $D_{\mathbf z}\times D_{\mathbf a}$ holomorphic in $\mathbf z$ and smooth in $\mathbf a.$
For a fixed $\mathbf a\in D_{\mathbf a},$ let $f^{\mathbf a},$ $g^{\mathbf a}$ and $f_r^{\mathbf a}$ be the holomorphic functions  on $D_{\mathbf z}$ defined by
$f^{\mathbf a}(\mathbf z)=f(\mathbf z,\mathbf a),$ $g^{\mathbf a}(\mathbf z)=g(\mathbf z,\mathbf a)$ and $f_r^{\mathbf a}(\mathbf z)=f_r(\mathbf z,\mathbf a).$ Suppose $\{\mathbf a_r\}$ is a convergent sequence in $D_{\mathbf a}$ with $\lim_r\mathbf a_r=\mathbf a_0,$ $f_r^{\mathbf a_r}$ is of the form
$$ f_r^{\mathbf a_r}(\mathbf z) = f^{\mathbf a_r}(\mathbf z) + \frac{\upsilon_r(\mathbf z,\mathbf a_r)}{r^2},$$
$\{S_r\}$ is a sequence of embedded real $n$-dimensional closed disks in $D_{\mathbf z}$ sharing the same boundary and converging to an embedded $n$-dimensional disk $S_0$, and $\mathbf c_r$ is a point on $S_r$ such that $\{\mathbf c_r\}$ is convergent  in $D_{\mathbf z}$ with $\lim_r\mathbf c_r=\mathbf c_0.$ If for each $r$
\begin{enumerate}[(1)]
\item $\mathbf c_r$ is a critical point of $f^{\mathbf a_r}$ in $D_{\mathbf z},$
\item $\mathrm{Re}f^{\mathbf a_r}(\mathbf c_r) > \mathrm{Re}f^{\mathbf a_r}(\mathbf z)$ for all $\mathbf z \in S_r\setminus \{\mathbf c_r\},$
\item the domain $\{\mathbf z\in D_{\mathbf z}\ |\ \mathrm{Re} f^{\mathbf a_r}(\mathbf z) < \mathrm{Re} f^{\mathbf a_r}(\mathbf c_r)\}$ deformation retracts to $S_r\setminus\{\mathbf c_r\},$
\item $|g^{\mathbf a_r}(\mathbf c_r)|$ is bounded from below by a positive constant independent of $r,$
\item $|\upsilon_r(\mathbf z, \mathbf a_r)|$ is bounded from above by a constant independent of $r$ on $D_{\mathbf z},$ and
\item  the Hessian matrix $\mathrm{Hess}(f^{\mathbf a_0})$ of $f^{\mathbf a_0}$ at $\mathbf c_0$ is non-singular,
\end{enumerate}
then
\begin{equation*}
\begin{split}
 \int_{S_r} g^{\mathbf a_r}(\mathbf z) e^{rf_r^{\mathbf a_r}(\mathbf z)} d\mathbf z= \Big(\frac{2\pi}{r}\Big)^{\frac{n}{2}}\frac{g^{\mathbf a_r}(\mathbf c_r)}{\sqrt{-\det\mathrm{Hess}(f^{\mathbf a_r})(\mathbf c_r)}} e^{rf^{\mathbf a_r}(\mathbf c_r)} \Big( 1 + O \Big( \frac{1}{r} \Big) \Big).
 \end{split}
 \end{equation*}
\end{proposition}

In Section 6.4, We will apply Proposition \ref{saddle} to obtain the asymptotic expansion formula for the leading Fourier coefficient (see Proposition \ref{lfc} for more details).

\subsubsection{Convexity and preliminary estimate}
Next, to show that conditions in Proposition \ref{saddle} are satisfied, we need the following result about the function $U$ defined in (\ref{defU}). Recall that the function $U(\alpha_1,\dots,\alpha_6,\xi)$ in (\ref{defU}) is given by
\begin{equation*}
\begin{split}
U(\alpha_1,\dots,\alpha_6,\xi)=&\pi^2+\frac{1}{2}\sum_{i=1}^4\sum_{j=1}^3(\eta_j-\tau_i)^2-\frac{1}{2}\sum_{i=1}^4(\tau_i-\pi)^2\\
&+(\xi-\pi)^2-\sum_{i=1}^4(\xi-\tau_i)^2-\sum_{j=1}^3(\eta_j-\xi)^2\\
&-2\Li(1)-\frac{1}{2}\sum_{i=1}^4\sum_{j=1}^3\Li\big(e^{2i(\eta_j-\tau_i)}\big)+\frac{1}{2}\sum_{i=1}^4\Li\big(e^{2i(\tau_i-\pi)}\big)\\
&-\Li\big(e^{2i(\xi-\pi)}\big)+\sum_{i=1}^4\Li\big(e^{2i(\xi-\tau_i)}\big)+\sum_{j=1}^3\Li\big(e^{2i(\eta_j-\xi)}\big).\\
\end{split}
\end{equation*}

Let $\boldsymbol\alpha=(\alpha_1,\dots,\alpha_6)$ and $\mathrm{Re}(\boldsymbol\alpha) = (\mathrm{Re}(\alpha_1),\dots,\mathrm{Re}(\alpha_6))$, where $\mathrm{Re}(\alpha_i)$ is the real part of $\alpha_i$ for $i=1,\dots,6$. Let
\begin{equation*}
B_{H,\mathbb C} = 
\left\{
 (\boldsymbol\alpha,\xi)\in\mathbb C^7 \;\middle|\;
  \begin{aligned}
  & \mathrm{Re}(\boldsymbol\alpha)\text{ is of the hyperideal type},\\
  &  \max\{\mathrm{Re}(\tau_i)\}\leqslant \mathrm{Re}(\xi)\leqslant \min\{\mathrm{Re}(\eta_j), 2\pi\}
  \end{aligned}
\right\}
\end{equation*}
and
$$ B_H = B_{H,\mathbb C}\cap \RR^7. $$
By (\ref{QDtoL}), on $B_{H}$, we have
\begin{align}\label{Ureal}
U(\boldsymbol\alpha, \xi)
=2\pi^2 + 2\sqrt{-1} V(\boldsymbol\alpha, \xi),
\end{align}
where $V: B_H \to \RR$ is defined by
\begin{align}\label{defV}
V(\boldsymbol\alpha,\xi)
&= \delta(\alpha_1,\alpha_2,\alpha_3)+\delta(\alpha_1,\alpha_5,\alpha_6)+
\delta(\alpha_2,\alpha_4,\alpha_6)+\delta(\alpha_3,\alpha_4,\alpha_5) \notag\\
&\quad - \Lambda(\xi) + \sum_{i=1}^4 \Lambda(\xi-\tau_i) + \sum_{j=1}^3 \Lambda(\eta_j - \xi)
\end{align}
with 
\begin{align*}
\delta(x,y,z)
= -\frac{1}{2}\Lambda\lt( \frac{x+y-z}{2} \rt) -\frac{1}{2}\Lambda\lt( \frac{y+z-x}{2} \rt) 
-\frac{1}{2}\Lambda\lt( \frac{z+x-y}{2} \rt) +\frac{1}{2}\Lambda\lt( \frac{x+y+z}{2} \rt).
\end{align*}

The function $V$ has been studied by Costantino in \cite{C1}. In particular, in the proof of \cite[Theorem 3.9]{C1}, he proved that for each $\alpha$ of the hyperideal type,
\begin{enumerate}
\item $V(\boldsymbol\alpha,\xi)$ is strictly concave down in $\xi$,
\item there exists a unique $\xi(\boldsymbol\alpha)$ so that 
$$ (\boldsymbol\alpha, \xi(\boldsymbol\alpha)) \in B_H \quad\text{and}\quad \lt.\frac{\partial V(\boldsymbol\alpha, \xi)}{\partial \xi} \rt|_{\xi = \xi(\boldsymbol\alpha)} = 0, \text{ and}$$
\item $V(\boldsymbol\alpha,\xi)$ attains its maximum at $\xi(\boldsymbol\alpha)$ with the critical value $V(\boldsymbol\alpha,\xi(\boldsymbol\alpha))$ given by
$$ V(\boldsymbol\alpha, \xi(\boldsymbol\alpha)) = \Vol(\Delta_{|\alpha - \pi|}), $$
where $\Vol(\Delta_{|\alpha-\pi|})$ is the volume of the ideal or the truncated hyperideal tetrahedron with dihedral angles $|\alpha_1-\pi|, \dots, |\alpha_6-\pi|$.
\end{enumerate}

As a special case, when $\alpha_1=\dots=\alpha_6 = \pi$, by direct computation we have 
\begin{align}\label{xipi}
\xi(\pi,\dots,\pi) = \frac{7\pi}{4}.
\end{align}
Furthermore, for $i,j\in\{1,\dots,6\}$ with $i\neq j$, at $\lt(\pi,\dots,\pi,\frac{7\pi}{4}\rt)$ we have
$$ \frac{\partial^2 V}{\partial \alpha_i^2} = -2, \quad
\frac{\partial^2 V}{\partial \alpha_i \alpha_j} = -1,\quad
\frac{\partial^2 V}{\partial \alpha_i \partial \xi}= 2     \quad\text{and}\quad
\frac{\partial^2 V}{\partial \xi^2}= -8.
$$

From this, we have the following lemma, which will be used later to prove the convexity result in Proposition \ref{prop53}.
\begin{lemma}\label{convexity}
The Hessian matrix of $V(\boldsymbol\alpha, \xi)$ is negative definite at $\lt(\pi,\dots,\pi,\frac{7\pi}{4}\rt)$.
\end{lemma}

We also need to following estimation of $V$ from \cite{BDKY}.

\begin{lemma}\label{Vmax} For each $(\alpha_1,\dots, \alpha_6, \xi)\in B_H$, we have $ V(\alpha_1,\dots,\alpha_6,\xi) \leq v_8$,
where $v_8$ is the volume of the regular ideal octahedron. Moreover, the equality holds if and only if $(\alpha_1,\dots,\alpha_6,\xi) = \lt(\pi,\dots,\pi,\frac{7\pi}{4}\rt)$.
\end{lemma}
\begin{proof}
This result is proved in \cite[Lemma 3.5]{BDKY}. To be precise, the authors of \cite{BDKY} studied the maximum of $V$ on boundary points of $B_H$, the non-smooth points and the critical points of the interior smooth points. From this, they proved that $V$ attains its maximum at the unique maximum point $(\alpha_1,\dots,\alpha_6,\xi) = \lt(\pi,\dots,\pi,\frac{7\pi}{4}\rt)$ with value $v_8$. See \cite{BDKY} for more details.
\end{proof}

For $(x_1, \dots, x_n), (y_1, \dots, y_n )\in \CC^n$, let $d_{\infty}$ be the real maximum norm on $\mathbb{C}^n$ defined by
\begin{equation*}
    d_{\infty} ((x_1, \dots, x_n), (y_1, \dots, y_n )) = \max_{i \in \{1, \dots, n\}} \{|\Re(x_i) - \Re(y_i)|, |\im(x_i) - \im(y_i)|\}.
\end{equation*}

\begin{lemma}\label{less2pi}
There exists $\delta_1>0$ such that if $d_{\infty}\big( (\alpha_1,\dots,\alpha_6,\xi), (\pi,\dots,\pi, \frac{7\pi}{4}) \big) < \delta_1$, then
$$
\Big|\frac{\partial \im U}{\partial \im \xi} \Big| < 2\pi.
$$
\end{lemma}
\begin{proof}
The result follows from the facts that $\im U$ is smooth and $\frac{\partial U }{\partial \xi} (\pi,\dots,\pi, \frac{7\pi}{4}) = 0$.
\end{proof}

\subsubsection{Geometry of 6j-symbol}
For $\boldsymbol\alpha=(\alpha_1,\dots,\alpha_6)\in\mathbb C^6$ such that $(\mathrm{Re}(\alpha_1),\dots,\mathrm{Re}(\alpha_6))$ is of the hyperideal type, let $U_{\boldsymbol\alpha}(\xi) = U(\boldsymbol\alpha, \xi)$ and let $\xi(\boldsymbol\alpha)$ be such that
\begin{equation}\label{xia}
\frac{d U_{\boldsymbol\alpha}(\xi)}{d\xi} \Big|_{\xi=\xi(\boldsymbol\alpha)} = \frac{\partial U(\boldsymbol\alpha,\xi)}{\partial \xi}\Big|_{\xi=\xi(\boldsymbol\alpha)}=0.
\end{equation}
It is proved in \cite{BY} that such $\xi(\boldsymbol\alpha)$ exists. In particular, for $\boldsymbol\alpha\in\mathbb C^6$ so that $(\boldsymbol\alpha,\xi(\boldsymbol\alpha))\in B_{H,\mathbb C},$ we define
\begin{equation}\label{W}
W(\boldsymbol\alpha)=U(\boldsymbol\alpha,\xi(\boldsymbol\alpha)).
\end{equation}

\begin{theorem}\label{co-vol}(\cite[Theorem 3.5]{BY}) For a partition $(I,J)$ of $\{1,\dots,6\}$ and a deeply truncated tetrahedron $\Delta$ of type $(I,J),$ we let $\{l_i\}_{i\in I}$ and $\{\theta_i\}_{i\in I}$ respectively be the lengths of and dihedral angles at the edges of deep truncation,  and let $\{\theta_j\}_{j\in J}$ and  $\{l_j\}_{j\in J}$ respectively be the dihedral angles  at and the lengths of the regular edges. Then 
$$W\big((\pi\pm \sqrt{-1} l_i)_{i\in I},(\pi\pm \theta_j)_{j\in J}\big)=2\pi^2+2\sqrt{-1} \mathrm{Cov}\big((l_i)_{i\in I},(\theta_j)_{j\in J}\big)$$
 where $\mathrm{Cov}$ is the co-volume function defined by
 $$\mathrm{Cov}\big((l_i)_{i\in I},(\theta_j)_{j\in J}\big)=\mathrm{Vol}(\Delta)+\frac{1}{2}\sum_{i\in I}\theta_il_i,$$
which for $i\in I$ satisfies
$$\frac{\partial \mathrm{Cov}}{\partial l_i}=\frac{\theta_i}{2}$$
and  for $j\in J$ satisfies
$$\frac{\partial \mathrm{Cov}}{\partial \theta_j}=-\frac{l_j}{2}.$$
\end{theorem}

\subsubsection{Neumann-Zagier potential functions of fundamental shadow link complements}
Finally, to understand the geometry of the critical points of the function $G^{\boldsymbol { E}_I}(\boldsymbol {\alpha}_{\zeta_I}, \boldsymbol \xi)$ defined in (\ref{defG}), we need the following result from \cite{WY1}.

For $s\in\{1,\dots,c\},$ let $\boldsymbol\alpha_s=(\alpha_{s_1},\dots,\alpha_{s_6}).$ Consider the following function
$$\mathcal U(\boldsymbol \alpha_{\zeta_I}, \boldsymbol \alpha_{J})= -\sum_{i=1}^n\frac{\iota_i}{2}(\alpha_i-\pi)^2+\sum_{s=1}^c U(\boldsymbol\alpha_s,\xi(\boldsymbol\alpha_s))+\Big(\sum_{i=1}^n\frac{\iota_i}{2}\Big)\pi^2.$$
for all $(\boldsymbol \alpha_{\zeta_I}, \boldsymbol \alpha_J)$ such that $(\boldsymbol\alpha_s,\xi(\boldsymbol\alpha_s))\in B_{H,\mathbb C}$ for all $s\in\{1,\dots,c\}.$ Then we have

\begin{proposition}(\cite[Proposition 4.1]{WY1})\label{NeuZ} For each component $T_i$ of the boundary of $M_c\setminus L_{\text{FSL}},$ choose the basis $(u_i,v_i)$ of $\pi_1(T_i)$ as in (\ref{m}) and (\ref{l}), and let $\Phi$ be the Neumann-Zagier potential function characterized by
\begin{equation}\label{char2}
\left \{\begin{array}{l}
\frac{\partial \Phi(\mathrm{H}(u_1),\dots,\mathrm{H}(u_n))}{\partial \mathrm{H}(u_i)}=\frac{\mathrm H(v_i)}{2},\\
\\
\Phi(0,\dots,0)=\sqrt{-1}\bigg(\mathrm{Vol}(M_c\setminus L_{\text{FSL}})+\sqrt{-1}\mathrm{CS}(M_c\setminus L_{\text{FSL}})\bigg)\quad\quad \mod \pi^2\mathbb Z,
\end{array}\right.
\end{equation} 
where $ M_c\setminus L_{\text{FSL}}$ is with the complete hyperbolic metric. If $\mathrm H(u_i)=\pm 2\sqrt{-1}(\alpha_{i,\zeta_i}-\pi)$ for each $i\in I$ and $\mathrm H(u_j)=\pm 2\sqrt{-1}(\alpha_{j}-\pi)$ for each $j\in J$, then
$$\mathcal U(\boldsymbol \alpha_{\zeta_I}, \boldsymbol \alpha_J)=2c\pi^2+\Phi(\mathrm H(u_1),\dots,\mathrm H(u_n)).$$
\end{proposition}

\subsection{Convexity}
In this subsection we study the convexity of the function $G^{\mathbf E_I}$. Recall from (\ref{defG}) that 
\begin{align}\label{defG}
G^{\boldsymbol  E_I}(\boldsymbol \alpha_{\zeta_I}, \boldsymbol \xi) 
&=  \sum_{i\in I} 
\lt[- \lt(\frac{p_i'}{q_i} + a_{i,0}\rt) (\beta_i - \pi)^2 - \frac{p_i (\alpha_{i,\zeta_i} - \pi)^2 +   2 E_i(\beta_i - \pi)(\alpha_{i,\zeta_i} - \pi)}{q_i}
\rt] \notag
\\
&\qquad - \sum_{j \in J} a_{j,0}(\alpha_j - \pi)^2  -\sum_{i=1}^n\frac{\iota_i}{2}(\alpha_i-\pi)^2+\sum_{s=1}^c U(\alpha_{s_1},\dots,\alpha_{s_6},\xi_s)+\Big(\sum_{i=1}^n\frac{\iota_i}{2}\Big)\pi^2.
\end{align}For $\delta >0$, we denote by $D_{\delta, \mathbb{C}}$ the $\delta$-neighborhood of $(\pi, \dots, \pi, \frac{7 \pi}{4}, \dots , \frac{7 \pi}{4})$ in $\mathbb{C}^{|I|+ c}$ with respect to the maximum norm, that is 
\begin{equation*}
    D_{\delta, \mathbb{C}} = \bigg \{ (\boldsymbol\alpha_{\zeta_I}, \boldsymbol\xi) \in \mathbb{C}^{|I| + c} \bigg | d_{\infty} \bigg( (\boldsymbol\alpha_{\zeta_I}, \boldsymbol\xi ), \lt(\pi , \dots, \pi , \frac{7\pi}{4}, \dots , \frac{7\pi}{4}\rt)\bigg) < \delta \bigg \},
\end{equation*}
where $d_{\infty}$ is the real maximum norm on $\mathbb{C}^n$ defined by
\begin{equation*}
    d_{\infty} ((x_1, \dots, x_n), (y_1, \dots, y_n )) = \max_{i \in \{1, \dots, n\}} \{|\Re(x_i) - \Re(y_i)|, |\im(x_i) - \im(y_i)|\}.
\end{equation*}
We will also consider the region 
\begin{equation*}
    D_{\delta} = D_{\delta, \mathbb{C}} \cap \mathbb{R}^{|I|+ c}.
\end{equation*}

Let 
\begin{align}\label{deftU}
\tilde U(\boldsymbol \alpha_{\zeta_I}, \boldsymbol \xi) 
=&
- \sum_{j \in J} a_{j,0}(\alpha_j - \pi)^2  -\sum_{i=1}^n\frac{\iota_i}{2}(\alpha_i-\pi)^2+\sum_{s=1}^c U(\alpha_{s_1},\dots,\alpha_{s_6},\xi_s)+\Big(\sum_{i=1}^n\frac{\iota_i}{2}\Big)\pi^2.
\end{align}
Let $\delta_1>0$ be the constant in Lemma \ref{less2pi}.
\begin{proposition}\label{convextU} There exists a $ \delta_0 \in (0,\delta_1)$ such that if all $\{ \alpha_j \}_{ j \in J}$ are in $(\pi - \delta_0, \pi + \delta_0)$, then $\im \tilde U(\boldsymbol \alpha_{\zeta_I}, \boldsymbol \xi) $ is strictly concave down in $\{\Re (\alpha_{i,\zeta_i})\}_{i \in I}$  and $\{\Re (\xi_s)\}_{s=1}^c$ and is strictly concave up in $\{\im (\alpha_{i,\zeta_i})\}_{i \in I}$  and $\{\im (\xi_s)\}_{s=1}^c$ on $D_{\delta_0, \mathbb{C}}$.
\end{proposition}
\begin{proof}
Note that when all $\{\alpha_{i,\zeta_i}\}_{i\in I}$ and $\{\xi_s\}_{s=1}^c$ are real, by (\ref{Ureal}) we have
$$ \im \tilde U(\boldsymbol \alpha_{\zeta_I}, \boldsymbol \xi) = \sum_{s=1}^c 2V(\alpha_{s_1},\dots, \alpha_{s_6}, \xi_s). $$
Therefore, when $\alpha_{i,\zeta_i} = \pi$ for all $i\in I$ and $\xi_s=\frac{7\pi}{4}$ for $s=1,\dots, c$, by Lemma \ref{convexity}, the Hessian matrix of $\im  \tilde U(\boldsymbol \alpha_{\zeta_I}, \boldsymbol \xi) $ is negative definite in $\{ \Re (\alpha_{i,\zeta_i})\}_{i\in I}$ and $\{\Re(\xi_s)\}_{s=1}^c$. 

By continuity, we can find a sufficiently small $\delta_0\in(0,\delta_1)$ such that for all $\{\alpha_j\}_{j \in J}$ in $(\pi - \delta_0, \pi+ \delta_0)$ and $(\boldsymbol \alpha_{\zeta_I}, \boldsymbol\xi) \in D_{\delta_0, \CC}$, the Hessian matrix of $\im \tilde U(\boldsymbol \alpha_{\zeta_I}, \boldsymbol \xi) 
$ is negative definite in $\{ \Re (\alpha_{i,\zeta_i})\}_{i\in I}$ and $\{\Re(\xi_s)\}_{s=1}^c$. As a result, $\im \tilde U(\boldsymbol \alpha_{\zeta_I}, \boldsymbol \xi) 
$ is strictly concave down in $\{ \Re (\alpha_{i,\zeta_i})\}_{i\in I}$ and $\{\Re(\xi_s)\}_{s=1}^c$. Finally, by the holomorphicity of the function $ \tilde U(\boldsymbol \alpha_{\zeta_I}, \boldsymbol \xi) $, $\im \tilde U(\boldsymbol \alpha_{\zeta_I}, \boldsymbol \xi) 
$ is strictly concave up in $\{\im (\alpha_{i,\zeta_i})\}_{i \in I}$  and $\{\im (\xi_s)\}_{s=1}^c$.
\end{proof}

Proposition \ref{prop53} and \ref{prop54} are analogue of Proposition 5.3 and 5.4 in \cite{WY1}. 
\begin{proposition}\label{prop53} For any $\mathbf E_I \in \{1,-1\}^{|I|}$, $\im G^{\boldsymbol { E_I}}(\boldsymbol \alpha_{\zeta_I}, \boldsymbol \xi)$ is strictly concave down in $\{\Re (\alpha_{i,\zeta_i})\}_{i \in I}$  and $\{\Re (\xi_s)\}_{s=1}^c$ and is strictly concave up in $\{\im (\alpha_{i,\zeta_i})\}_{i \in I}$  and $\{\im (\xi_s)\}_{s=1}^c$ on $D_{\delta_0, \mathbb{C}}$.
\end{proposition}
\begin{proof}
Note that
\begin{align*}
&G^{\boldsymbol  E_I}(\boldsymbol \alpha_{\zeta_I}, \boldsymbol \xi) \\
=&  \sum_{i\in I} 
\lt[- \lt(a_{i,0} + \frac{p_i'}{q_i}\rt) (\beta_i - \pi)^2 - \frac{p_i (\alpha_{i,\zeta_i} - \pi)^2 + 2 E_i(\beta_i - \pi)(\alpha_{i,\zeta_i} - \pi)}{q_i}
\rt] + \tilde U(\boldsymbol \alpha_{\zeta_I}, \boldsymbol \xi) .
\end{align*}
In particular, the $\im(G^{\boldsymbol  E_I}(\boldsymbol \alpha_{\zeta_I}, \boldsymbol \xi) - 
\tilde U(\boldsymbol \alpha_{\zeta_I}, \boldsymbol \xi))$ is a linear function in $\{\Re (\alpha_{i,\zeta_i})\}_{i \in I}$ and $\{\im (\alpha_{i,\zeta_i})\}_{i \in I}$. Since the convexity of a function does not change under addition of linear functions, the result follows from Proposition \ref{convextU}.
\end{proof}

\begin{proposition}\label{prop54} If all $\{ \alpha_j\}_{j \in J}$ are in $(\pi - \delta_0, \pi + \delta_0)$, then the Hessian matrix $\Hess({G}^{\boldsymbol{ E}_I} )$ with respect to $\{\alpha_{i,\zeta_i}\}_{i \in I}$  and $\{\xi_s\}_{s=1}^c$ is non singular on $D_{\delta_0, \mathbb{C}}$.
\end{proposition}
\begin{proof}
From Proposition \ref{prop53}, we see that the real part of $\Hess({G}^{\mathbf E_I}) $ is negative definite. By [\cite{L}, Lemma], the matrix $\Hess({G}^{\mathbf E_I}) $ is non-singular.
\end{proof}

\begin{remark}
The constant $\delta_0>0$ in Proposition \ref{prop53} and \ref{prop54} depends only on the fundamental shadow link but not on $(p_i,q_i), \mathbf E_I$ and $\beta_i$.
\end{remark}

\subsection{Critical Points and critical values}
In Proposition~\ref{prop52} we will prove that certain critical value of the function $G^{\boldsymbol { E_I}}(\boldsymbol \alpha_{\zeta_I}, \boldsymbol \xi)$ gives the hyperbolic volume and the Chern-Simons invariant of the cone manifold $M_{L_{\boldsymbol\theta}}$. 

For $i \in I$, let $\theta_i = 2|\beta_i-\pi|$ and let $\mu_i =1 \text{ if } \beta_i-\pi \geq 0, \mu_i = -1 \text{ if } \beta_i-\pi \leq 0 $. By definition, we have $\theta_i = 2\mu_i (\beta_i-\pi)$. Consider the $(p_i, q_i)$ Dehn-filling equation with cone angle $\theta_i$
\begin{align}\label{coneqn}
p_i \mathrm{H}(u_i) + q_i \mathrm{H}(v_i) = \sqrt{-1} \theta_i,
\end{align}
where $\mathrm{H}(u_i)$ and $\mathrm{H}(v_i)$ are the logarithmic holonomies of the meridian and the longitude respectively.

Then we have the following analogue of Proposition 5.2 in \cite{WY1}.
\begin{proposition}\label{prop52}For each $i\in I,$ let $\mathrm H(u_i)$ be the logarithmic holonomy of $u_i$ of the hyperbolic cone manifold $M_{L_{\boldsymbol\theta}}$ and let
\begin{equation}\label{alpha}
\alpha^*_i=\pi+\frac{E_i\mu_i\sqrt{-1}}{2}\mathrm H(u_i).
\end{equation}
For $ s \in \{1,2,\dots,c\}$, let $\xi^* = \xi(\alpha_{s_1}^*,\dots, \alpha_{s_6}^*)$ be as defined in (\ref{xia}). Assume that 
$$ \mathbf z^{\mathbf E_I} = \lt( (\alpha_i^*)_{i\in I}, (\xi_s^*)_{s=1}^c \rt) \in D_{\delta_0, \mathbb{C}}$$ 
for $\delta_0$ defined in Proposition \ref{prop53}.  
Then $G^{\boldsymbol { E_I}}(\boldsymbol \alpha_{\zeta_I}, \boldsymbol \xi)$ has a critical point 
\begin{equation*}
    \mathbf z^{\mathbf E_I} = \lt( (\alpha_i^*)_{i\in I}, (\xi_s^*)_{s=1}^c \rt)
\end{equation*}
in $D_{\delta_0, \mathbb{C}}$ with critical value
\begin{equation*}
    2c\pi^2 + \sqrt{-1}( \Vol(M_{L_{\boldsymbol\theta}}) + \sqrt{-1}\CS(M_{L_{\boldsymbol\theta}})  )
\end{equation*}
\end{proposition}

\begin{proof}
For $s\in \{1,\dots, c\}$, we let $\boldsymbol\alpha_s=(\alpha_{s_1},\dots, \alpha_{s_6})$ and 
$\boldsymbol\alpha_s^* = (\alpha_{s_1}^*,\dots, \alpha_{s_6}^*)$. 
For each $s \in \{1, \dots, c\},$ by Equation (\ref{xia}),
\begin{equation}\label{6jeq}
    \left.\frac{\partial G^{\boldsymbol  E_I}}{\partial \xi_s}\right\vert_{\mathbf z^{\mathbf E_I}} = \left. \frac{\partial U(\boldsymbol\alpha_s, \xi_s)}{\partial \xi_s}\right\vert_{\xi_s^*} = 0
\end{equation}

Besides, by the chain rule, for each $s \in \{ 1, \dots , c\}$ and $i \in I$,
\begin{equation*}
    \left.\frac{\partial U(\boldsymbol\alpha_s, \xi(\boldsymbol\alpha_s))}{\partial \alpha_{i,\zeta_i}}\right\vert_{\boldsymbol\alpha_s^*}
    = \left. \frac{\partial U(\boldsymbol\alpha_s, \xi_s)}{\partial \alpha_{i,\zeta_i}}\right\vert_{(\boldsymbol\alpha_s^*, \xi_s^*)} + \left.\frac{\partial U(\boldsymbol\alpha_s, \xi_s)}{\partial \xi_s}\right\vert_{(\boldsymbol\alpha_s^*, \xi_s^*)} . \left. \frac{\partial \xi(\boldsymbol\alpha_s)}{\partial \alpha_{i,\zeta_i}} \right\vert_{\alpha_s} 
    =  \left. \frac{\partial U(\boldsymbol\alpha_s, \xi_s)}{\partial \alpha_{i,\zeta_i}}\right\vert_{(\boldsymbol\alpha_s^*, \xi^*)}.
\end{equation*}
Hence, by (\ref{char2}),
\begin{align} \label{df2}
    \left. \frac{\sum_{s=1}^c \partial U(\boldsymbol\alpha_s, \xi(\boldsymbol\alpha_s))}{\partial \alpha_{i,\zeta_i}} \right \vert_{\mathbf z^{\mathbf E_I}} = \left. \frac{\partial \mathcal{U}}{\partial \alpha_{i,\zeta_i}} \right \vert_{(\alpha_i^*)_{i\in I}} = -  E_i \mu_i \sqrt{-1} \mathrm{H}(v_i).
\end{align}
As a result,
\begin{align}\label{dfeq}
    \lt.\frac{\partial G^{\boldsymbol  E_I}}{\partial \alpha_{i,\zeta_i}}\rt\vert_{\mathbf z^{\mathbf E_I}}
   &= \frac{-2p_i (\alpha_{i}^*-\pi) -  2 E_i (\beta_i - \pi)}{q_i} + \lt.\frac{\partial \mathcal{U}}{\partial \alpha_{i,\zeta_i}}\rt \vert_{(\alpha_i^*)_{i\in I}}
   \notag \\
    &= \frac{-2p_i (\alpha_{i}^*-\pi) -  2 E_i (\beta_i - \pi)}{q_i}  - E_i \mu_i \sqrt{-1}\mathrm{H}(v_i) \notag \\
    &= -\frac{ E_i\mu_i\sqrt{-1}}{q_i}
    (p_i \mathrm{H}(u_i) + q_i \mathrm{H}(v_i) - \sqrt{-1} \theta_i) \notag\\
    &=0,
\end{align}
where the last equality comes from the $(p_i, q_i)$ Dehn-filling equation with cone angle $\theta_i$. Thus, from Equations (\ref{6jeq}) and (\ref{dfeq}), we see that $\mathbf z^{\mathbf E_I}$ is a critical point of $G^{\boldsymbol  E_I}$.

To compute the critical value, by Proposition \ref{NeuZ}, we have 
\begin{align}
\mathcal{U}(\boldsymbol \alpha_{\zeta_I}, \boldsymbol \alpha_J) = 2c\pi^2 + \Phi(\mathrm{H}(u_1), \dots, \mathrm{H}(u_n)). \label{neuZ1}
\end{align}
For each $i \in I$, let 
$ \gamma_i  = (-q_i'u_i + p_i'v_i)+ a_{i,0}(p_i u_i + q_i v_i) $ so that it is the curve on the boundary of a tubular neighborhood of $L_i$ that is isotopic to $L_i$ given by the framing $a_{i,0}$ of $L_i$ and with the orientation so that $ (p_i u_i + q_i v_i).\gamma_i  = 1$.
By definition, we have $\theta_i = 2\mu_i (\beta_i - \pi)$ and $\mathrm{H}(u_i) = -2\sqrt{-1} E_i \mu_i (\alpha_i^* - \pi)$. Besides, by the $(p_i,q_i)$ Dehn-filling equation $p_i \mathrm{H}(u_i) + q_i \mathrm{H}(v_i) = \sqrt{-1} \theta_i$, we have
\begin{align}
\mathrm{H}(v_i) &= \frac{\sqrt{-1}\theta_i -p_i \mathrm{H}(u_i) }{q_i}
= \frac{2\mu_i \sqrt{-1}}{q_i} [(\beta_i - \pi) + p_i  E_i (\alpha_i^* - \pi)].
\end{align}
As a result,
\begin{align}\label{cvuv}
-\frac{\mathrm{H}(u_i)\mathrm{H}(v_i)}{4}
= -\frac{ E_i (\alpha_i^*-\pi)(\beta- \pi)}{q_i} - \frac{p_i(\alpha_i^* - \pi)^2}{q_i}.
\end{align}
Besides, by (\ref{corecurveHol}), we have
\begin{align} \label{hgammacompu}
\mathrm{H}(\gamma_i) 
&= -q_i'\mathrm{H}(u_i) + p_i'\mathrm{H}(v_i) + a_{i,0}\theta_i \sqrt{-1} \notag\\
&= -\lt(q_i' + \frac{p_i p_i'}{q_i}\rt)\mathrm{H}(u_i) + \lt(\frac{p_i'}{q_i}+a_{i,0}\rt)\theta_i \sqrt{-1} \notag\\ 
&= -\frac{\mathrm{H}(u_i)}{q_i}+ \lt(\frac{p_i'}{q_i}+a_{i,0}\rt)\theta_i \sqrt{-1}.
\end{align}
This implies that
\begin{align}\label{cvgamma}
\frac{\sqrt{-1}\theta_i\mathrm{H}(\gamma_i)}{4}
= -\frac{ E_i (\alpha_i^* - \pi)(\beta_i - \pi)}{q_i} - \lt(\frac{p_i'}{q_i}+a_{i,0}\rt) ( \beta_i - \pi)^2.
\end{align}
By Equations (\ref{cvuv}) and (\ref{cvgamma}), we have
\begin{align} 
&-\sum_{i \in I} \frac{ \mathrm{H}(u_i)\mathrm{H}(v_i)}{4} +  \sum_{i \in I} \frac{\sqrt{-1} \theta_i \mathrm{H}(\gamma_i)}{4}\notag\\ 
=& 
-\sum_{i \in I} \frac{2 E_i (\alpha_i^* - \pi)(\beta_i - \pi)}{q_i} 
- \sum_{i \in I} \frac{p_i}{q_i}(\alpha_i^* - \pi)^2 
- \sum_{i \in I }\lt(\frac{p_i'}{q_i}+a_{i,0}\rt) ( \beta_i - \pi)^2. \label{uvI}
\end{align}

For each $j \in J$, let $ \gamma_j = a_{j,0} u_j +  v_j$ so that the curve on the boundary of a tubular neighborhood of $L_j$ that is isotopic to $L_j$ given by the framing $a_{j,0}$ of $L_j$ and with the orientation such that $u_j . \gamma_j = 1$.
\\
Then we have $\theta_j = 2|\alpha_j - \pi | = 2\mu_j(\alpha_j - \pi)$ for some $\mu_j \in\{-1,1\}$, $\mathrm{H}(u_j) = 2\sqrt{-1} |\alpha_j - \pi|$ and 
$ \mathrm{H}(\gamma_j) = a_{j,0}  \mathrm{H}(u_j) + \mathrm{H}(v_j)$. As a consequence, we have
\begin{align} 
    -\sum_{j \in J}\frac{\mathrm{H}(u_j)\mathrm{H}(v_j)}{4} + \sum_{j \in J} \frac{\sqrt{-1}\theta_j \mathrm{H}(\gamma_j)}{4} 
    &=     -\sum_{j \in J}\frac{\mathrm{H}(u_j)\mathrm{H}(v_j)}{4} + \sum_{j \in J} \frac{\mathrm{H}(u_j) (a_{j,0}  \mathrm{H}(u_j) + \mathrm{H}(v_j))}{4}  \notag\\
    &= -\sum_{j \in J} a_{j,0} (\alpha_j - \pi)^2. \label{uvJ}
\end{align}
From (\ref{neuZ1}), (\ref{uvI}), (\ref{uvJ}) and (\ref{VCS}), we have
\begin{align*}
G^{\mathbf E_I}(\mathbf z^{\mathbf E_I})  
&= \sum_{i\in I} 
\lt[- \lt(\frac{p_i'}{q_i} + a_{i,0}\rt) (\beta_i - \pi)^2 - \frac{p_i (\alpha_i^* - \pi)^2 +  2E_i (\beta_i - \pi)(\alpha_i^* - \pi)}{q_i}
\rt]
\\
&\qquad - \sum_{j \in J} a_{j,0}(\alpha_j - \pi)^2  + \mathcal{U}(\boldsymbol \alpha_{\zeta_I}, \boldsymbol \alpha_J) \\
&= 2c\pi^2 + \Phi(\mathrm{H}(u_1),\dots, \mathrm{H}(u_n)) - \sum_{i=1}^n \frac{\mathrm{H}(u_i)\mathrm{H}(v_i)}{4} + 
\sum_{i=1}^n \frac{\sqrt{-1}\theta_i \mathrm{H}(\gamma_i)}{4} \\
&= 2c\pi^2 + \sqrt{-1}(\Vol(M_{L_{\boldsymbol\theta}}) + \sqrt{-1}\CS(M_{L_{\boldsymbol\theta}}))
\end{align*}
\end{proof}

\subsection{Asymptotics of the leading Fourier coefficients }
\begin{proposition}\label{lfc}
Let $\boldsymbol E_{I}\in \{1,-1\}^{|I|}$ and let $\mathbf {z}^{ \boldsymbol E_{I}}$ be the critical point described in Proposition \ref{prop52}. Assume that 
\begin{enumerate}
\item 
$\mathbf {z}^{ \boldsymbol E_{I}} \in D_{\delta_0,\CC}$ and 
\item 
$ \Vol(M_{L_{\boldsymbol\theta}})> \max_{(\boldsymbol \alpha_{\zeta_I}, \boldsymbol \xi) \in \overline{D_H\setminus D_{\delta_0}}} \im \tilde U(\boldsymbol \alpha_{\zeta_I}, \boldsymbol \xi) $, where $ \tilde U(\boldsymbol \alpha_{\zeta_I}, \boldsymbol \xi) $ is defined in (\ref{deftU}) and $\overline{D_H\setminus D_{\delta_0}}$ is the closure of $D_H\setminus D_{\delta_0}$.
\end{enumerate}
Then the asymptotics of the integral on the right hand side of Proposition \ref{lfcexpress} 
\begin{align*}
&\int_{D_H} 
    \phi_r\lt(\mathbf{s}^{\mathbf E_I}, \boldsymbol{\alpha}_{\zeta_I}, \boldsymbol\xi \rt)     
   e^{\frac{r}{4\pi \sqrt{-1}}G_r^{\mathbf E_I}(\boldsymbol\alpha_{\xi_I},\boldsymbol\zeta)} d\boldsymbol{\alpha}_{\zeta_I} d\boldsymbol\xi \notag \\
=& \lt(\frac{2}{r}\rt)^c \lt(\frac{2\pi}{r}\rt)^{\frac{|I|+c}{2}} (4\pi\sqrt{-1})^{\frac{|I|+c}{2}}
\frac{(-1)^{-\frac{rc}{2}}C^{\mathbf E_I}(\mathbf{z}^{\boldsymbol E_{I}})}{\sqrt{-\det \Hess({G}^{\mathbf E_I} )(\mathbf{z^{\mathbf E_I}})}} e^{\frac{r}{4 \pi} (\Vol (M_{L_{\boldsymbol\theta}})  + \sqrt{-1}\CS(M_{L_{\boldsymbol\theta}}))} \Big( 1 + O\Big(\frac{1}{r}\Big ) \Big),
\end{align*}
where each $C^{\mathbf E_I}(\mathbf {z}^{\mathbf E_I})$ depends continuously on $\{ \beta_i\}_{i \in I}$ and $\{ \alpha_j\}_{j \in J}$ and is given by
\begin{align}\label{Cformula}
&C^{\mathbf E_I}(\mathbf {z}^{\mathbf E_I}) 
\notag\\
=&  
 e^{ \sqrt{-1}\Big(\sum_{i\in I}\Big(\frac{p_i'}{q_i}(\beta_i - \pi) + \frac{p_i}{q_i}(\alpha_{i}^* - \pi) + \frac{E_i(\alpha_{i}^*+\beta_i-2\pi)}{q_i}\Big)+ \sum_{i\in I}a_{i,0}\beta_i + \sum_{i\in I} \Big(\frac{\iota_i}{2}\Big)\alpha_{i}^* + \sum_{j\in J}\Big(a_{j,0}+\frac{\iota_j}{2}\Big)\alpha_j \Big) + \kappa(\mathbf {z}^{\mathbf E_I})},
\end{align}
where $\kappa$ is defined in Lemma \ref{qdtod}.
\end{proposition}

\begin{proof} Let $\delta_0>0$ be as in Proposition \ref{prop53}. We write 
\begin{align*}
&\int_{D_H} 
    \phi_r\lt(\mathbf{s}^{\mathbf E_I}, \boldsymbol{\alpha}_{\zeta_I}, \boldsymbol\xi \rt)     
   e^{\frac{r}{4\pi \sqrt{-1}}G_r^{E_I}(\boldsymbol\alpha_{\xi_I},\boldsymbol\zeta)} d\boldsymbol{\alpha}_{\zeta_I} d\boldsymbol\xi  \notag \\
=&
\int_{D_{\delta_0}} 
 \phi_r\lt(\mathbf{s}^{\mathbf E_I}, \boldsymbol{\alpha}_{\zeta_I}, \boldsymbol\xi \rt)     
   e^{\frac{r}{4\pi \sqrt{-1}}G_r^{E_I}(\boldsymbol\alpha_{\xi_I},\boldsymbol\zeta)} d\boldsymbol{\alpha}_{\zeta_I} d\boldsymbol\xi 
   +
\int_{D_{H} \setminus D_{\delta_0}} 
 \phi_r\lt(\mathbf{s}^{\mathbf E_I}, \boldsymbol{\alpha}_{\zeta_I}, \boldsymbol\xi \rt)     
   e^{\frac{r}{4\pi \sqrt{-1}}G_r^{E_I}(\boldsymbol\alpha_{\xi_I},\boldsymbol\zeta)} d\boldsymbol{\alpha}_{\zeta_I} d\boldsymbol\xi .
\end{align*}

\noindent{\bf {Step 1: Estimation of the integral over $D_H\setminus D_{\delta_0}$.}}\\
From (\ref{defG}), on $D_H\setminus D_{\delta_0}$ we have
\begin{align*}
\im G^{\boldsymbol  E_I}(\boldsymbol \alpha_{\zeta_I}, \boldsymbol \xi) 
&=  \im \tilde U(\boldsymbol \alpha_{\zeta_I}, \boldsymbol \xi) ,
\end{align*}
where $\tilde U(\boldsymbol \alpha_{\zeta_I}, \boldsymbol \xi) $ is defined in (\ref{deftU}).
By assumption (2), we can find $\epsilon>0$ such that 
\begin{align*}
\lt|\int_{D_{H} \setminus D_{\delta_0}} 
 \phi_r\lt(\mathbf{s}^{\mathbf E_I}, \boldsymbol{\alpha}_{\zeta_I}, \boldsymbol\xi \rt)     
   e^{\frac{r}{4\pi \sqrt{-1}}G_r^{E_I}(\boldsymbol\alpha_{\xi_I},\boldsymbol\zeta)} d\boldsymbol{\alpha}_{\zeta_I} d\boldsymbol\xi\rt| = O\lt(e^{\frac{r}{4\pi}\Vol(M_{L_{\boldsymbol\theta}})-\epsilon}\rt).
\end{align*}

\noindent{\bf{Step 2: Deforming the integral over $D_{\delta_0}$.}}\\

Consider the surface $S^{\mathbf E_I} = S^{\mathbf E_I}_{\text{top}} \cup S^{\mathbf E_I}_{\text{bottom}}$ defined by
$$ S^{\mathbf E_I}_{\text{top}} = \{ (\boldsymbol \alpha_{\zeta_I}, \xi) \in D_{\delta_0, \CC} \mid \im  (\boldsymbol \alpha_{\zeta_I} , \boldsymbol \xi) = \im(\mathbf z^{\mathbf E_I})\}$$
and
$$ S^{\mathbf E_I}_{\text{side}} = \{ (\boldsymbol \alpha_{\zeta_I} , \boldsymbol \xi) + t\sqrt{-1} \im(\mathbf z^{ \mathbf E_I}) \mid (\boldsymbol \alpha_{\zeta_I}, \boldsymbol \xi) \in \partial D_{\delta}, t\in [0,1] )\}.$$
By the definition of the bump function $\psi$, on $D_{\delta_0}$ we have 
\begin{align}\label{IS1}
&\int_{D_{\delta_0}} 
 \phi_r\lt(\mathbf{s}^{\mathbf E_I}, \boldsymbol{\alpha}_{\zeta_I}, \boldsymbol\xi \rt)     
   e^{\frac{r}{4\pi \sqrt{-1}}G_r^{E_I}(\boldsymbol\alpha_{\xi_I},\boldsymbol\zeta)} d\boldsymbol{\alpha}_{\zeta_I} d\boldsymbol\xi \notag\\
= &
\int_{D_{\delta_0}} 
e^{ \sqrt{-1}\Big(\sum_{i\in I}\Big(\frac{p_i'}{q_i}(\beta_i - \pi) + \frac{p_i}{q_i}(\alpha_{i,\zeta_i} - \pi) + \frac{E_i(\alpha_{i,\zeta_i}+\beta_i-2\pi)}{q_i}\Big)+ \sum_{i\in I}a_{i,0}\beta_i + \sum_{i\in I} \Big(\frac{\iota_i}{2}\Big)\alpha_{i,\zeta_i} + \sum_{j\in J}\Big(a_{j,0}+\frac{\iota_j}{2}\Big)\alpha_j \Big)} \notag\\
&\qquad\qquad\times e^{\frac{r}{4\pi \sqrt{-1}}G_r^{E_I}(\boldsymbol\alpha_{\xi_I},\boldsymbol\zeta)} d\boldsymbol{\alpha}_{\zeta_I} d\boldsymbol\xi \notag\\
=&
\int_{S^{\mathbf E_I}} 
e^{ \sqrt{-1}\Big(\sum_{i\in I}\Big(\frac{p_i'}{q_i}(\beta_i - \pi) + \frac{p_i}{q_i}(\alpha_{i,\zeta_i} - \pi) + \frac{E_i(\alpha_{i,\zeta_i}+\beta_i-2\pi)}{q_i}\Big)+ \sum_{i\in I}a_{i,0}\beta_i + \sum_{i\in I} \Big(\frac{\iota_i}{2}\Big)\alpha_{i,\zeta_i} + \sum_{j\in J}\Big(a_{j,0}+\frac{\iota_j}{2}\Big)\alpha_j \Big)} \notag\\
&\qquad\qquad\times e^{\frac{r}{4\pi \sqrt{-1}}G_r^{E_I}(\boldsymbol\alpha_{\xi_I},\boldsymbol\zeta)} d\boldsymbol{\alpha}_{\zeta_I} d\boldsymbol\xi ,
\end{align}
where the last equality follows from the analyticity of the integrand and $\partial D_{\delta_0} = \partial S^{\mathbf E_I}$.

By Lemma \ref{qdtod}, we have
\begin{align}
&\int_{S^{\mathbf E_I}} 
e^{ \sqrt{-1}\Big(\sum_{i\in I}\Big(\frac{p_i'}{q_i}(\beta_i - \pi) + \frac{p_i}{q_i}(\alpha_{i,\zeta_i} - \pi) + \frac{E_i(\alpha_{i,\zeta_i}+\beta_i-2\pi)}{q_i}\Big)+ \sum_{i\in I}a_{i,0}\beta_i + \sum_{i\in I} \Big(\frac{\iota_i}{2}\Big)\alpha_{i,\zeta_i} + \sum_{j\in J}\Big(a_{j,0}+\frac{\iota_j}{2}\Big)\alpha_j \Big)} \notag\\
&\qquad\qquad\times e^{\frac{r}{4\pi \sqrt{-1}}G_r^{E_I}(\boldsymbol\alpha_{\xi_I},\boldsymbol\zeta)} d\boldsymbol{\alpha}_{\zeta_I} d\boldsymbol\xi \notag \\
=& \lt(\frac{r}{2}\rt)^{-c}\int_{S^{\mathbf E_I}} 
e^{ \sqrt{-1}\Big(\sum_{i\in I}\Big(\frac{p_i'}{q_i}(\beta_i - \pi) + \frac{p_i}{q_i}(\alpha_{i,\zeta_i} - \pi) + \frac{E_i(\alpha_{i,\zeta_i}+\beta_i-2\pi)}{q_i}\Big)+ \sum_{i\in I}a_{i,0}\beta_i + \sum_{i\in I} \Big(\frac{\iota_i}{2}\Big)\alpha_{i,\zeta_i} + \sum_{j\in J}\Big(a_{j,0}+\frac{\iota_j}{2}\Big)\alpha_j \Big)} \notag\\
&\qquad\qquad\qquad\qquad\times e^{\kappa(\boldsymbol \alpha_{\zeta_I}, \boldsymbol \xi) + \frac{r}{4\pi \sqrt{-1}}\Big(G^{E_I}(\boldsymbol\alpha_{\xi_I},\boldsymbol\zeta) + \frac{\upsilon_r(\boldsymbol \alpha_{\zeta_I}, \boldsymbol \xi) }{r^2}\Big)} d\boldsymbol{\alpha}_{\zeta_I} d\boldsymbol\xi  .\label{IS3}
\end{align}

\noindent{\bf{Step 3: Verification of the conditions in Proposition \ref{saddle}.}}\\

Now, we apply Proposition~\ref{saddle} to the integral in (\ref{IS3}). We check the conditions (1)-(6) below:
\begin{enumerate}
\item By the definition of $S^{\mathbf E_I}_{\text{top}}$ and Proposition \ref{prop52}, we have $\mathbf z^{\mathbf E_I} \in S^{\mathbf E_I}_{\text{top}}$.
\item On $S^{\mathbf E_I}_{\text{top}}$, by Proposition \ref{prop53}, since $\im G^{\mathbf E_I}(\boldsymbol \alpha_{\zeta_I}, \boldsymbol \xi)$ is strictly concave down in $\{\Re (\alpha_{i,\zeta_i})\}_{i \in I}$  and $\{\Re (\xi_s)\}_{s=1}^c$, $\im G^{\mathbf E_I}(\boldsymbol \alpha_{\zeta_I}, \boldsymbol \xi)$ attains its unique maximum at $\mathbf z^{\mathbf E_I}$.

On $S^{\mathbf E_I}_{\text{sides}}$, by Proposition \ref{prop53}, since $\im G^{\mathbf E_I}$ is strictly concave up in $\{\im (\alpha_{i,\zeta_i})\}_{i \in I}$  and $\{\im (\xi_s)\}_{s=1}^c$, for each $(\boldsymbol \alpha_{\zeta_I}, \boldsymbol \xi)\in \partial D_{\delta_0}$ and $t\in [0,1]$ we have
$$ \im G^{\mathbf E_I}( (\boldsymbol \alpha_{\zeta_I}, \boldsymbol \xi) + t\sqrt{-1} \im(\mathbf z^{\mathbf E_I}) )
< \max\{ \im G^{\mathbf E_I} (\boldsymbol \alpha_{\zeta_I}, \boldsymbol \xi),
\im G^{\mathbf E_I}( (\boldsymbol \alpha_{\zeta_I}, \boldsymbol \xi) + \sqrt{-1} \im(\mathbf z^{\mathbf E_I}) )
\}.
$$ 
For $(\boldsymbol \alpha_{\zeta_I}, \boldsymbol\xi) \in \partial D_{\delta_0}$, by assumption (2) we have 
$$  \im G^{\mathbf E_I} (\boldsymbol \alpha_{\zeta_I}, \boldsymbol \xi) < \im G^{\mathbf E_I}(\mathbf z^{\mathbf E_I}) .$$
For $(\boldsymbol \alpha_{\zeta_I}, \boldsymbol \xi) + \sqrt{-1} \im(\mathbf z^{\mathbf E_I})  \in S^{\mathbf E_I}_{\text{top}}$, since on $S^{\mathbf E_I}_{\text{top}}$ the function $\im G^{\mathbf E_I}$ attains its maximum at $\mathbf z^{\mathbf E_I}$, we have
$$  \im G^{\mathbf E_I}( (\boldsymbol \alpha_{\zeta_I}, \boldsymbol \xi) + \sqrt{-1} \im(\mathbf z^{\mathbf E_I}) ) < \im G^{\mathbf E_I}(\mathbf z^{\mathbf E_I}) .$$
Altogether, on $S^{\mathbf E_I}$, $\im G^{\mathbf E_I}$ has a unique maximum at $\mathbf z^{\mathbf E_I}$.
\item For any $k\in \NN$ and any $k$-tuple of complex number $(z_1,\dots,z_k)\in \CC^k$, we let 
$$\Re(z_1,\dots,z_k) = (\Re z_1,\dots, \Re z_k) \in \RR^k,$$ 
where $\Re z_i$ is the real part of $z_i$ for $i=1,\dots, k$. For any $(\boldsymbol\alpha_{\xi_I}, \boldsymbol \xi) \in D_{\delta_0}$, we consider the set
\begin{equation*}
P_{(\boldsymbol\alpha_{\xi_I}, \boldsymbol \xi)} = 
\left\{
  (\boldsymbol{\tilde\alpha_{\xi_I}}, \boldsymbol{\tilde\xi}) \in D_{\delta_0,\CC} \;\middle|\;
  \begin{aligned}
  & \Re(\boldsymbol{\tilde\alpha}_{\xi_I}, \boldsymbol{\tilde\xi})=\Re(\boldsymbol\alpha_{\xi_I}, \boldsymbol \xi),\\
  &  \im G^{\mathbf E_I} (\boldsymbol{\tilde\alpha}_{\xi_I}, \boldsymbol{\tilde\xi}) < \im G^{\mathbf E_I}(\mathbf z^{\mathbf E_I}) 
  \end{aligned}
\right\}.
\end{equation*}
Note that for $(\boldsymbol\alpha_{\xi_I}, \boldsymbol \xi) = \mathbf z^{\mathbf E_I}$, since it is a critical point of $G^{\mathbf E_I} (\boldsymbol \alpha_{\zeta_I}, \boldsymbol \xi)$, by the Cauchy-Riemann equation we know that
$$\frac{\partial}{\partial \im \alpha_{i,\zeta}}\im G^{\mathbf E_I} (\boldsymbol \alpha_{\zeta_I}, \boldsymbol \xi) 
= \frac{\partial}{\partial \im \xi_k}\im G^{\mathbf E_I} (\boldsymbol \alpha_{\zeta_I}, \boldsymbol \xi) = 0$$
for $i\in I$ and $k=1,\dots, c$.  
By Proposition \ref{prop53}, since $\im G^{\mathbf E_I} (\boldsymbol \alpha_{\zeta_I}, \boldsymbol \xi)$ is strictly concave up in $\{\im (\alpha_{i,\zeta_i})\}_{i \in I}$  and $\{\im (\xi_s)\}_{s=1}^c$ on $D_{\delta_0, \mathbb{C}}$, we know that $P_{\mathbf z^{\mathbf E_I}}$ is an empty set.

Next, for any $(\boldsymbol \alpha_{\zeta_I}, \boldsymbol \xi) \in S^{\mathbf E_I}_{\text{top}}$, by Proposition \ref{prop53}, we know that $(\boldsymbol \alpha_{\zeta_I}, \boldsymbol \xi)  \in P_{(\boldsymbol \alpha_{\zeta_I}, \boldsymbol \xi)}$. Moreover, by Proposition \ref{prop53}, since $\im G^{\mathbf E_I} (\boldsymbol \alpha_{\zeta_I}, \boldsymbol \xi)$ is strictly concave up in $\{\im (\alpha_{i,\zeta_i})\}_{i \in I}$  and $\{\im (\xi_s)\}_{s=1}^c$ on $D_{\delta_0, \mathbb{C}}$, $P_{(\boldsymbol \alpha_{\zeta_I}, \boldsymbol \xi)}$ is a convex set. This implies that each $P_{(\boldsymbol \alpha_{\zeta_I}, \boldsymbol \xi)}$ with $(\boldsymbol \alpha_{\zeta_I}, \boldsymbol \xi)\in S^{\mathbf E_I}_{\text{top}}$ is a topological $(|I|+c)$-dimensional disk which admits a deformation retract to the point $(\boldsymbol \alpha_{\zeta_I}, \boldsymbol \xi)$. This verifies condition (3).
\item By continuity and compactness of $S^{\mathbf E_I}$,
\begin{align*}
\lt| e^{ \sqrt{-1} \lt(\sum_{i \in I} a_{i,0} \beta_i 
    + \sum_{i \in I} \lt(a_{i,\zeta_i} + \frac{\iota_i}{2}\rt) \alpha_{i,\zeta_i} 
    + \sum_{j \in J} \lt(a_{j,0} + \frac{\iota_j}{2}\rt) \alpha_j 
    \sum_{i\in I}  \lt(\lt(\frac{p_i'}{q_i} \rt) (\beta_i - \pi) + \frac{E_i(\alpha_{i,\zeta_i} + \beta_i - 2\pi)}{q_i} \rt)
    \rt) + \kappa(\boldsymbol \alpha_{\zeta_I}, \boldsymbol\xi)} \rt| 
\end{align*}
is non-zero and bounded below by a positive constant independent of $r$.

\item By Lemma \ref{qdtod}, $|\upsilon_r(\boldsymbol \alpha_{\zeta_I}, \boldsymbol \xi)|$ is bounded from above by a constant independent of $r$ on any compact subset of $D_{H,\CC}$. 
\item Note that 
$$ \lim_{r\to \infty} \Hess G_r^{\mathbf E_I}(\mathbf z^{\mathbf E_I}) = 
\Hess G^{\mathbf E_I}(\mathbf z^{\mathbf E_I}).
$$
By Proposition \ref{prop54}, the Hessian matrix $\Hess G^{\mathbf E_I}(\mathbf z^{\mathbf E_I})$ is non-singular. By continuity, the Hessian matrix $\Hess G_r^{\mathbf E_I}(\mathbf z^{\mathbf E_I})$ is non-singular

\end{enumerate}
The result then follows from Proposition \ref{saddle}.
\end{proof}

\subsection{Reidemeister torsion}
The goal of this subsection is to prove Proposition \ref{Rtorsion} and \ref{0FCasym}, which relates the asymptotics of the leading Fourier coefficients obtained in Proposition \ref{lfc} with the adjoint twisted Reideimester torsion of the cone manifold $M\setminus L$.

\begin{proposition}\label{Rtorsion} Consider the system of meridians $\mathrm{\bf \Upsilon}=(\Upsilon_1,\dots,\Upsilon_n)$ with $\mathrm{ \Upsilon}_i = p_iu_i + q_iv_i$ and $\mathrm{ \Upsilon}_j = u_j$. Let $\mathbb{T}_{(M\setminus L,\mathbf \Upsilon)}([\rho_{M_{L_{\boldsymbol\theta}}}])$ be the Reideimester torsion of $M\setminus L$  twisted by the adjoint action of  $\rho_{M^{(r)}}$ with respect to the system of meridians $\mathrm{\bf \Upsilon}$. Then we have
$$\frac{C^{\mathbf E_I}(\mathbf{z}^{\boldsymbol E_{I}})}{\sqrt{- \lt(\prod_{i\in I}q_i \rt)\det \Hess({G}^{\mathbf E_I} )(\mathbf{z^{\mathbf E_I}})}}
= \frac{e^{\bigg(\sum_{i\in I} \Big(a_{i,0}+\frac{\iota_i}{2}\Big)+\sum_{j\in J}\Big(a_{j,0}+\frac{\iota_j}{2}\Big)\bigg)\sqrt{-1}\pi + \frac{1}{2}\sum_{k=1}^n\mu_k \mathrm{H}(\gamma_k)}}{2^{\frac{|I|+c}{2}}\sqrt{\pm\mathbb{T}_{(M\setminus L,\mathrm{\bf \Upsilon})}([\rho_{M_{L_{\boldsymbol\theta}}}])}}.$$
\end{proposition}
For  each $s\in\{1,\dots,c\},$ we let $I_s=\{s_1,\dots,s_6\}\cap I$ and $J_s=\{s_1,\dots,s_6\}\cap J.$ We also let $\alpha_{I_s}^*=(\alpha^*_{s_i})_{s_i\in I_s},$ $\alpha_{J_s}=(\alpha_{s_j})_{j\in J},$ $\xi_s^*=\xi(\alpha_{I_s}^*,\alpha_{J_s})$ and $z_s^*=(\alpha_{I_s}^*,\alpha_{J_s},\xi_s^*).$ 

To prove Proposition \ref{Rtorsion}, we need Lemmas \ref{L3.1}, \ref{L3.2}, \ref{L3.4} and \ref{L3.5}. 

\begin{lemma}\label{L3.1}  For each $i\in I,$ consider the system of meridian $\Upsilon_i=p_iu_i+ q_iv_i$. Then 
$$- \lt(\prod_{i\in I} q_i \rt)\det\mathrm{Hess}\text{ } G^{\mathbf E_I} (\mathbf z^{\mathbf E_I})=-(-2)^{|I|}\det\bigg(\frac{\partial \mathrm H(\Upsilon_{i_1})}{\partial \mathrm H(u_{i_2})}\bigg)_{i_1,i_2\in I} \prod _{s=1}^{c} \frac{\partial ^2U}{\partial \xi_s^2}\bigg|_{z^*_s}.$$
\end{lemma}

\begin{proof} The proof is similar to the proof of Lemma 3.3 in \cite{WY4}. 
For $s\in\{1,\dots,c\}$ and $i\in I,$ we denote by $s\sim i$ if the tetrahedron $\Delta_s$ intersects the component $L_{\text{FSL},i}$ of $L_{\text{FSL}},$  and for $\{i_1, i_2\}\subset I$ we denote by $s\sim i_1,i_2$ if $\Delta_s$ intersects both $L_{\text{FSL},i_1}$ and $L_{\text{FSL},i_2}.$ For $s\in\{1,\dots,c\},$ let $\alpha_s=(\alpha_{s_1},\dots,\alpha_{s_6})$ and let $\alpha_s^*=(\alpha_{I_s}^*,\alpha_{J_s}).$ The following claims (1)-(3) are from \cite[Lemma 3.3]{WY4} and we include the proof below for reader's convenience. Claims (4)-(5) can be proved by suitably modifying the proof of (4)-(5) in \cite[Lemma 3.3]{WY4}.
\begin{enumerate}[(1)]
\item For $s\in\{1,\dots,c\},$
$$\frac{\partial^2 G^{\mathbf E_I}}{\partial \xi_s^2}\bigg|_{\mathbf z^{\mathbf E_I}}=\frac{\partial^2 U}{\partial \xi_s^2}\bigg|_{z_s^*}.$$

\item For $\{s_1,s_2\}\subset\{1,\dots,c\},$
$$\frac{\partial^2 G^{\mathbf E_I}}{\partial \xi_{s_1}\partial \xi_{s_2}}\bigg|_{\mathbf z^{\mathbf E_I}}=0.$$

\item For $i\in I$ and $s\in\{1,\dots,c\},$
$$\frac{\partial^2 G^{\mathbf E_I}}{\partial \alpha_i\partial \xi_s}\bigg|_{\mathbf z^{ \mathbf E_I}}=-\frac{\partial^2 U}{\partial \xi_s^2}\bigg|_{z^*_s} \frac{\xi_s(\alpha_s)}{\partial \alpha_i}\bigg|_{\alpha_s^*}.$$

\item For $i\in I,$ 
$$\frac{\partial^2 G^{\mathbf E_I}}{\partial \alpha_i^2}\bigg|_{\mathbf z^{\mathbf E_I}}=-\frac{2}{q_i}\frac{\partial \mathrm H(\Upsilon_i)}{\partial \mathrm H(u_i)}+\sum_{s\sim i}\frac{\partial^2 U}{\partial \xi_s^2}\bigg|_{z^*_s} \bigg(\frac{\xi_s(\alpha_s)}{\partial \alpha_i}\bigg|_{\alpha_s^*}\bigg)^2.$$

\item For $\{i_1,i_2\}\subset I,$ 
$$\frac{\partial^2 G^{\mathbf E_I}}{\partial \alpha_{i_1}\partial \alpha_{i_2}}\bigg|_{\mathbf z^{\mathbf E_I}}=-\frac{2}{q_{i_1}}\frac{E_{i_1}\mu_{i_1}}{E_{i_2}\mu_{i_2}} \frac{\partial \mathrm H(\Upsilon_{i_1})}{\partial \mathrm H(u_{i_2})}+\sum_{s\sim i_1,i_2}\frac{\partial^2 U}{\partial \xi_s^2}\bigg|_{z^*_s} \frac{\xi_s(\alpha_s)}{\partial \alpha_{i_1}}\bigg|_{\alpha_s^*} \frac{\xi_s(\alpha_s)}{\partial \alpha_{i_2}}\bigg|_{\alpha_s^*}.$$
\end{enumerate}
Assuming these claims, then 
\begin{equation}\label{congurent0}
\mathrm{Hess}G^{\mathbf E_I}(\mathbf z^{{\mathbf E_I}})=A\cdot D\cdot A^T,
\end{equation}
with $D$ and $A$ defined as follows. 
The matrix $D$ is a block matrix with the left-top block the $|I|\times|I|$ matrix 
$$\bigg(-\frac{2}{q_{i_1}}\frac{E_{i_1}\mu_{i_1}}{E_{i_2}\mu_{i_2}} \frac{\partial \mathrm H(\Upsilon_{i_1})}{\partial \mathrm H(u_{i_2})}\bigg)_{i_1,i_2\in I},$$
the right-top and the left-bottom blocks consisting of $0$'s, and the right-bottom block the $c \times c$ diagonal matrix with the diagonal entries
$\frac{\partial^2 U}{\partial \xi_1^2}\Big|_{z_1^*},\dots, \frac{\partial^2 U}{\partial \xi_c^2}\Big|_{z_c^*}.$
Then 
\begin{equation}\label{detD0}
\begin{split}
\det D=&\frac{(-2)^{|I|}}{\prod_{i\in I} q_i}\det\bigg(\frac{E_{i_1}\mu_{i_1}}{E_{i_2}\mu_{i_2}} \frac{\partial \mathrm H(\Upsilon_{i_1})}{\partial \mathrm H(u_{i_2})}\bigg)_{i_1,i_2\in I} \prod _{s=1}^{c} \frac{\partial ^2U}{\partial \xi_s^2}\bigg|_{z^*_s}\\
=&\frac{(-2)^{|I|}}{\prod_{i\in I} q_i}\det\bigg(\frac{\partial \mathrm H(\Upsilon_{i_1})}{\partial \mathrm H(u_{i_2})}\bigg)_{i_1,i_2\in I}  \prod _{s=1}^c \frac{\partial ^2U}{\partial \xi_s^2}\bigg|_{z^*_s}.
\end{split}
\end{equation}
The matrix $A$ is a block matrix with the left-top and the right-bottom blocks respectively the $|I|\times|I|$ and $c\times c$ identity matrices, the left-bottom block consisting of $0$'s and the right-top block the $|I|\times c$ matrix with entries $a_{is},$ $i\in I$ and $s\in\{1,\dots,c\},$ given by
$$a_{is}=-\frac{\xi_s(\alpha_s)}{\partial \alpha_i}\bigg|_{\alpha_s^*}$$
if $s\sim i$ and $a_{is}=0$
if otherwise. Since $A$ is upper triangular with all diagonal entries equal to $1,$ 
\begin{equation}\label{detA0}
\det A=1.
\end{equation}
The result then follows from (\ref{congurent0}), (\ref{detD0}) and (\ref{detA0}), and we are left to prove the claims (1) -- (5). 

Claims (1) and (2) are straightforward  from the definition of $G^{\mathbf E_I}.$ For (3), we have
\begin{equation}\label{WU0}
\frac{\partial G^{\mathbf E_I}}{\partial \xi_s}\bigg|_{\big((\alpha_i)_{i\in I},\xi_1,\dots,\xi_c\big)}=\frac{\partial U}{\partial \xi_s}\bigg|_{(\alpha_s,\xi_s)}.
\end{equation}
Let 
$$f(\alpha_s,\xi_s)\doteq \frac{\partial U}{\partial \xi_s}\bigg|_{(\alpha_s,\xi_s)}$$
and 
$$g(\alpha_s)\doteq f(\alpha_s,\xi_s(\alpha_s)).$$
Then 
$$g(\alpha_s)=\frac{\partial U}{\partial \xi_s}\bigg|_{(\alpha_s,\xi_s(\alpha_s))}=\frac{dU_{\alpha_s}}{d \xi_s}\bigg|_{\xi_s(\alpha_s)}\equiv 0,$$
and hence
\begin{equation}\label{g=00}
\frac{\partial g}{\partial \alpha_{s_i}}\bigg|_{\alpha_s}=0.
\end{equation}
On the other hand, we have
\begin{equation}\label{+0}
\begin{split}
\frac{\partial g}{\partial \alpha_{s_i}}\bigg|_{\alpha_s}=&\frac{\partial f}{\partial \alpha_{s_i}}\bigg|_{(\alpha_s,\xi_s(\alpha_s))}+\frac{\partial f}{\partial \xi_s}\bigg|_{(\alpha_s,\xi_s(\alpha_s))}  \frac{\partial \xi_s(\alpha_s)}{\partial \alpha_{s_i}}\bigg|_{\alpha_s}\\
=&\frac{\partial^2 U}{\partial \alpha_{s_i}\partial \xi_s}\bigg|_{(\alpha_s,\xi_s(\alpha_s))}+\frac{\partial^2 U}{\partial \xi_s^2}\bigg|_{(\alpha_s,\xi_s(\alpha_s))} \frac{\partial \xi_s(\alpha_s)}{\partial \alpha_{s_i}}\bigg|_{\alpha_s}.
\end{split}
\end{equation}
Putting (\ref{g=00}) and (\ref{+0}) together, we have
\begin{equation}\label{alphaxi0}
\frac{\partial^2 U}{\partial \alpha_{s_i}\partial \xi_s}\bigg|_{(\alpha_s,\xi_s(\alpha_s))}=-\frac{\partial^2 U}{\partial \xi_s^2}\bigg|_{(\alpha_s,\xi_s(\alpha_s))} \frac{\partial \xi_s(\alpha_s)}{\partial \alpha_{s_i}}\bigg|_{\alpha_s},
\end{equation}
and (3) follows from (\ref{WU0}) and (\ref{alphaxi0}).

For (4) and (5), we have 
\begin{equation}\label{3.10}
\frac{\partial ^2G^{\mathbf E_I}}{\partial \alpha_i^2}\bigg|_{\mathbf z^{\mathbf E_I}}= -\frac{2p_i}{q_i}-\iota_i+\sum_{s\sim i}\frac{\partial ^2U}{\partial \alpha_i^2}\bigg|_{z^*_s},
\end{equation}
and
\begin{equation}\label{3.11}
\frac{\partial ^2G^{\mathbf E_I}}{\partial \alpha_{i_1}\partial \alpha_{i_2}}\bigg|_{\mathbf z^{\mathbf E_I}}=\sum_{s\sim i_1,i_2}\frac{\partial ^2U}{\partial \alpha_{i_1}\partial \alpha_{i_2}}\bigg|_{z^*_s}.
\end{equation}
Let $W$ be the function defined in (\ref{W}). By the Chain Rule and (\ref{xia}), we have
$$\frac{\partial U}{\partial \xi_s}\bigg|_{(\alpha_s,\xi_s(\alpha_s))}=\frac{d U_{\alpha_s}}{d \xi_s}\bigg|_{\xi_s(\alpha_s)} =0,$$
and hence for $j\in\{1,\dots,6\},$
$$\frac{\partial W}{\partial \alpha_{s_j}}\bigg|_{\alpha_s}=\frac{\partial U}{\partial \alpha_{s_j}}\bigg|_{\alpha_s}+\frac{\partial U}{\partial \xi_s}\bigg|_{(\alpha_s,\xi_s(\alpha_s))}\frac{\partial \xi_s(\alpha_s)}{\partial \alpha_{s_j}}\bigg|_{\alpha_s}=\frac{\partial U}{\partial \alpha_{s_j}}\bigg|_{\alpha_s}.$$
Then using the Chain Rule again, for  $j,k\in\{1,\dots,6\}$ we have
\begin{equation*}\label{3.12}
\begin{split}
\frac{\partial^2 W}{\partial \alpha_{s_j}\partial \alpha_{s_k}}\bigg|_{\alpha_s}
=\frac{\partial ^2U}{\partial \alpha_{s_j}\partial \alpha_{s_k}}\bigg|_{(\alpha_s,\xi_s(\alpha_s))}+\frac{\partial^2 U}{\partial \alpha_{s_k}\partial \xi_s}\bigg|_{(\alpha_s,\xi_s(\alpha_s))}  \frac{\partial \xi_s(\alpha_s)}{\partial \alpha_{s_j}}\bigg|_{\alpha_s}.
\end{split}
\end{equation*}
Together with (\ref{alphaxi0}), for $j,k\in\{1,\dots, 6\}$ we have
\begin{equation}\label{3.13}
\begin{split}
\frac{\partial^2 U}{\partial \alpha_{s_j}\partial \alpha_{s_k}}\bigg|_{(\alpha_s,\xi_s(\alpha_s))}=&\frac{\partial^2 W}{\partial \alpha_{s_j}\partial \alpha_{s_k}}\bigg|_{\alpha_s}-\frac{\partial^2 U}{\partial \alpha_{s_k}\partial \xi_s}\bigg|_{(\alpha_s,\xi_s(\alpha_s))}  \frac{\partial \xi_s(\alpha_s)}{\partial \alpha_{s_j}}\bigg|_{\alpha_s}\\
=&\frac{\partial^2 W}{\partial \alpha_{s_j}\partial \alpha_{s_k}}\bigg|_{\alpha_s}+\frac{\partial^2 U}{\partial \xi_s^2}\bigg|_{(\alpha_s,\xi_s(\alpha_s))} \frac{\partial \xi_s(\alpha_s)}{\partial \alpha_{s_j}}\bigg|_{\alpha_s}  \frac{\partial \xi_s(\alpha_s)}{\partial \alpha_{s_k}}\bigg|_{\alpha_s}.
\end{split}
\end{equation}
By (\ref{3.10}), (\ref{3.11}) and (\ref{3.13}) and we have
\begin{equation}\label{3.15}
\begin{split}
\frac{\partial ^2G^{\mathbf E_I}}{\partial \alpha_i^2}\bigg|_{\mathbf z^{\mathbf E_I}}
=&-\frac{2p_i}{q_i}-\iota_i+\sum_{s\sim i}\sum_{s_k}\frac{\partial^2 W}{\partial \alpha_{s_k}^2}\bigg|_{\alpha^*_s}+\sum_{s\sim i}\frac{\partial^2 U}{\partial \xi_s^2}\bigg|_{z^*_s} \bigg(\frac{\xi_s(\alpha_s)}{\partial \alpha_i}\bigg|_{\alpha_s^*}\bigg)^2,
\end{split}
\end{equation}
where the second sum in the third term of the right hand side is over $s_k$ such that the edge $e_{s_k}$ in $\Delta_s$ intersects the component  $L_{\text{FSL},i};$ and
\begin{equation}\label{3.17}
\begin{split}
\frac{\partial ^2G^{\mathbf E_I}}{\partial \alpha_{i_1}\partial \alpha_{i_2}}\bigg|_{\mathbf z^{\mathbf E_I}}=&\sum_{s\sim i_1,i_2}\sum_{s_j,s_k}\frac{\partial^2 W}{\partial \alpha_{s_j}\partial \alpha_{s_k}}\bigg|_{\alpha^*_s}+\sum_{s\sim i_1,i_2}\frac{\partial^2 U}{\partial \xi_s^2}\bigg|_{z^*_s} \frac{\xi_s(\alpha_s)}{\partial \alpha_{i_1}}\bigg|_{\alpha_s^*} \frac{\xi_s(\alpha_s)}{\partial \alpha_{i_2}}\bigg|_{\alpha_s^*},
\end{split}
\end{equation}
where the second sum in the first term of the right hand side is over $s_j,s_k$ such that the edge $e_{s_j}$ in $\Delta_s$ intersects the component  $L_{\text{FSL},i_1}$ and the edge $e_{s_k}$ in $\Delta_s$ intersects the component  $L_{\text{FSL},i_2}.$

At a hyperbolic cone metric on $M_c$ with singular locus $L_{\text{FSL}},$ by Theorem \ref{co-vol},  for $i,j\in\{1,\dots, 6\}$ we have
\begin{equation}\label{3.18}
\frac{\partial^2 W}{\partial \alpha_{s_i}\partial \alpha_{s_j}}\bigg|_{\alpha^*_s}=-\sqrt{-1}\frac{E_{s_i}\mu_{s_i}}{E_{s_j}\mu_{s_j}} \frac{\partial l_{s_i}}{\partial \theta_{s_j}},
\end{equation}
where $l_{s_k}$ is the length of $e_{s_k}$ of $\Delta_s,$ and  if  $e_{s_k}$ intersects $L_{\text{FSL},i}$ then $E_{s_k}=E_i,$ $\mu_{s_k}=\mu_i$ and $\theta_{s_k}=\frac{\theta_i}{2}$ is the half of the cone angle at $L_{\text{FSL},i}.$ We also observe that that
\begin{equation}\label{3.14}
l_i=\sum_{s\sim i}\sum_{s_k}l_{s_k},
\end{equation}
where the second sum is over $s_k$ such that the edge $e_{s_k}$ in $\Delta_s$ intersects the component  $L_{\text{FSL},i}.$  

Then by (\ref{3.18}), (\ref{3.14}) (\ref{m}) and (\ref{l}) we have
\begin{equation}\label{3.16}
\begin{split}
-\frac{2p_i}{q_i}-\iota_i+\sum_{s\sim i}\sum_{s_k}\frac{\partial^2 W}{\partial \alpha_{s_k}^2}\bigg|_{\alpha^*_s}
=&-\frac{2p_i}{q_i}-\iota_i-\sqrt{-1}\sum_{s\sim i}\sum_{s_k}\frac{\partial l_{s_k}}{\partial \theta_{s_k}}\\
=&-\frac{2p_i}{q_i}-\iota_i-2\sqrt{-1}\frac{\partial l_i}{\partial \theta_i}\\
=&-\frac{2}{q_i}\frac{\partial \big(p_i\sqrt{-1}\theta_i - q_i l_i+\frac{q_i\iota_i}{2}\sqrt{-1}\theta_i \big)}{\partial (\sqrt{-1}\theta_i)}\\
=&-\frac{2}{q_i}\frac{\partial \big(p_i\mathrm H(u_i) + q_i\mathrm H(v_i)\big)}{\partial (\sqrt{-1}\theta_i)}
=-\frac{2}{q_i}\frac{\partial \mathrm H(\Upsilon_i)}{\partial \mathrm H(u_i)}.
\end{split}
\end{equation}
From (\ref{3.15}) and (\ref{3.16}), (4) holds at  hyperbolic cone metrics on $M_c$ with singular locus $L_{\text{FSL}}.$ By the analyticity of the involved functions (see for e.g. \cite[Lemma 4.2]{WY2}),  (\ref{3.16}) still holds in a neighborhood of the complete hyperbolic structure on $M_c\setminus L_{\text{FSL}},$ from which (4) follows. 

By (\ref{3.18}), (\ref{3.14}), (\ref{m}) and (\ref{l}) we have
\begin{equation}\label{3.19}
\begin{split}
\sum_{s\sim i_1,i_2}\sum_{s_j,s_k}\frac{\partial^2 W}{\partial \alpha_{s_j}\partial \alpha_{s_k}}\bigg|_{\alpha^*_s}=&-2\sqrt{-1}\sum_{s\sim i_1}\sum_{s_j}\frac{E_{s_{j}}\mu_{s_{j}}}{E_{{i_2}}\mu_{{i_2}}}\frac{\partial l_{s_j}}{\partial \theta_{i_2}}\\
=&-2\sqrt{-1}\frac{E_{i_1}\mu_{i_1}}{E_{{i_2}}\mu_{{i_2}}}\frac{\partial l_{i_1}}{\partial \theta_{i_2}}\\
=&-2\frac{E_{i_1}\mu_{i_1}}{E_{{i_2}}\mu_{{i_2}}}\frac{\partial \big( \frac{p_{i_1}}{q_{i_1}}\sqrt{-1}\theta_{i_1}-l_{i_1}+\frac{\iota_{i_1}}{2}\sqrt{-1}\theta_{i_1}\big)}{\partial (\sqrt{-1}\theta_{i_2})}\\
=&-\frac{2}{q_{i_1}}\frac{E_{i_1}\mu_{i_1}}{E_{{i_2}}\mu_{{i_2}}}\frac{\partial \big(p_{i_1}\mathrm H(u_{i_1}) + q_{i_1}\mathrm H(v_{i_1})\big)}{\partial (\sqrt{-1}\theta_{i_2})}\\
=&-\frac{2}{q_{i_1}}\frac{E_{i_1}\mu_{i_1}}{E_{{i_2}}\mu_{{i_2}}}\frac{\partial \mathrm H(\Upsilon_{i_1})}{\partial \mathrm H(u_{i_2})},
\end{split}
\end{equation}
where the second sum on the left hand side is over $s_j,s_k$ such that the edge $e_{s_j}$ in $\Delta_s$ intersects the component  $L_{i_1}$ and the edge $e_{s_k}$ in $\Delta_s$ intersects the component $L_{i_2},$ the second sum on the right hand side of the first equation is over $s_j$ such that the edge $e_{s_j}$ in $\Delta_s$ intersects the component  $L_{i_1},$ and the third equality comes from the fact that
$$\frac{\partial (\sqrt{-1}\theta_{i_1})}{\partial (\sqrt{-1}\theta_{i_2})}=\frac{\partial \mathrm H(u_{i_1})}{\partial \mathrm H(u_{i_2})}=0.$$
 From (\ref{3.17}) and (\ref{3.19}), (5) holds at  hyperbolic cone metrics on $M_c\setminus L_{\text{FSL}}.$ By the analyticity of the involved functions, (\ref{3.19}) still holds in a neighborhood of the complete hyperbolic structure on $M_c\setminus L_{\text{FSL}},$ from which (5) follows. 
\end{proof}

The following lemma is from \cite{WY4}.
\begin{lemma}\label{L3.2} (\cite[Lemma 3.4]{WY4})
\begin{equation*}
\begin{split}
\kappa(\mathbf z^{\mathbf E_I})=&-\frac{\sqrt{-1}}{2}\sum_{k=1}^6\frac{\partial U}{\partial \alpha_{s_k}}\bigg|_{z^*_s}\\
&-\frac{\sqrt{-1}}{2}\sum_{s_i\in I_s}\alpha^*_{s_i}-\frac{\sqrt{-1}}{2}\sum_{s_j\in J_s}\alpha_{s_j}+2\sqrt{-1}\xi_s^*-\frac{1}{2}\sum_{i=1}^4\log\big(1-e^{2\sqrt{-1}(\xi_s^*-\tau^*_{s_i})}\big),
\end{split}
\end{equation*}
where with the notation that $\alpha^*_{s_j}=\alpha_{s_j}$ for $s_j\in J_s,$
$\tau^*_{s_1}=\frac{\alpha^*_{s_1}+\alpha^*_{s_2}+\alpha^*_{s_3}}{2},$ $\tau^*_{s_2}=\frac{\alpha^*_{s_1}+\alpha^*_{s_5}+\alpha^*_{s_6}}{2},$ $\tau^*_{s_3}=\frac{\alpha^*_{s_2}+\alpha^*_{s_4}+\alpha^*_{s_6}}{2}$ and  $\tau^*_{s_4}=\frac{\alpha^*_{s_3}+\alpha^*_{s_4}+\alpha^*_{s_5}}{2}.$ 
\end{lemma}

The following lemma is an analogue of Lemma 3.5 in \cite{WY4}.
\begin{lemma}\label{L3.4} For $i\in I$ recall that $\gamma_i=(-q_i'u_i + p_i'v_i)+a_{i,0}(p_iu_i+q_iv_i)$ is the parallel of copy of $L_i$ given by the framing $a_{i,0},$ and for each $j\in J$ recall that $\gamma_j=a_{j,0}u_j+v_j$ is the parallel copy of $L_j$ given by the framing $a_{j,0}.$  Then
\begin{equation*}
\begin{split}
&\sqrt{-1} \lt[\sum_{i \in I} a_{i,0} \beta_i 
    + \sum_{i \in I} \lt( \frac{\iota_i}{2}\rt) \alpha_i^* 
    + \sum_{j \in J} \lt(a_{j,0} + \frac{\iota_j}{2}\rt) \alpha_j \rt.\\
&\qquad\qquad   \lt. + \sum_{i\in I}  \lt(\lt(\frac{p_i'}{q_i} \rt) (\beta_i - \pi) + \frac{p_i}{q_i}(\alpha_i^* - \pi) + \frac{E_i(\alpha_i^* + \beta_i - 2\pi)}{q_i} \rt)
    \rt]\\
& \qquad\qquad\qquad -\frac{\sqrt{-1}}{2}\sum_{s=1}^c\bigg(\sum_{k=1}^6\frac{\partial U}{\partial \alpha_{s_k}}\bigg|_{z_s^*}\bigg)\\
=&\bigg(\sum_{i\in I} \Big(a_{i,0}+\frac{\iota_i}{2}\Big)+\sum_{j\in J}\Big(a_{j,0}+\frac{\iota_j}{2}\Big)\bigg)\sqrt{-1}\pi+\frac{1}{2}\sum_{k=1}^n\mu_k\mathrm H(\gamma_k).
\end{split}
\end{equation*}
\end{lemma}

\begin{proof} 
We first prove the result for the case that $M_c$ is with a hyperbolic cone metric with singular locus $L_{\text{FSL}},$ 
\begin{equation}\label{L3.3}
-\frac{\sqrt{-1}}{2}\sum_{s=1}^c\bigg(\sum_{k=1}^6\frac{\partial U}{\partial \alpha_{s_k}}\bigg|_{z_s^*}\bigg)=-\frac{1}{2}\sum_{i\in I}E_i\mu_i\Big(\mathrm H(v_i)-\frac{\iota_i}{2}\mathrm H(u_i)\Big)-\frac{1}{2}\sum_{j\in J}\mu_jl_j.
\end{equation}
In this case, the hyperbolic cone manifold $M_c\setminus L_{\text{FSL}}$ is obtained by gluing hyperideal tetrahedra $\Delta_1,\dots,\Delta_s$ together along the hexagonal faces then taking the orientable double. For each $s\in \{1,\dots, c\}$ let $e_{s_1},\dots,e_{s_6}$ be the edges of $\Delta_s$ and for each $k\in\{1,\dots,6\}$ let $l_{s_k}$ and $\theta_{s_k}$ respectively be the length of and the dihedral angle at $e_{s_k}.$ If $e_{s_k}$ intersects the component $L_{\text{FSL},i}$ of $L_{\text{FSL}}$ for some $i\in I,$ then $\mathrm H(u_i)=\sqrt{-1} \theta_i=2\sqrt{-1}\theta_{s_k}$ and let $\alpha_{s_k}=\alpha^*_i=\pi+\frac{E_i\mu_i\sqrt{-1}\mathrm H(u_i)}{2}=\pi-E_{s_k}\mu_{s_k}\theta_{s_k},$ where $E_{s_k}=E_i$ and $\mu_{s_k}=\mu_i;$ and if $e_{s_k}$ intersects the component $L_{\text{FSL},j}$ of $L_{\text{FSL}}$ for some $j\in J,$ then $\theta_j=2\theta_{s_k}$ and let $\alpha_{s_k}=\alpha_j=\pi+\frac{\mu_j\theta_j}{2}=\pi+\mu_{s_k}\theta_{s_k},$ where $\mu_{s_k}=\mu_j.$ 
We claim that for $s_k\in I_s$ 
$$\frac{\partial U}{\partial \alpha_{s_k}}\bigg|_{z_s^*}=\sqrt{-1}E_{s_k}\mu_{s_k}l_{s_k},$$
and for $s_k\in J_s.$
$$\frac{\partial U}{\partial \alpha_{s_k}}\bigg|_{z_s^*}=-\sqrt{-1}\mu_{s_k}l_{s_k}.$$
Indeed, let $W$ again be the function defined in (\ref{W}). Then by Theorem \ref{co-vol}, we have for $s_k\in I_s$
\begin{equation}\label{3.3}
\frac{\partial W}{\partial \alpha_{s_k}}\bigg|_{\big(a_{I_s}^*, \alpha_{J_s}\big)}=\sqrt{-1}E_{s_k}\mu_{s_k}l_{s_k}
\end{equation}
and for $s_k\in J_s$ 
\begin{equation}\label{3.4}
\frac{\partial W}{\partial \alpha_{s_k}}\bigg|_{\big(a_{I_s}^*, \alpha_{J_s}\big)}=-\sqrt{-1}\mu_{s_k}l_{s_k}.
\end{equation}
On the other hand, by the Chain Rule and (\ref{xia}), we have for $k\in\{1,\dots, 6\},$ 
\begin{equation}\label{3.5}
\begin{split}
\frac{\partial W}{\partial \alpha_{s_k}}\bigg|_{\big(a_{I_s}^*, \alpha_{J_s}\big)}=&\frac{\partial U}{\partial \alpha_{s_k}}\bigg|_{z^*_s}+\frac{\partial U}{\partial \xi_s}\bigg|_{z^*_s}  \frac{\partial \xi_s(\alpha_s)}{\partial \alpha_{s_k}}\bigg|_{\big(a_{I_s}^*, \alpha_{J_s}\big)}=\frac{\partial U}{\partial \alpha_{s_k}}\bigg|_{z^*_s}.
\end{split}
\end{equation}
Putting (\ref{3.3}), (\ref{3.4}) and (\ref{3.5}) together, we have 
\begin{equation}\label{3.26}
\begin{split}
\sum_{s=1}^c\bigg(\sum_{k=1}^6\frac{\partial U}{\partial \alpha_{s_k}}\bigg|_{z_s^*}\bigg)=&\sqrt{-1}\sum_{i\in I}\sum_{s_k\sim i} E_{s_k}\mu_{s_k}l_{s_k}-\sqrt{-1}\sum_{j\in J}\sum_{s_k\sim j} \mu_{s_k}l_{s_k}\\
=&\sqrt{-1} \sum_{i\in I}E_i\mu_i\Big(\sum_{s_k\sim i}l_{s_k}\Big)-\sqrt{-1}\sum_{j\in J}\mu_j\Big(\sum_{s_k\sim j} l_{s_k}\Big)\\
=&\sqrt{-1} \sum_{i\in I}E_i\mu_il_i-\sqrt{-1}\sum_{j\in J}\mu_jl_j\\
=&-\sqrt{-1}\sum_{i\in I}E_i\mu_i\Big(\mathrm H(v_i)-\frac{\iota_i}{2}\mathrm H(u_i)\Big)-\sqrt{-1}\sum_{j\in J}\mu_jl_j,
\end{split}
\end{equation}
where $s_k\sim i$ if $e_{s_k}$ intersects $L_{\text{FSL},i}$ for $i\in I$ and $s_k\sim j$ if $e_{s_k}$ intersects $L_{\text{FSL},j}$ for $j\in J,$ and the last equality come from that $\mathrm H(u_i)=\sqrt{-1}\theta_i$ and 
$\mathrm H(v_i)=-l_i+\frac{\iota_i}{2}\sqrt{-1}\theta_i.$

Next, recall that for each $i\in I,$ $\alpha^*_i=\pi+\frac{E_i\mu_i\sqrt{-1}\mathrm H(u_i)}{2}$ and $\beta_i=\pi+\frac{\mu_i\theta_i}{2},$ and for each $j\in J,$ $\alpha_j=\pi+\frac{\mu_j\theta_j}{2}.$ \
For  $i\in I,$ we have
\begin{equation}\label{3.27}
\begin{split}
&\sqrt{-1} \bigg(\Big(\frac{\iota_i}{2}\Big)\alpha^*_i +\frac{E_i}{q_i}(\beta_i-\pi) + \frac{p_i}{q_i}(\alpha^*_i - \pi)\bigg)-\frac{E_i\mu_i}{2}\Big(\mathrm H(v_i)-\frac{\iota_i}{2}\mathrm H(u_i)\Big)\\
=&\sqrt{-1} \bigg(\Big(\frac{\iota_i}{2}\Big)\Big(\pi+\frac{E_i\mu_i\sqrt{-1}\mathrm H(u_i)}{2}\Big)+\frac{E_i}{q_i}\Big(\frac{\mu_i\theta_i}{2}\Big) + \frac{p_i}{q_i}\lt(\frac{E_i\mu_i \sqrt{-1} \mathrm{H}(u_i)}{2}\rt)\bigg) \\
&-\frac{E_i\mu_i}{2}\Big(\mathrm H(v_i)-\frac{\iota_i}{2}\mathrm H(u_i)\Big)\\
=&\lt(\frac{\iota_i}{2}\rt)\sqrt{-1}\pi+\frac{E_i\mu_i}{2}\Big( 
-\frac{\iota_i}{2}\mathrm H(u_i)+\frac{\sqrt{-1}\theta_i}{q_i} - \frac{p_i}{q_i} \mathrm H(u_i) -\mathrm H(v_i)+\frac{\iota_i}{2}\mathrm H(u_i)\Big)\\
=&\lt(\frac{\iota_i}{2}\rt)\sqrt{-1}\pi,
\end{split}
\end{equation}
where the last equality comes from $p_i\mathrm H(u_i)+ q_i\mathrm H(v_i)=\sqrt{-1}\theta_i.$
For $i\in I,$  we also have
\begin{equation}\label{3.28}
\begin{split}
&\sqrt{-1}\lt(a_{i,0}\beta_i+\frac{E_i}{q_i}(\alpha_i^*-\pi) + \frac{p_i'}{q_i}(\beta_i-\pi)\rt)\\
=&\sqrt{-1}\bigg(a_{i,0}\Big(\pi+\frac{\mu_i\theta_i}{2}\Big)+\frac{E_i}{q_i}\Big(\frac{E_i\mu_i\sqrt{-1}\mathrm H(u_i)}{2}\Big) + \frac{p_i'}{q_i}\lt(\frac{\mu_i\theta_i}{2}\rt)\bigg)\\
=&a_{i,0}\sqrt{-1}\pi+\frac{\mu_i}{2}\Big(a_{i,0}\sqrt{-1}\theta_i-\frac{\mathrm H(u_i)}{q_i} + \frac{pi'}{q_i}\sqrt{-1}\theta_i\Big)\\
=&a_{i,0}\sqrt{-1}\pi+\frac{\mu_i}{2}\mathrm H(\gamma_i),
\end{split}
\end{equation}
where the last equality comes from Equation (\ref{hgammacompu}).

For each $j\in J,$ we have
\begin{equation}\label{3.29}
\begin{split}
\sqrt{-1}\Big(a_{j,0}+\frac{\iota_j}{2}\Big)\alpha_j-\frac{\mu_j}{2}l_j=&\sqrt{-1}\Big(a_{j,0}+\frac{\iota_j}{2}\Big)\Big(\pi+\frac{\mu_j\theta_j}{2}\Big)-\frac{\mu_j}{2}l_j\\
=&\Big(a_{j,0}+\frac{\iota_j}{2}\Big)\sqrt{-1}\pi+\frac{\mu_j}{2}\Big(a_{j,0}\sqrt{-1}\theta_j+\frac{\iota_j}{2}\sqrt{-1}\theta_j-l_j\Big)\\
=&\Big(a_{j,0}+\frac{\iota_j}{2}\Big)\sqrt{-1}\pi+\frac{\mu_j}{2}\Big(a_{j,0}\mathrm H(u_j)+\mathrm H(v_j)\Big)\\
=&\Big(a_{j,0}+\frac{\iota_j}{2}\Big)\sqrt{-1}\pi+\frac{\mu_j}{2}\mathrm H(\gamma_j).
\end{split}
\end{equation}
Then the result follows from (\ref{3.27}), (\ref{3.28}), (\ref{3.29}) and Lemma \ref{L3.3}. For the general case, the result follows the analyticity of the involved functions.
\end{proof}

Finally, we need the following lemma from \cite{WY4}.
\begin{lemma}\label{L3.5} (\cite[Lemma 3.6]{WY4}) For $s\in \{1,\dots, c\},$ let $u_{s_1},\dots, u_{s_6}$ be the meridians of a tubular neighborhood of the components of $L_{\text{FSL}}$ intersecting the six edges of $\Delta_s.$ Then
\begin{equation}\label{CM0} 
\begin{split}
&\frac{e^{-{\sqrt{-1}}\sum_{s_i\in I_s}\alpha^*_{s_i}-{\sqrt{-1}}\sum_{s_j\in J_s}\alpha_{s_j}+4\sqrt{-1}\xi_s^*-\sum_{i=1}^4\log\big(1-e^{2\sqrt{-1}(\xi^*_s-\tau^*_{s_i})}\big)}}{{\frac{\partial ^2U}{\partial \xi_s^2}\Big|_{z^*_s}}}\\&\quad\quad\quad\quad\quad\quad\quad\quad\quad\quad\quad\quad\quad\quad\quad\quad\quad\quad\quad\quad=\frac{-1}{16\sqrt{\det\mathbb G\bigg(\frac{\mathrm H(u_{s_1})}{2},\dots,\frac{\mathrm H(u_{s_6})}{2}\bigg)}}.
\end{split}
\end{equation}
\end{lemma}

\begin{proof}[Proof of Proposition \ref{Rtorsion}]
From (\ref{Cformula}), Lemmas \ref{L3.1}, \ref{L3.2}, \ref{L3.3}, \ref{L3.4} and \ref{L3.5}, we have
\begin{align*}
&\frac{C^{\mathbf E_I}(\mathbf{z}^{\boldsymbol E_{I}})}{\sqrt{-\det \Hess({G}^{\mathbf E_I} )(\mathbf{z^{\mathbf E_I}})}}\\
= &
\frac{e^{\bigg(\sum_{i\in I} \Big(a_{i,0}+\frac{\iota_i}{2}\Big)+\sum_{j\in J}\Big(a_{j,0}+\frac{\iota_j}{2}\Big)\bigg)\sqrt{-1}\pi + \frac{1}{2}\sum_{k=1}^n\mu_k\mathrm H(\gamma_k)}}{\sqrt{-(-16)^{c} (-2)^{|I|}\det\bigg(\frac{\partial \mathrm H(\Upsilon_{i_1})}{\partial \mathrm H(u_{i_2})}\bigg)_{i_1,i_2\in I} \prod_{s=1}^{c}\sqrt{\det\mathbb G\bigg(\frac{\mathrm H(u_{s_1})}{2},\dots,\frac{\mathrm H(u_{s_6})}{2}\bigg)}}}\\
= & \frac{e^{\bigg(\sum_{i\in I} \Big(a_{i,0}+\frac{\iota_i}{2}\Big)+\sum_{j\in J}\Big(a_{j,0}+\frac{\iota_j}{2}\Big)\bigg)\sqrt{-1}\pi + \frac{1}{2}\sum_{k=1}^n\mu_k\mathrm H(\gamma_k)}}{2^{\frac{|I|+c}{2}}\sqrt{\pm\mathbb{T}_{(M\setminus L,\mathrm{\mathbf\Upsilon})}([\rho_{M_{L_{\boldsymbol\theta}}}])}},
\end{align*}
where the last equality follows from Theorem \ref{Rtorthm} (2).
\end{proof}

\begin{proposition}\label{0FCasym} Under the assumptions in Proposition \ref{lfc}, we have
$$\sum_{\mathbf E_I}\widehat{f_r}(\mathbf{s}^{\mathbf E_I}, \mathbf{1-2m^{E_I}}, \mathbf{0})
= C_1\frac{e^{\frac{1}{2}\sum_{k=1}^{n}\mu_k\mathrm H^{(r)}(\gamma_k)}}{\sqrt{\pm\mathbb{T}_{(M\setminus L, \mathrm{\mathbf\Upsilon})}([\rho_{M^{(r)}}])}}e^{\frac{r}{4\pi}\big(\mathrm{Vol}(M^{(r)})+\sqrt{-1}\mathrm{CS}(M^{(r)})\big)}\bigg(1+O\Big(\frac{1}{r}\Big)\bigg), 
$$
where $C_1  = Y 2^{c}r^{\sum_{i\in I} \frac{\zeta_i}{2} - \frac{c}{2}}
(-1)^{ - \frac{rc}{2} +\sum_{i\in I} \Big(a_{i,0}+\frac{\iota_i}{2}\Big)+\sum_{j\in J}\Big(a_{j,0}+\frac{\iota_j}{2}\Big)} $ and $Y$ is defined in (\ref{defY}).
\end{proposition}

\begin{proof}
By Proposition \ref{lfcexpress}, Lemma \ref{YEIY}, Proposition \ref{lfc} and \ref{Rtorsion},
\begin{align*}
&\widehat{f_r^{\mathbf E_I}}(\mathbf{s}^{\mathbf E_I}, \mathbf{1-2m^{E_I}}, \mathbf{0})
 \\
=&
    \frac{Y r^{|I| + c} }
    { 2^{|I| + c}\pi^{|I|+c} } \lt(\frac{2}{r}\rt)^c \lt(\frac{2\pi}{r}\rt)^{\frac{|I|+c}{2}} (4\pi\sqrt{-1})^{\frac{|I|+c}{2}} 
\\
&\qquad \frac{(-1)^{-\frac{rc}{2}}C^{\mathbf E_I}(\mathbf{z}^{\boldsymbol E_{I}})}{\sqrt{- \lt(\prod_{i\in I}q_i\rt)\det \Hess({G}^{\mathbf E_I} )(\mathbf{z^{\mathbf E_I}})}} e^{\frac{r}{4 \pi} (\Vol (M_{L_{\boldsymbol\theta}})  + \sqrt{-1}\CS(M_{L_{\boldsymbol\theta}}))} \Big( 1 + O\Big(\frac{1}{r}\Big ) \Big)\\
=& 
C_0\frac{e^{\frac{1}{2}\sum_{k=1}^{n}\mu_k\mathrm H^{(r)}(\gamma_k)}}{\sqrt{\pm\mathbb{T}_{(M\setminus L, \mathrm{\mathbf\Upsilon})}([\rho_{M^{(r)}}])}}e^{\frac{r}{4\pi}\big(\mathrm{Vol}(M^{(r)})+\sqrt{-1}\mathrm{CS}(M^{(r)})\big)}\bigg(1+O\Big(\frac{1}{r}\Big)\bigg), 
\end{align*}
where $C_0 = Y 2^{-|I|+c}r^{ \frac{|I|-c}{2}}
(-1)^{- \frac{rc}{2} + \frac{|I|+c}{4} +\sum_{i\in I} \Big(a_{i,0}+\frac{\iota_i}{2}\Big)+\sum_{j\in J}\Big(a_{j,0}+\frac{\iota_j}{2}\Big)}$.
Thus, we have
\begin{align*}
&\sum_{\mathbf E_I}\widehat{f_r}(\mathbf{s}^{\mathbf E_I}, \mathbf{1-2m^{E_I}}, \mathbf{0}) \\
=& C_1\frac{e^{\frac{1}{2}\sum_{k=1}^{n}\mu_k\mathrm H^{(r)}(\gamma_k)}}{\sqrt{\pm\mathbb{T}_{(M\setminus L, \mathrm{\mathbf\Upsilon})}([\rho_{M^{(r)}}])}}e^{\frac{r}{4\pi}\big(\mathrm{Vol}(M^{(r)})+\sqrt{-1}\mathrm{CS}(M^{(r)})\big)}\bigg(1+O\Big(\frac{1}{r}\Big)\bigg), 
\end{align*}
where $C_1 = Y 2^{c}r^{\sum_{i\in I} \frac{|I|-c}{2}}
(-1)^{- \frac{rc}{2}  +\frac{|I|+c}{4}+ \sum_{i\in I} \Big(a_{i,0}+\frac{\iota_i}{2}\Big)+\sum_{j\in J}\Big(a_{j,0}+\frac{\iota_j}{2}\Big)} $.
\end{proof}

\subsection{Estimate of other Fourier coefficients}

\begin{proposition}\label{fail0}
Assume that 
$$ \Vol(M_{L_{\boldsymbol\theta}})> \max\Bigg\{\max_{(\boldsymbol \alpha_{\zeta_I}, \boldsymbol \xi) \in \overline{D_H\setminus D_{\delta_0}}} \im \tilde U(\boldsymbol \alpha_{\zeta_I}, \boldsymbol \xi), 2cv_8 - 4\pi \delta_0
\Bigg\} ,$$
where $ \tilde U(\boldsymbol \alpha_{\zeta_I}, \boldsymbol \xi) $ is defined in (\ref{deftU}) and $\overline{D_H\setminus D_{\delta_0}}$ is the closure of $D_H\setminus D_{\delta_0}$.
Then for any $\mathbf E^I$ and $\mathbf s_I$, there exists $\epsilon'>0$ such that if $B_{k_0}\neq 0$ for some $k_0 \in\{1,2,\dots,c\}$, then  
\begin{equation*}
   \lt| \widehat{f_r^{\mathbf E_I}} (\mathbf s_I, \mathbf A_{\zeta_I}, \mathbf B) \rt|  < O\Big(e^{\frac{r}{4 \pi}(\Vol(M_{L_{\boldsymbol\theta}}) - \epsilon')}\Big).
\end{equation*}
\end{proposition}

\begin{proof} Let 
$$ G_r^{\mathbf E_I, \mathbf A_{\zeta_I}, \mathbf B}(\mathbf s_I, \boldsymbol{\alpha}_{\zeta_I}, \boldsymbol \xi)
=
W_r^{\mathbf E_I}(\mathbf s_I, \boldsymbol{\alpha}_{\zeta_I}, \boldsymbol \xi)  - 2\pi \sum_{i\in I}  A_{i,\zeta_i}\alpha_{i,\zeta_i} 
- 4\pi\sum_{s=1}^c  B_s \xi_s.$$
Recall from Proposition \ref{formulaFC} that
\begin{align*}
   & \widehat{f_r^{\mathbf E_I}}(\mathbf s_I, \mathbf A_{\zeta_I},  \boldsymbol  B)  =
    \frac{r^{|I| + c}\big(\prod_{i \in I} E_i\big) }{2^{|I| + c} \pi ^{|I| + c}}\\
&  \times  \int_{D_H} 
   (-1)^{\sum_{i\in I} A_{i,\zeta_i}} \phi_r\lt(\mathbf s_I, \boldsymbol{\alpha}_{\zeta_I}, \boldsymbol\xi \rt) 
    e^{\frac{r}{4\pi \sqrt{-1}} G_r^{\mathbf E_I, \mathbf A_{\zeta_I}, \mathbf B}(\mathbf s_I, \boldsymbol{\alpha}_{\zeta_I}, \boldsymbol \xi)} d\boldsymbol{\alpha}_{\zeta_I} d\boldsymbol\xi.
\end{align*}
When $\alpha_{i,\zeta_i}\in \RR$ for all $i\in I$, by Lemma \ref{kappadef}, on any compact subset of $D_{H,\CC}$, $ \im G_r^{\mathbf E_I, \mathbf A_{\zeta_I}, \mathbf B}(\mathbf s_I, \boldsymbol{\alpha}_{\zeta_I}, \boldsymbol \xi)$ converges uniformly to 
\begin{align*}
 \im G^{\mathbf E_I, \mathbf A_{\zeta_I}, \mathbf B}(\mathbf s_I, \boldsymbol{\alpha}_{\zeta_I}, \boldsymbol \xi)
&=  \im \tilde U(\boldsymbol \alpha_{\zeta_I}, \boldsymbol \xi).
\end{align*}

We first estimate the integral on $D_H\setminus D_{\delta_0}$. By assumption, we have
\begin{align}\label{bdylessVol}
\Vol(M_{L_{\boldsymbol\theta}}) &
> \max_{\overline{D_H \setminus D_{\delta_1}}} 
\im \Big(G^{\mathbf E_I, \mathbf A_{\zeta_I}, \mathbf B}(\mathbf s_I, \boldsymbol{\alpha}_{\zeta_I}, \boldsymbol \xi)
\Big)  + \epsilon'.
\end{align} 
for some $\epsilon'>0$. Thus, we have
\begin{align}
&\lt| \int_{D_H\setminus D_{\delta_0}} 
   (-1)^{\sum_{i\in I} A_{i,\zeta_i}} \phi_r\lt(\mathbf s_I, \boldsymbol{\alpha}_{\zeta_I}, \boldsymbol\xi \rt) 
    e^{\frac{r}{4\pi \sqrt{-1}} G_r^{\mathbf E_I, \mathbf A_{\zeta_I}, \mathbf B}(\mathbf s_I, \boldsymbol{\alpha}_{\zeta_I}, \boldsymbol \xi)}d\boldsymbol{\alpha}_{\zeta_I} d\boldsymbol\xi
\rt| \notag\\
=& o\lt( e^{\frac{r}{4\pi} (\Vol(M_{L_{\boldsymbol\theta}})-\epsilon')} \rt) .
\end{align}

Next, we estimate the integral on $D_{\delta_0}$. For simplicity, we assume that $B_c \neq 0$. The following arguments also work for other possibilities.

First, we consider the case where $ B_c > 0$. Consider the surface $S^+ = S^+_{\text{top}} \cup S^+_{\text{sides}}$ in the closure of $D_{\delta_0,\CC}$, where
$$ S^{+}_{\text{top}} = \{ (\boldsymbol \alpha_{\zeta_I}, \boldsymbol \xi) + (0, \dots, 0, \sqrt{-1}\delta_0) \mid (\boldsymbol \alpha_{\zeta_I}, \boldsymbol \xi)  \in D_{\delta_0, \CC} \} $$
and
$$ S^+_{\text{sides}} = \{ (\boldsymbol \alpha_{\zeta_I}, \boldsymbol \xi) + (0, \dots, 0, t\sqrt{-1}\delta_0)
\mid (\boldsymbol \alpha_{\zeta_I}, \boldsymbol \xi) \in \partial D_{\delta_0}, t\in [0,1] \}. $$
On $S^+_{\text{top}}$, by the Mean Value Theorem,
\begin{align}
&\im G^{\mathbf E_I} \big((\boldsymbol \alpha_{\zeta_I}, \boldsymbol \xi) + (0, \dots, 0, \sqrt{-1}\delta_0) \big)
- \im G^{\mathbf E_I} (\boldsymbol \alpha_{\zeta_I}, \boldsymbol \xi) \notag\\
= &
\frac{\partial \im G^{\mathbf E_I}}{\partial \im\xi_c} \big( (\boldsymbol \alpha_{\zeta_I}, \boldsymbol \xi) + (0, \dots, 0, \sqrt{-1}\delta_0') \big) \cdot \delta_0 
\end{align}
for some $\delta_0' \in (0,\delta_0)$. Note that
\begin{align*}
\frac{\partial \im G^{\mathbf E_I}}{\partial \im\xi_c} \big( (\boldsymbol \alpha_{\zeta_I}, \boldsymbol \xi) + (0, \dots, 0, \sqrt{-1}\delta_0') \big) 
=  \frac{\partial \im U}{\partial \im \xi}\Bigg\vert_{(\boldsymbol \alpha_s,\xi_s+\sqrt{-1}\delta_0')} - 4\pi B_c  
< 2\pi - 4\pi = -2\pi,
\end{align*}
where the last inequality follows from Lemma \ref{less2pi} and $B_c \geq 1$.
This implies that
\begin{align*}
&\im G^{\mathbf E_I} \big((\boldsymbol \alpha_{\zeta_I}, \boldsymbol \xi) + (0, \dots, 0, \sqrt{-1}\delta_0) \big) \notag \\
<& \im G^{\mathbf E_I} (\boldsymbol \alpha_{\zeta_I}, \boldsymbol \xi)
-2\pi\delta_0 
< 2cv_8 - 2\pi\delta_0
< \Vol(M_{L_{\boldsymbol\theta}}),
\end{align*}
where the second last inequality follows from Lemma \ref{Vmax} and the last inequality follows from the assumption (1). By making $\epsilon'>0$ smaller if necessary, we have
\begin{align}\label{xitop}
&\im G^{\mathbf E_I} \big((\boldsymbol \alpha_{\zeta_I}, \boldsymbol \xi) + (0, \dots, 0, \sqrt{-1}\delta_0) \big)  
< \Vol(M_{L_{\boldsymbol\theta}}) - \epsilon'.
\end{align}

Next, on $S^+_{\text{sides}}$, by Proposition \ref{prop53}, for $t\in [0,1]$ we have
\begin{align}\label{xiside}
&\im G^{\mathbf E_I} \big((\boldsymbol \alpha_{\zeta_I}, \boldsymbol \xi) + t(0, \dots, 0, \sqrt{-1}\delta_0) \big) \notag \\
\leq& \max\{\im G^{\mathbf E_I}(\boldsymbol \alpha_{\zeta_I}, \boldsymbol \xi) , \im G^{\mathbf E_I} \big((\boldsymbol \alpha_{\zeta_I}, \boldsymbol \xi) + (0, \dots, 0, \sqrt{-1}\delta_0) \big) \} 
< \Vol(M_{L_{\boldsymbol\theta}}) - \epsilon',
\end{align}
where the last inequality follows from (\ref{xitop}) and the assumption. From (\ref{xitop}) and (\ref{xiside}), we have
\begin{align*}
&\lt| \int_{D_{\delta_0}} 
   (-1)^{\sum_{i\in I} A_{i,\zeta_i}} \phi_r\lt(\mathbf s_I, \boldsymbol{\alpha}_{\zeta_I}, \boldsymbol\xi \rt) 
    e^{\frac{r}{4\pi \sqrt{-1}} G_r^{\mathbf E_I, \mathbf A_{\zeta_I}, \mathbf B}(\mathbf s_I, \boldsymbol{\alpha}_{\zeta_I}, \boldsymbol \xi)}d\boldsymbol{\alpha}_{\zeta_I} d\boldsymbol\xi
\rt| \notag\\
=&\lt| \int_{S^+} 
   (-1)^{\sum_{i\in I} A_{i,\zeta_i}} \phi_r\lt(\mathbf s_I, \boldsymbol{\alpha}_{\zeta_I}, \boldsymbol\xi \rt) 
    e^{\frac{r}{4\pi \sqrt{-1}} G_r^{\mathbf E_I, \mathbf A_{\zeta_I}, \mathbf B}(\mathbf s_I, \boldsymbol{\alpha}_{\zeta_I}, \boldsymbol \xi)} d\boldsymbol{\alpha}_{\zeta_I} d\boldsymbol\xi
\rt| \notag\\
=& o\lt( e^{\frac{r}{4\pi} (\Vol(M_{L_{\boldsymbol\theta}})-\epsilon')} \rt) .
\end{align*}

If $B_c < 0$, then we consider the surface $S^- = S^-_{\text{top}} \cup S^-_{\text{sides}}$ in the closure of $D_{\delta_0,\CC}$, where
$$ S^{-}_{\text{top}} = \{ (\boldsymbol \alpha_{\zeta_I}, \boldsymbol \xi) - (0, 0, \dots, 0, \sqrt{-1}\delta_0) \mid (\boldsymbol \alpha_{\zeta_I}, \boldsymbol \xi)   \in D_{\delta_0, \CC} \} $$
and
$$ S^-_{\text{sides}} = \{ (\boldsymbol \alpha_{\zeta_I}, \boldsymbol \xi)  - (0, \dots, 0, t\sqrt{-1}\delta_0)
\mid (\boldsymbol \alpha_{\zeta_I}, \boldsymbol \xi) \in \partial D_{\delta_0}, t\in [0,1] \}. $$
Using the same arguments as in the previous case, on $S^-$ we have
\begin{align}
\im G^{(\boldsymbol A_{0,I}, \boldsymbol A_{\zeta_I}, \boldsymbol  B_I)} (\boldsymbol \alpha_{\zeta_I}, \boldsymbol \xi) 
 < \Vol(M_{L_{\boldsymbol\theta}})  - \epsilon'
\end{align}
This completes the proof.
\end{proof}

From Proposition \ref{fail0}, it remains to consider the Fourier coefficients with $\mathbf B = \mathbf 0 = (0,\dots,0)$. To do this, for $i \in I$, consider the functions $k_i^\pm : \{0,1,\dots, |q_i|-1\} \times \ZZ \to \RR$ defined by
$$ k_i^\pm(s_i, A_{\zeta_i}) = \frac{I_i(s_i) \mp 1}{q_i} + A_{\zeta_i}.$$
\begin{lemma}\label{fcneq00}
When $|q_i|$ is odd, 
\begin{itemize}
\item $k_i^+(s_i, A_{\zeta_i}) = 0 $ if and only if 
$(s_i,A_{\zeta_i}) = (s_i^+, 1-2m_i^+)$; 
\item $k_i^-(s_i, A_{\zeta_i}) = 0 $ if and only if 
$(s_i,A_{\zeta_i}) = (s_i^-, 1-2m_i^- )$.
\end{itemize}
Moreover, 
\begin{itemize}
\item if $(s_i, A_{\zeta_i}) \neq (s_i^+, 1-2m_i^+)$, then $ \big|k^+(s_i, A_{\zeta_i})\big| \geq \frac{1}{|q_i|};$
\item if $(s_i, A_{\zeta_i}) \neq (s_i^-, 1-2m_i^-)$, then $ \big|k^-(s_i, A_{\zeta_i})\big| \geq \frac{1}{|q_i|}$.
\end{itemize}
\end{lemma}
\begin{proof}
Suppose $k_i^\pm(s_i, A_{\zeta_i}) = 0$ for some $(s_i, A_{\zeta_i})\in\{0,1,\dots, |q_i|-1\} \times \ZZ$. Then we have
$$ I_i(s_i, A_{\zeta_i}) = \pm 1 - q_i A_{\zeta_i} = \pm 1 - q_i - q_i (A_{\zeta_i}-1).$$
Suppose $A_{\zeta_i}$ is even. Then $I_i(s_i, A_{\zeta_i}) = \pm 1\pmod{2|q_i|}$ is an odd number. 
However, by Lemma \ref{arith}, the image of $I_i^+$ has the same parity of $1-q_i$, which is even when $|q_i|$ is odd. This leads to a contradiction. Thus, $A_{\zeta_i}$ is odd. Then we have $I_i(s_i, A_{\zeta_i})= \pm 1 - q_i - q_i (A_{\zeta_i}-1) \equiv \pm 1 - q_i \pmod{2|q_i|}$. Since $I_i: \{0,1,\dots,|q_i|-1\} \to \{0,1,\dots,2|q_i|-1\} $ is injective, we have $s_i = s_i^\pm$. This proves the first claim.

For the second claim, if $(s_i, A_{\zeta_i}) \neq (s_i^+, 1-2m_i^+)$, then
$$ |q_i k_i^+(s_i, A_{\zeta_i})| = | I_i(s_i) - 1 - q A_{\zeta_i} |$$
is a non-zero integer. The other part can be proved similarly.
\end{proof}

\begin{lemma}\label{fcneq01}
When $|q_i|$ is even, there exist $\tilde s_i^+, \tilde s_i^- \in \{0,1,\dots, |q_i| - 1\}$ and $\tilde m_i^+, \tilde m_i^- \in \ZZ$ such that 
\begin{itemize}
\item $k_i^+(s_i, A_{\zeta_i}) = 0 $ if and only if 
$(s_i,A_{\zeta_i}) = (s_i^+, 1-2m_i^+)$ or $(\tilde s_i^+, -2\tilde m_i^+)$; 
and 
\item $k_i^-(s_i, A_{\zeta_i}) = 0 $ if and only if 
$(s_i,A_{\zeta_i}) = (s_i^-, 1-2m_i^-)$ or $(\tilde s_i^-, -2\tilde m_i^-).$
\end{itemize}
Furthermore, $\tilde s_i^\pm = s_i^\pm + \frac{q_i}{2} \pmod{|q_i|}$. Moreover, 
\begin{itemize}
\item if $(s_i, A_{\zeta_i}) \not\in \big\{(s_i^+, 1-2m_i^+), (\tilde s_i^+, -2\tilde m_i^+)\big\}$, then
$ \big|k^+(s_i, A_{\zeta_i})\big| \geq \frac{1}{|q_i|}$;
\item if $(s_i, A_{\zeta_i}) \not\in \big\{(s_i^-, 1-2m_i^-), (\tilde s_i^-, -2\tilde m_i^-)\big\}$, then
$ \big|k^-(s_i, A_{\zeta_i})\big| \geq \frac{1}{|q_i|}.$
\end{itemize}
\end{lemma}
\begin{proof} 
Note that when $k^\pm(s_i, A_{\zeta_i}) = 0$, we have $\frac{I_i(s_i) \mp 1}{q_i} \in \ZZ$. By Lemma \ref{arith}, we have $I(s_i) = \pm 1$ or $\pm 1 + |q_i|$. Recall that $I_i(s_i^\pm) = \pm 1 - q_i + 2m_i^\pm q_i$. In particular, we have $k^\pm(s_i^\pm, 1-2m_i^\pm) = 0$. 

Besides, let $\tilde s_i^\pm \in \{0,1,\dots, |q_i|-1\}$ such that
$$ \tilde s_i^\pm \equiv  s_i^\pm + \frac{|q_i|}{2} \pmod{|q_i|}.$$ 
By the definition of $I_i$, we have 
$$I_i(\tilde s_i^\pm) \equiv I(s_i^\pm) - C_{k-1}|q_i| \pmod{2|q_i|}.$$ 
Since $q_i = A_{i, \zeta_i-1}$ is even and $(A_{i,\zeta_i-1}, C_{i,\zeta_i -1})$ is a pair of coprime integers, $C_{i,\zeta_i - 1}$ must be odd. Thus, 
$$I(\tilde s_i^\pm) \equiv I(s_i^\pm) + q_i \pmod{2|q_i|} \equiv \pm 1 \pmod{2|q_i|} .$$
Define $ \tilde m_i^\pm \in \ZZ$ such that
\begin{align}\label{deftm}
I(\tilde s_i^\pm) = \pm 1 + 2\tilde m_i^\pm q_i.
\end{align}
Then $k(\tilde s_i^\pm, - 2\tilde m_i^\pm) = 2\tilde m_i^\pm - 2\tilde m_i^\pm = 0$. Since $I$ is injective, $s_i^\pm$ is the unique integer in $\{0,\dots, |q_i|-1\}$ such that $I(\tilde s_i^\pm)  \equiv 1 \pmod{2|q_i|}$.

Finally, if $(s_i, A_{\zeta_i}) \not\in \big\{(s_i^+, 1-2m_i^+), (\tilde s_i^+, -2\tilde m_i^+)\big\}$,
$$ |q_i k_i^+(s_i, A_{\zeta_i})| = | I_i(s_i) - 1 - q_i A_{\zeta_i} |$$
is a non-zero integer. The other part can be proved similarly.
\end{proof}

For $i \in I$, let 
\begin{align}\label{defS_i}
S_i
=
\begin{cases}
\big\{(s_i^+, 1-2m_i^+), (s_i^-, 1-2m_i^-)\big\} & \text{if $|q_i|$ is odd,}\\
\big\{(s_i^+, 1-2m_i^+), (s_i^-, 1-2m_i^-), (\tilde s_i^+, -2\tilde m_i^+), (\tilde s_i^-, -2\tilde m_i^-)\big\} & \text{if $|q_i|$ is even.}
\end{cases}
\end{align}

\begin{proposition}\label{fail1} Assume that $\theta_i=2|\beta_i -\pi| < \pi$ for all $i\in I$ and 
$$ \Vol(M_{L_{\boldsymbol\theta}})> \max_{(\boldsymbol \alpha_{\zeta_I}, \boldsymbol \xi) \in \overline{D_H\setminus D_{\delta_0}}} \im \tilde U(\boldsymbol \alpha_{\zeta_I}, \boldsymbol \xi).$$
Then there exists $\epsilon'>0$ such that if
$(s_{i_0}, A_{\zeta_{i_0}}) \not\in S_{i_0}$ for some $i_0 \in I$, then
\begin{equation*}
   \lt| \widehat{f_r^{\mathbf E_I}} (\mathbf s_I, \mathbf A_{\zeta_I}, \mathbf 0) \rt|  < O\Big(e^{\frac{r}{4 \pi}(\Vol(M_{L_{\boldsymbol\theta}}) - \epsilon')}\Big).
\end{equation*}
\end{proposition}

\begin{proof}
Recall that
$$ G_r^{\mathbf E_I, \mathbf A_{\zeta_I}, \mathbf 0}(\mathbf s_I, \boldsymbol{\alpha}_{\zeta_I}, \boldsymbol \xi)
=
W_r^{\mathbf E_I}(\mathbf s_I, \boldsymbol{\alpha}_{\zeta_I}, \boldsymbol \xi)  - 2\pi \sum_{i\in I}  A_{i,\zeta_i}\alpha_{i,\zeta_i} .$$
By Proposition \ref{formulaFC},
\begin{align*}
   & \widehat{f_r^{\mathbf E_I}}(\mathbf s_I, \mathbf A_{\zeta_I},  \boldsymbol  0)  =
    \frac{r^{|I| + c}\big(\prod_{i \in I} E_i\big) }{2^{|I| + c} \pi ^{|I| + c}}\\
&  \times  \int_{D_H} 
   (-1)^{\sum_{i\in I} A_{i,\zeta_i}} \phi_r\lt(\mathbf s_I, \boldsymbol{\alpha}_{\zeta_I}, \boldsymbol\xi \rt) 
    e^{\frac{r}{4\pi \sqrt{-1}} G_r^{\mathbf E_I, \mathbf A_{\zeta_I}, \mathbf 0}(\mathbf s_I, \boldsymbol{\alpha}_{\zeta_I}, \boldsymbol \xi)} d\boldsymbol{\alpha}_{\zeta_I} d\boldsymbol\xi.
\end{align*}

Let $I_0 = \{i_0 \in I \mid (s_{i_0}, A_{\zeta_{i_0}}) \not\in S_{i_0}\}$. By a direct computation, we have
\begin{align*}
 G_r^{\mathbf E_{I}, \mathbf A_{\zeta_I}, \mathbf B}(\mathbf s_I, \boldsymbol{\alpha}_{\zeta_I}, \boldsymbol \xi)
=& G_r^{\mathbf E_{I}} (\boldsymbol \alpha_{\zeta_I}, \boldsymbol \xi)    - 2\pi \sum_{i \in I} k_i^{E_i}(s_i,A_{\zeta_i})(\boldsymbol \alpha_{i,\zeta_i} - \pi)
+ C^{\mathbf E_I}(\mathbf s_I)
\end{align*}
where  
\begin{align*}
k_i^{E_i}(s_i, A_{\zeta_i}) =
\begin{cases}
k_i^+(s_i, A_{\zeta_i}) & \text{ if $E_i = -1$,}\\ 
k_i^-(s_i, A_{\zeta_i}) & \text{ if $E_i = 1$}
\end{cases}
\end{align*}
and 
$C^{\mathbf E_I}(\mathbf s_I)$ is a real numbers independent of $\boldsymbol{\alpha}_{\zeta_I}$ and $\boldsymbol \xi$. By Lemma \ref{fcneq00} and \ref{fcneq01}, 
\begin{align*}
 G_r^{\mathbf E_{I}, \mathbf A_{\zeta_I}, \mathbf B}(\mathbf s_I, \boldsymbol{\alpha}_{\zeta_I}, \boldsymbol \xi)
=& G_r^{\mathbf E_{I}} (\boldsymbol \alpha_{\zeta_I}, \boldsymbol \xi)    - 2\pi \sum_{i \in I_0} k_i^{E_i}(s_i,A_{\zeta_i})(\boldsymbol \alpha_{i,\zeta_i} - \pi)
+ C^{\mathbf E_I}(\mathbf s_I).
\end{align*}

Let $i_0 \in I_0$, $\mathbf{E}_I\in \{-1,1\}^{|I|}$ and let $\mathbf{E}_{I}'\in \{-1,1\}^{|I|}$ be obtained by changing $E_{i_0}$ in $\mathbf E_I$ into $-E_{i_0}$. Since
$$
\frac{-2E_{i_0}(\alpha_{i_0,\zeta_{i_0}}-\pi)(\beta_{i_0}-\pi)}{q_{i_0}}
= \frac{-2(-E_{i_0})(\alpha_{i_0,\zeta_{i_0}}-\pi)(\beta_{i_0}-\pi)}{q_{i_0}}
- \frac{4E_{i_0}(\alpha_{i_0,\zeta_{i_0}}-\pi)(\beta_{i_0}-\pi)}{q_{i_0}},
$$
by a direct computation, we have
\begin{align*}
 G_r^{\mathbf E_{I}, \mathbf A_{\zeta_I}, \mathbf B}(\mathbf s_I, \boldsymbol{\alpha}_{\zeta_I}, \boldsymbol \xi)
=& G_r^{\mathbf E_{I}'} (\boldsymbol \alpha_{\zeta_I}, \boldsymbol \xi) - 2\pi \sum_{i \in I_0\setminus\{i_0\}} k_i^{E_i}(s_i,A_{\zeta_i})(\boldsymbol \alpha_{i,\zeta_i} - \pi)   \\
& - 2\pi \Big(k_{i_0}^{E_{i_0}}(s_{i_0},A_{\zeta_{i_0}}) + \frac{2E_{i_0}(\beta_{i_0}-\pi)}{\pi q_{i_0}} \Big)(\boldsymbol \alpha_{i_0,\zeta_{i_0}} - \pi)  + C^{\mathbf E_I}(\mathbf s_I),
\end{align*}
For all $i\in I$, under the assumption that $\theta_i=2|\beta_i -\pi| < \pi$, we have 
$$
\Bigg| \frac{2E_{i}(\beta_{i}-\pi)}{\pi q_{i}} \Bigg| < \frac{1}{q_{i}}.
$$
By Lemma \ref{fcneq00} and \ref{fcneq01}, since $|k_{i_0}^{E_{i_0}}(s_{i_0},A_{\zeta_{i_0}})|\geq  \frac{1}{q_{i_0}}$, $k_{i_0}^{E_{i_0}}(s_{i_0},A_{\zeta_{i_0}})$ and $k_{i_0}^{E_{i_0}}(s_{i_0},A_{\zeta_{i_0}}) + \frac{2E_{i_0}(\beta_{i_0}-\pi)}{\pi q_{i_0}} $ are either both positive or both negative.
Besides, by Proposition \ref{prop52}, we know that the $\alpha_i$ component of $\mathbf{z}^{\mathbf E_I}$ and that of $\mathbf{z}^{\mathbf E_I'}$ have opposite sign. Altogether, for each $\mathbf E_I\in \{-1,1\}^{|I|}$, by changing some $E_i$ in $\mathbf E_I$ into $-E_i$ if necessary, we can always find $\mathbf E_I''\in \{-1,1\}^{|I|}$ such that 
\begin{align}\label{Gexpress}
G_r^{\mathbf E_{I}, \mathbf A_{\zeta_I}, \mathbf B}(\mathbf s_I, \boldsymbol{\alpha}_{\zeta_I}, \boldsymbol \xi)
=& G_r^{\mathbf E_{I}''} (\boldsymbol \alpha_{\zeta_I}, \boldsymbol \xi)    - 2\pi \sum_{i \in I_0} k_i(\boldsymbol \alpha_{i,\zeta_i} - \pi)
 + C^{\mathbf E_I}(\mathbf s_I),
\end{align}
where $k_i \in \mathbb{R}\setminus\{0\}$ is some nonzero constant such that the product of $k_i$ and the imaginary part of the $\alpha_i$ component of $\mathbf{z}^{\mathbf E_I''}$ is less than or equal to $0$ for all $i\in I_0$.

By Lemma \ref{kappadef}, on any compact subset of $D_{H,\CC}$, $ G_r^{\mathbf E_I, \mathbf A_{\zeta_I}, \mathbf B}(\mathbf s_I, \boldsymbol{\alpha}_{\zeta_I}, \boldsymbol \xi)$ converges uniformly to 
\begin{align*}
G^{\mathbf E_{I}, \mathbf A_{\zeta_I}, \mathbf B}(\mathbf s_I, \boldsymbol{\alpha}_{\zeta_I}, \boldsymbol \xi)
=& G^{\mathbf E_{I}''} (\boldsymbol \alpha_{\zeta_I}, \boldsymbol \xi)    - 2\pi \sum_{i \in I_0} k_i(\boldsymbol \alpha_{i,\zeta_i} - \pi)
+ C^{\mathbf E_I}(\mathbf s_I),
\end{align*}
where $G^{\mathbf E_I''} (\boldsymbol \alpha_{\zeta_I}, \boldsymbol \xi)$ is defined in (\ref{defG}).

In particular, 
\begin{align}\label{onDHClarge}
\im \Big(G^{\mathbf E_I, \mathbf A_{\zeta_I}, \mathbf B}(\mathbf s_I, \boldsymbol{\alpha}_{\zeta_I}, \boldsymbol \xi)\Big)
&= \im\Big( G^{\mathbf E_I''} (\boldsymbol \alpha_{\zeta_I}, \boldsymbol \xi)    - 2\pi \sum_{i \in I_0} k_i(\boldsymbol \alpha_{i,\zeta_i} - \pi)
 \Big)  .
\end{align}
on $D_{H,\CC}$ and 
\begin{align}\label{onDHlarge}
 \im \Big(G^{\mathbf E_I, \mathbf A_{\zeta_I}, \mathbf B}(\mathbf s_I, \boldsymbol{\alpha}_{\zeta_I}, \boldsymbol \xi)\Big)
&= \im \tilde U(\boldsymbol \alpha_{\zeta_I}, \boldsymbol \xi)  .
\end{align}
on $D_{H}$.

We first estimate the integral on $D_H\setminus D_{\delta_0}$. By assumption, we can find $\epsilon'>0$ such that
\begin{align}\label{small1}
\Vol(M_{L_{\boldsymbol\theta}}) &
> \max_{\overline{D_H \setminus D_{\delta_1}}} 
\im \Big(G^{\mathbf E_I, \mathbf A_{\zeta_I}, \mathbf B}(\mathbf s_I, \boldsymbol{\alpha}_{\zeta_I}, \boldsymbol \xi)
\Big)  + \epsilon'.
\end{align} 
Thus, we have
\begin{align}
&\lt| \int_{D_H\setminus D_{\delta_0}} 
   (-1)^{\sum_{i\in I} A_{i,\zeta_i}} \phi_r\lt(\mathbf s_I, \boldsymbol{\alpha}_{\zeta_I}, \boldsymbol\xi \rt) 
      e^{\frac{r}{4\pi \sqrt{-1}} G_r^{\mathbf E_I, \mathbf A_{\zeta_I}, \mathbf 0}(\mathbf s_I, \boldsymbol{\alpha}_{\zeta_I}, \boldsymbol \xi)}  d\boldsymbol{\alpha}_{\zeta_I} d\boldsymbol\xi
\rt| \notag\\
=& o\lt( e^{\frac{r}{4\pi} (\Vol(M_{L_{\boldsymbol\theta}})-\epsilon')} \rt) .
\end{align}

Next, we estimate the integral on $D_{\delta_0}$. Let $i_0\in I_0$ such that $(s_{i_0}, A_{\zeta_{i_0}}) \not\in S_{i_0}$. Consider the surface $S^{\mathbf E_I''} = S^{\mathbf E_I''}_{\text{top}} \cup S^{\mathbf E_I''}_{\text{bottom}}$ defined by
$$ S^{\mathbf E_I''}_{\text{top}} = \{ (\boldsymbol \alpha_{\zeta_I}, \xi) \in D_{\delta_0, \CC} \mid \im  (\boldsymbol \alpha_{\zeta_I} , \boldsymbol \xi) = \im(\mathbf z^{\mathbf E_I''})\}$$
and
$$ S^{\mathbf E_I''}_{\text{side}} = \{ (\boldsymbol \alpha_{\zeta_I} , \boldsymbol \xi) + t\sqrt{-1} \im(\mathbf z^{ \mathbf E_I''}) \mid (\boldsymbol \alpha_{\zeta_I}, \boldsymbol \xi) \in \partial D_{\delta}, t\in [0,1] )\}.$$

On $S^{\mathbf E_I''}_{\text{top}}$, in the proof of Proposition \ref{lfc}, we showed that $\im G^{\mathbf E_I''} (\boldsymbol \alpha_{\zeta_I}, \boldsymbol \xi)$ attains its unique  maximum at $\mathbf{z}^{\mathbf E_I''}= \lt( (\alpha_i^*)_{i\in I}, (\xi_s^*)_{s=1}^c \rt)$. By (\ref{onDHClarge}), for $(\boldsymbol \alpha_{\zeta_I}, \boldsymbol \xi)\in S^{\mathbf E_I''}_{\text{top}}$, 
\begin{align}\label{onDHClarge201}
\im \Big(G^{\mathbf E_I, \mathbf A_{\zeta_I}, \mathbf B}(\mathbf s_I, \boldsymbol{\alpha}_{\zeta_I}, \boldsymbol \xi)\Big) 
&= \im\Big( G^{\mathbf E_I''} (\boldsymbol \alpha_{\zeta_I}, \boldsymbol \xi)    - 2\pi \sum_{i \in I_0} k_i(\boldsymbol \alpha_{i,\zeta_i} - \pi)\Big)  \notag\\
&=  \im  G^{\mathbf E_I''} (\boldsymbol \alpha_{\zeta_I}, \boldsymbol \xi) - 2\pi \sum_{i \in I_0} k_i \mathrm{Im} (\alpha_{i}^* ) \notag\\
&\leq \im  G^{\mathbf E_I''} (\mathbf z^{\mathbf E_I''})
- 2\pi \sum_{i \in I_0} k_i \mathrm{Im} (\alpha_{i}^*).
\end{align}
From Proposition \ref{prop52}, we know that $\im G^{\mathbf E_I''} (\mathbf z^{\mathbf E_I''})= \mathrm{Vol}(M_{L_{\boldsymbol\theta}})$. Thus,
\begin{align}\label{onDHClarge21}
\im \Big(G^{\mathbf E_I, \mathbf A_{\zeta_I}, \mathbf B}(\mathbf s_I, \boldsymbol{\alpha}_{\zeta_I}, \boldsymbol \xi)\Big) 
&\leq \Vol(M_{L_{\boldsymbol\theta}})
- 2\pi \sum_{i \in I_0} k_i \mathrm{Im} (\alpha_{i}^*).
\end{align}
We have the following two cases:

{\bf{Case 1: $\mathrm{Im} (\alpha_{i_0}^*) \neq 0$ for some $i_0 \in I_0$} }

By Lemma \ref{fcneq00} and \ref{fcneq01}, since $|k_{i}^{E_{i}}(s_{i},A_{\zeta_{i}})|\geq  \frac{1}{q_{i}}$ for all $i\in I$, we have
\begin{align}\label{lbki}
 |k_i| \geq \frac{1}{q_i} - \frac{2|\beta_i-\pi|}{\pi q_{i}}>0.
\end{align}
Besides, recall that we choose $\mathbf E_I''$ in such a way that $k_{i_0} \im(\alpha_{i_0}^*) \leq 0$. As a result, from (\ref{onDHClarge21}), by making $\epsilon'>0$ smaller if necessary, \begin{align}\label{onDHClarge2}
&\im \Big(G^{\mathbf E_I, \mathbf A_{\zeta_I}, \mathbf B}(\mathbf s_I, \boldsymbol{\alpha}_{\zeta_I}, \boldsymbol \xi)\Big) \notag\\
\leq& \mathrm{Vol}(M_{L_{\boldsymbol\theta}}) - 2\pi \min\Bigg\{ \Bigg|\Big(\frac{1}{q_i} - \frac{2|\beta_i-\pi|}{\pi q_{i}}\Big)\mathrm{Im} (\alpha_{i}^*) \Bigg| \quad\Bigg\vert \quad i\in I, \mathrm{Im} (\alpha_{i}^*)\neq 0\Bigg\}\notag\\
\leq& \mathrm{Vol}(M_{L_{\boldsymbol\theta}}) - \epsilon'.
\end{align}
Next, on $S^{\mathbf E_I''}_{\text{sides}}$, by Proposition \ref{prop53}, $\im G^{\mathbf E_I''}$ is strictly concave up in $\{\im (\alpha_{i,\zeta_i})\}_{i \in I}$ and $\{\im (\xi_s)\}_{s=1}^c$.
Besides, 
$$
\im \Big(G^{\mathbf E_I, \mathbf A_{\zeta_I}, \mathbf B}(\mathbf s_I, \boldsymbol{\alpha}_{\zeta_I}, \boldsymbol \xi) - G^{\mathbf E_I''} (\boldsymbol \alpha_{\zeta_I}, \boldsymbol \xi)\Big)
= \im\Bigg(- 2\pi \sum_{i \in I_0} k_i(\boldsymbol \alpha_{i,\zeta_i} - \pi)\Bigg)  
$$
is a linear function in $\{\im (\alpha_{i,\zeta_i})\}_{i \in I}$. As a result, $\im \Big(G^{\mathbf E_I, \mathbf A_{\zeta_I}, \mathbf B}(\mathbf s_I, \boldsymbol{\alpha}_{\zeta_I}, \boldsymbol \xi)\Big)$ is also strictly concave up in $\{\im (\alpha_{i,\zeta_i})\}_{i \in I}$  and $\{\im (\xi_s)\}_{s=1}^c$. By convexity, for each $(\boldsymbol \alpha_{\zeta_I}, \boldsymbol \xi)\in \partial D_{\delta_0}$ and $t\in [0,1]$ we have
\begin{align*}
&\im \Big(G^{\mathbf E_I, \mathbf A_{\zeta_I}, \mathbf B}\Big(\mathbf s_I, (\boldsymbol \alpha_{\zeta_I}, \boldsymbol \xi) + t\sqrt{-1} \im(\mathbf z^{\mathbf E_I}) \Big)\Big)\\
<& \max\Big\{ \im \Big(G^{\mathbf E_I, \mathbf A_{\zeta_I}, \mathbf B}(\mathbf s_I,\boldsymbol \alpha_{\zeta_I}, \boldsymbol \xi)\Big),
\im \Big(G^{\mathbf E_I, \mathbf A_{\zeta_I}, \mathbf B}\Big(\mathbf s_I, (\boldsymbol \alpha_{\zeta_I}, \boldsymbol \xi) + \sqrt{-1} \im(\mathbf z^{\mathbf E_I}) \Big)\Big)\Big\}.
\end{align*}
For $(\boldsymbol \alpha_{\zeta_I}, \boldsymbol\xi) \in \partial D_{\delta_1}$, by (\ref{small1}) we have 
$$  \im G^{\mathbf E_I} (\boldsymbol \alpha_{\zeta_I}, \boldsymbol \xi) < \mathrm{Vol}(M_{L_{\boldsymbol\theta}}) - \epsilon' .$$
For $(\boldsymbol \alpha_{\zeta_I}, \boldsymbol \xi) + \sqrt{-1} \im(\mathbf z^{\mathbf E_I''})  \in S^{\mathbf E_I''}_{\text{top}}$, by (\ref{onDHClarge2}), we have
$$  \im G^{\mathbf E_I}( (\boldsymbol \alpha_{\zeta_I}, \boldsymbol \xi) + \sqrt{-1} \im(z^{\mathbf E_I}) ) < \mathrm{Vol}(M_{L_{\boldsymbol\theta}}) - \epsilon'.$$
Thus, we have
\begin{align}\label{small2}
&\lt| \int_{D_{\delta_0}} 
   (-1)^{\sum_{i\in I} A_{i,\zeta_i}} \phi_r\lt(\mathbf s_I, \boldsymbol{\alpha}_{\zeta_I}, \boldsymbol\xi \rt) 
    e^{\frac{r}{4\pi \sqrt{-1}}\lt( 
   W_r^{\mathbf E_I}\lt(\mathbf s_I, \boldsymbol \alpha_{\zeta_I}, \boldsymbol\xi\rt) 
    - \sum_{i \in I} 2\pi A_{i,\zeta_i} \alpha_{i,\zeta_i}  \rt)} d\boldsymbol{\alpha}_{\zeta_I} d\boldsymbol\xi
\rt| \notag\\
=&\lt| \int_{S^{\mathbf E_I''}} 
   (-1)^{\sum_{i\in I} A_{i,\zeta_i}} \phi_r\lt(\mathbf s_I, \boldsymbol{\alpha}_{\zeta_I}, \boldsymbol\xi \rt) 
    e^{\frac{r}{4\pi \sqrt{-1}}\lt( 
   W_r^{\mathbf E_I}\lt(\mathbf s_I, \boldsymbol \alpha_{\zeta_I}, \boldsymbol\xi\rt) 
    - \sum_{i \in I} 2\pi A_{i,\zeta_i} \alpha_{i,\zeta_i}  \rt)} d\boldsymbol{\alpha}_{\zeta_I} d\boldsymbol\xi
\rt| \notag\\
=& o\lt( e^{\frac{r}{4\pi} (\Vol(M_{L_{\boldsymbol\theta}})-\epsilon')} \rt) .
\end{align}
The result follows from (\ref{small1}) and (\ref{small2}).

{\bf{Case 2: $\mathrm{Im} (\alpha_{i}^*) = 0$ for all $i \in I_0$} }
From (\ref{onDHClarge201}) and (\ref{onDHClarge21}), on $S^{\mathbf E_I''}$ we have
\begin{align}\label{c2lessv}
\im \Big(G^{\mathbf E_I, \mathbf A_{\zeta_I}, \mathbf B}(\mathbf s_I, \boldsymbol{\alpha}_{\zeta_I}, \boldsymbol \xi)\Big) 
&\leq  \im  G^{\mathbf E_I''} (\mathbf z^{\mathbf E_I''})
- 2\pi \sum_{i \in I_0} k_i \mathrm{Im} (\alpha_{i}^*)
=  \Vol(M_{L_{\boldsymbol\theta}})
\end{align}
and equality holds if and only if $(\boldsymbol{\alpha}_{\zeta_I}, \boldsymbol \xi) = \mathbf z ^{\mathbf E_I''}$. Since $ \mathbf z ^{\mathbf E_I''}$ is a critical point of $G^{\mathbf E_I''} (\boldsymbol{\alpha}_{\zeta_I}, \boldsymbol \xi)$, there exists $\delta\in(0,\delta_0)$ depending of $\mathbf E_I''$ such that for any $i\in I$,
\begin{align}\label{dalpha1small}
\Bigg| \frac{\partial \im G^{\mathbf E_I''} (\boldsymbol{\alpha}_{\zeta_I}, \boldsymbol \xi)}{\partial \im \alpha_i} \Bigg|
< \pi \min_{i\in I} \Big\{ \frac{1}{q_i} - \frac{2|\beta_i-\pi|}{\pi q_{i}} \Big\}
\end{align}
whenever $d_{\infty}\big((\boldsymbol{\alpha}_{\zeta_I}, \boldsymbol \xi), \mathbf z ^{\mathbf E_I''} \big) < \delta$. Let 
$$
S^{\mathbf E_I'',\delta}
=\{(\boldsymbol \alpha_{\zeta_I}, \boldsymbol \xi) \in S^{\mathbf E_I''}_{\text{top}} \mid   d_{\infty}\big((\boldsymbol{\alpha}_{\zeta_I}, \boldsymbol \xi), \mathbf z ^{\mathbf E_I''} \big) \leq \delta\}
$$
By the compactness of the closure of $S^{\mathbf E_I''}\setminus S^{\mathbf E_I'',\delta}$, we can find $\epsilon'>0$ such that 
\begin{align}\label{Sdlv}
\im \Big(G^{\mathbf E_I, \mathbf A_{\zeta_I}, \mathbf B}(\mathbf s_I, \boldsymbol{\alpha}_{\zeta_I}, \boldsymbol \xi)\Big) 
< \Vol(M_{L_{\boldsymbol\theta}}) - \epsilon'
\end{align}
for any $(\boldsymbol{\alpha}_{\zeta_I}, \boldsymbol \xi) \in S^{\mathbf E_I''}\setminus S^{\mathbf E_I'',\delta}$. Thus, we have
\begin{align}\label{}
&\lt| \int_{S^{\mathbf E_I''}\setminus S^{\mathbf E_I'',\delta}} 
   (-1)^{\sum_{i\in I} A_{i,\zeta_i}} \phi_r\lt(\mathbf s_I, \boldsymbol{\alpha}_{\zeta_I}, \boldsymbol\xi \rt) 
    e^{\frac{r}{4\pi \sqrt{-1}}\lt( 
   W_r^{\mathbf E_I}\lt(\mathbf s_I, \boldsymbol \alpha_{\zeta_I}, \boldsymbol\xi\rt) 
    - \sum_{i \in I} 2\pi A_{i,\zeta_i} \alpha_{i,\zeta_i}  \rt)} d\boldsymbol{\alpha}_{\zeta_I} d\boldsymbol\xi
\rt| \notag\\
=& o\lt( e^{\frac{r}{4\pi} (\Vol(M_{L_{\boldsymbol\theta}})-\epsilon')} \rt) .
\end{align}

Assume that $k_i >0$ for some $i\in I_0$. For simplicity we assume that $i=1$. 
Consider the surface $S^{\mathbf E_I'',\delta,+} = S^{\mathbf E_I'',\delta,+}_{\text{top}} \cup S^{\mathbf E_I'',\delta,+}_{\text{bottom}}$ defined by
$$
S^{\mathbf E_I'',\delta,+}_{\text{top}} = 
\lt\{
  (\boldsymbol \alpha_{\zeta_I}, \boldsymbol \xi) +\sqrt{-1} (\delta,0,\dots,0) \mid
  (\boldsymbol \alpha_{\zeta_I} , \boldsymbol \xi)  \in S^{\mathbf E_I'',\delta}
  \rt\}
$$
and
$$
S^{\mathbf E_I'',\delta,+}_{\text{side}} = 
\left\{
  (\boldsymbol \alpha_{\zeta_I} , \boldsymbol \xi) + t\sqrt{-1} (\delta,0,\dots,0) \mid
 (\boldsymbol \alpha_{\zeta_I} , \boldsymbol \xi)  \in \partial S^{\mathbf E_I'',\delta}
\right\}.
$$
Note that on $S^{\mathbf E_I'',\delta,+}_{\text{top}}$, by the Mean Value Theorem, 
\begin{align}
&\im G_r^{\mathbf E_{I}, \mathbf A_{\zeta_I}, \mathbf B} \big(\mathbf s_I, (\boldsymbol \alpha_{\zeta_I}, \boldsymbol \xi) + \sqrt{-1}(\delta, 0, \dots, 0) \big)
- \im G_r^{\mathbf E_{I}, \mathbf A_{\zeta_I}, \mathbf B} (\mathbf s_I, \boldsymbol \alpha_{\zeta_I}, \boldsymbol \xi) \notag\\
= &
\frac{\partial \im G_r^{\mathbf E_{I}, \mathbf A_{\zeta_I}, \mathbf B}}{\partial \im \alpha_1} \big(\mathbf s_I, (\boldsymbol \alpha_{\zeta_I}, \boldsymbol \xi) + \sqrt{-1}(\delta',0, \dots, 0) \big) \cdot \delta'
\end{align}
for some $\delta' \in (0,\delta)$. Note that from (\ref{Gexpress}),
\begin{align*}
&\frac{\partial \im G_r^{\mathbf E_{I}, \mathbf A_{\zeta_I}, \mathbf B}}{\partial \im \alpha_1} \big(\mathbf s_I, (\boldsymbol \alpha_{\zeta_I}, \boldsymbol \xi) + \sqrt{-1}(\delta',0 \dots, 0) \big) \\
=&  \frac{\partial \im G_r^{\mathbf E_{I}''} }{\partial \im \alpha_1}\Bigg\vert_{(\boldsymbol \alpha_s,\xi_s+\sqrt{-1}\delta_0')} - 2\pi k_1
< -\pi \min_{i\in I} \Big\{ \frac{1}{q_i} - \frac{2|\beta_i-\pi|}{\pi q_{i}} \Big\},
\end{align*}
where the last inequality follows from (\ref{lbki}) and (\ref{dalpha1small}).
This implies that
\begin{align}\label{Sdtlv}
&\im G_r^{\mathbf E_{I}, \mathbf A_{\zeta_I}, \mathbf B} \big(\mathbf s_I, (\boldsymbol \alpha_{\zeta_I}, \boldsymbol \xi) + \sqrt{-1}(\delta, 0, \dots, 0) \big) \notag \\
<& \im G_r^{\mathbf E_{I}, \mathbf A_{\zeta_I}, \mathbf B} (\mathbf s_I, \boldsymbol \alpha_{\zeta_I}, \boldsymbol \xi)- \pi\min_{i\in I} \Big\{ \frac{1}{q_i} - \frac{2|\beta_i-\pi|}{\pi q_{i}} \Big\}\notag\\
\leq& \Vol(M_{L_{\boldsymbol\theta}})
- \pi \min_{i\in I} \Big\{ \frac{1}{q_i} - \frac{2|\beta_i-\pi|}{\pi q_{i}} \Big\},
\end{align}
where the last equality follows from (\ref{c2lessv}). 

Next, on $S^{\mathbf E_I'',\delta,+}_{\text{sides}}$, since
$\im \Big(G^{\mathbf E_I, \mathbf A_{\zeta_I}, \mathbf B}(\mathbf s_I, \boldsymbol{\alpha}_{\zeta_I}, \boldsymbol \xi)\Big)$ is strictly concave up in $\{\im (\alpha_{i,\zeta_i})\}_{i \in I}$  and $\{\im (\xi_s)\}_{s=1}^c$, for each $(\boldsymbol \alpha_{\zeta_I}, \boldsymbol \xi)\in \partial S^{\mathbf E_I'',\delta}$ and $t\in [0,1]$ we have
\begin{align*}
&\im G_r^{\mathbf E_{I}, \mathbf A_{\zeta_I}, \mathbf B} \big(\mathbf s_I, (\boldsymbol \alpha_{\zeta_I}, \boldsymbol \xi) + t\sqrt{-1}(\delta, 0, \dots, 0) \big)\\
<& \max\Big\{ \im \Big(G^{\mathbf E_I, \mathbf A_{\zeta_I}, \mathbf B}(\mathbf s_I,\boldsymbol \alpha_{\zeta_I}, \boldsymbol \xi)\Big),
\im \Big(G^{\mathbf E_I, \mathbf A_{\zeta_I}, \mathbf B}\Big(\mathbf s_I, (\boldsymbol \alpha_{\zeta_I}, \boldsymbol \xi) + \sqrt{-1}(\delta, 0, \dots, 0)  \Big)\Big)\Big\}.
\end{align*}
For $(\boldsymbol \alpha_{\zeta_I}, \boldsymbol\xi) \in \partial S^{\mathbf E_I'', \delta}$, by (\ref{Sdlv}) we have 
$$  \im G^{\mathbf E_I} (\boldsymbol \alpha_{\zeta_I}, \boldsymbol \xi) < \mathrm{Vol}(M_{L_{\boldsymbol\theta}}) - \epsilon' .$$
For $(\boldsymbol \alpha_{\zeta_I}, \boldsymbol \xi) + \sqrt{-1}(\delta, 0, \dots, 0) \in S^{\mathbf E_I'',\delta,+}_{\text{top}}$, by (\ref{Sdtlv}), by making $\epsilon'$ smaller if necessary, we have
$$  \im G^{\mathbf E_I}( (\boldsymbol \alpha_{\zeta_I}, \boldsymbol \xi) + \sqrt{-1} \im(z^{\mathbf E_I}) ) < \mathrm{Vol}(M_{L_{\boldsymbol\theta}}) - \epsilon'.$$
Thus, we have
\begin{align*}
&\lt| \int_{D_{\delta_0}} 
   (-1)^{\sum_{i\in I} A_{i,\zeta_i}} \phi_r\lt(\mathbf s_I, \boldsymbol{\alpha}_{\zeta_I}, \boldsymbol\xi \rt) 
    e^{\frac{r}{4\pi \sqrt{-1}}\lt( 
   W_r^{\mathbf E_I}\lt(\mathbf s_I, \boldsymbol \alpha_{\zeta_I}, \boldsymbol\xi\rt) 
    - \sum_{i \in I} 2\pi A_{i,\zeta_i} \alpha_{i,\zeta_i}  \rt)} d\boldsymbol{\alpha}_{\zeta_I} d\boldsymbol\xi
\rt| \notag\\
=&\lt| \int_{S^{\mathbf E_I''}} 
   (-1)^{\sum_{i\in I} A_{i,\zeta_i}} \phi_r\lt(\mathbf s_I, \boldsymbol{\alpha}_{\zeta_I}, \boldsymbol\xi \rt) 
    e^{\frac{r}{4\pi \sqrt{-1}}\lt( 
   W_r^{\mathbf E_I}\lt(\mathbf s_I, \boldsymbol \alpha_{\zeta_I}, \boldsymbol\xi\rt) 
    - \sum_{i \in I} 2\pi A_{i,\zeta_i} \alpha_{i,\zeta_i}  \rt)} d\boldsymbol{\alpha}_{\zeta_I} d\boldsymbol\xi
\rt| \notag\\
\leq&\lt| \int_{S^{\mathbf E_I''} \setminus S^{\mathbf E_I'',\delta}} 
   (-1)^{\sum_{i\in I} A_{i,\zeta_i}} \phi_r\lt(\mathbf s_I, \boldsymbol{\alpha}_{\zeta_I}, \boldsymbol\xi \rt) 
    e^{\frac{r}{4\pi \sqrt{-1}}\lt( 
   W_r^{\mathbf E_I}\lt(\mathbf s_I, \boldsymbol \alpha_{\zeta_I}, \boldsymbol\xi\rt) 
    - \sum_{i \in I} 2\pi A_{i,\zeta_i} \alpha_{i,\zeta_i}  \rt)} d\boldsymbol{\alpha}_{\zeta_I} d\boldsymbol\xi
\rt| \notag\\
&+\lt| \int_{S^{\mathbf E_I'',\delta,+}} 
   (-1)^{\sum_{i\in I} A_{i,\zeta_i}} \phi_r\lt(\mathbf s_I, \boldsymbol{\alpha}_{\zeta_I}, \boldsymbol\xi \rt) 
    e^{\frac{r}{4\pi \sqrt{-1}}\lt( 
   W_r^{\mathbf E_I}\lt(\mathbf s_I, \boldsymbol \alpha_{\zeta_I}, \boldsymbol\xi\rt) 
    - \sum_{i \in I} 2\pi A_{i,\zeta_i} \alpha_{i,\zeta_i}  \rt)} d\boldsymbol{\alpha}_{\zeta_I} d\boldsymbol\xi
\rt| \notag\\
=& o\lt( e^{\frac{r}{4\pi} (\Vol(M_{L_{\boldsymbol\theta}})-\epsilon')} \rt) .
\end{align*}
This finishes the proof under the assumption that $k_1>0$. For $k_1<0$, we consider the surface $S^{\mathbf E_I'',\delta,-} = S^{\mathbf E_I'',\delta,-}_{\text{top}} \cup S^{\mathbf E_I'',\delta,-}_{\text{bottom}}$ defined by
$$
S^{\mathbf E_I'',\delta,-}_{\text{top}} = 
\lt\{
  (\boldsymbol \alpha_{\zeta_I}, \boldsymbol \xi) -\sqrt{-1} (\delta,0,\dots,0) \mid
  (\boldsymbol \alpha_{\zeta_I} , \boldsymbol \xi)  \in S^{\mathbf E_I'',\delta}
  \rt\}
$$
and
$$
S^{\mathbf E_I'',\delta,-}_{\text{side}} = 
\left\{
  (\boldsymbol \alpha_{\zeta_I} , \boldsymbol \xi) - t\sqrt{-1} (\delta,0,\dots,0) \mid
 (\boldsymbol \alpha_{\zeta_I} , \boldsymbol \xi)  \in \partial S^{\mathbf E_I'',\delta}
\right\}.
$$
The result follows from a similar argument as the previous case.

\end{proof}

According to Proposition \ref{fail0} and \ref{fail1}, it remains to study the asymptotics of the $\widehat{f_r}(\mathbf{s}, \mathbf{A_{\zeta_I}}, \mathbf{0})$ with $(s_i, A_{\zeta_i}) \in S_i$ for all $i\in I$. where $S_i$ is defined in (\ref{defS_i}). If $|q_i|$'s are odd for all $i\in I$, the asymptotics of $\widehat{f_r}(\mathbf{s}^{\mathbf E_I}, \mathbf{1-2m^{E_I}}, \mathbf{0})$ are given in Proposition \ref{0FCasym}. When some $|q_i|$ is even, the following proposition shows that the leading terms in the asymptotics of the other Fourier coefficients cancel out with each other.

\begin{proposition}\label{fail2} Suppose the assumptions in Proposition \ref{lfc} hold. Further suppose there exists $i_0 \in I$ such that $|q_{i_0}|$ is even. Then for every pair $(\mathbf{E_I'}, \mathbf{s'}, \mathbf{A_{\zeta_I}'}, \mathbf{0})$ and $(\mathbf{E_I''},\mathbf{s''}, \mathbf{A_{\zeta_I}''}, \mathbf{0})$ with $E'_{i_0}=-1, E''_{i_0}=1$, $E'_i = E''_i$ for all $i \in I \setminus\{i_0\}$, $(s_{i_0}', A_{\zeta_{i_0}}') = (\tilde s_{i_0}^+, -2\tilde m_{i_0}^+)$, $(s_{i_0}'', A_{\zeta_{i_0}}'') = (\tilde s_{i_0}^-, -2\tilde m_{i_0}^-)$ and $(s_i', A_{\zeta_i}') = (s_i'', A_{\zeta_i}'')$ for all $i \in I \setminus\{i_0\}$, we have
\begin{align*}
\widehat{f_r}^{\mathbf E_I}(\mathbf{s'}, \mathbf{A_{\zeta_I}'}, \mathbf{0})
&= C'_r\frac{e^{\frac{1}{2}\sum_{k=1}^{n}\mu_k\mathrm H^{(r)}(\gamma_k)}}{\sqrt{\pm\mathbb{T}_{(M\setminus L, \mathrm{\mathbf\Upsilon})}([\rho_{M^{(r)}}])}}e^{\frac{r}{4\pi}\big(\mathrm{Vol}(M^{(r)})+\sqrt{-1}\mathrm{CS}(M^{(r)})\big)}\bigg(1+O\Big(\frac{1}{r}\Big)\bigg)
\end{align*}
and
\begin{align*}
\widehat{f_r}^{\mathbf E_I}(\mathbf{s''}, \mathbf{A_{\zeta_I}''}, \mathbf{0})
&= -C'_r\frac{e^{\frac{1}{2}\sum_{k=1}^{n}\mu_k\mathrm H^{(r)}(\gamma_k)}}{\sqrt{\pm\mathbb{T}_{(M\setminus L, \mathrm{\mathbf\Upsilon})}([\rho_{M^{(r)}}])}}e^{\frac{r}{4\pi}\big(\mathrm{Vol}(M^{(r)})+\sqrt{-1}\mathrm{CS}(M^{(r)})\big)}\bigg(1+O\Big(\frac{1}{r}\Big)\bigg)
\end{align*}
for some sequence of complex number $C'_r$ with norm $1$. 
\end{proposition}
\begin{proof} We first study the asymptotics of $\widehat{f_r}^{\mathbf E_I}(\mathbf{s'}, \mathbf{A_{\zeta_I}'}, \mathbf{0})$. Note that by Proposition \ref{formulaFC},
\begin{align*}
   & \widehat{f_r^{\mathbf E_I}}(\mathbf{s_I'}, \mathbf A'_{\zeta_I},  \boldsymbol  B) =
    \frac{r^{|I| + c}\big(\prod_{i \in I} E_i\big) }{2^{|I| + c} \pi ^{|I| + c}}  
    \\& \times\int_{D_H} 
   (-1)^{\sum_{i\in I} A_{i,\zeta_i}'} \phi_r\lt(\mathbf s_I, \boldsymbol{\alpha}_{\zeta_I}, \boldsymbol\xi \rt)   
    e^{\frac{r}{4\pi \sqrt{-1}}\lt( 
   W_r^{\mathbf E_I}\lt(\mathbf{s_I'}, \boldsymbol \alpha_{\zeta_I}, \boldsymbol\xi\rt) 
    - \sum_{i \in I} 2\pi A_{i,\zeta_i}' \alpha_{i,\zeta_i}  \rt)} d\boldsymbol{\alpha}_{\zeta_I} d\boldsymbol\xi,
\end{align*}
By Lemma \ref{fcneq00}, \ref{fcneq01} and a direct computation, we can write
\begin{align*}
&W_r^{\mathbf E_I}\lt(\mathbf{s'}, \boldsymbol \alpha_{\zeta_I}, \boldsymbol\xi\rt) 
    - \sum_{i \in I} 2\pi A_{i,\zeta_i}' \alpha_{i,\zeta_i}\\
=&  G_r^{\boldsymbol  E_I}(\boldsymbol \alpha_{\zeta_I}, \boldsymbol \xi) - 2\pi^2\sum_{i\in I}A_{\zeta_i}'
 + \sum_{i\in I}2\pi\beta_i\Big(-E_iJ_i(s_i') - \frac{p_i}{q_i} \Big) +\sum_{i\in I}\pi^2\Bigg(K_i(s_i')+\frac{p_i'}{q_i}\Bigg) +\frac{4\pi^2}{r^2} h_I .
\end{align*}
Moreover, by Lemma \ref{fcneq01} and Lemma \ref{arith} (2), since $\tilde s_i^\pm = s_i^\pm + \frac{q_i}{2} \pmod{q_i}$, we have
$$J_i(\tilde s_i^\pm) - J_i(\tilde s_i) = \mp \frac{p_i'}{q_i} \pmod{\ZZ}.$$
Thus, similar to the proof of Proposition \ref{lfcexpress}, we can write
\begin{align}
\widehat{f_r^{\mathbf E_I}}(\mathbf{s_I'}, \mathbf A'_{\zeta_I},  \boldsymbol  {0})&=\frac{Y'(\mathbf E_I)r^{|I| + c} }{2^{|I| + c} \pi ^{|I| + c}} \int_{D_H} 
    \phi_r\lt(\mathbf{s}^{\mathbf E_I}, \boldsymbol{\alpha}_{\zeta_I}, \boldsymbol\xi \rt)     
   e^{\frac{r}{4\pi \sqrt{-1}}G_r^{E_I}(\boldsymbol\alpha_{\xi_I},\boldsymbol\zeta)} d\boldsymbol{\alpha}_{\zeta_I} d\boldsymbol\xi, \label{lfcexpress2}
\end{align}
where
\begin{align}\label{defY'EI}
Y'(\mathbf E_I) =
- (-1)^{\sum_{i\in I} \Big(\frac{p_i'}{q_i} + E_i J_i(s_i')\Big) + |I|}\Bigg(\prod_{i \in I} E_i\Bigg)e^{\frac{r\pi}{4\sqrt{-1}}\sum_{i\in I}\Big(-2A'_{\zeta_i}+ K_i(s_i')+\frac{p_i'}{q_i} \Big)},
\end{align}
\begin{align*}
&\phi_r(\mathbf s_I, \boldsymbol{\alpha}_{\zeta_I}, \boldsymbol\xi)=
\psi(\boldsymbol{\alpha}_{\zeta_I}, \boldsymbol\xi)  \\
&\times
e^{ \sqrt{-1}\Big(\sum_{i\in I}\Big(\frac{p_i'}{q_i}(\beta_i - \pi) + \frac{p_i}{q_i}(\alpha_{i,\zeta_i} - \pi) + \frac{E_i(\alpha_{i,\zeta_i}+\beta_i-2\pi)}{q_i}\Big)+ \sum_{i\in I}a_{i,0}\beta_i + \sum_{i\in I} \Big(\frac{\iota_i}{2}\Big)\alpha_{i,\zeta_i} + \sum_{j\in J}\Big(a_{j,0}+\frac{\iota_j}{2}\Big)\alpha_j \Big)},
\end{align*}
and
\begin{align*}
G_r^{\boldsymbol  E_I}(\boldsymbol \alpha_{\zeta_I}, \boldsymbol \xi) 
&=  \sum_{i\in I} 
\lt[- \lt(\frac{p_i'}{q_i} + a_{i,0}\rt) (\beta_i - \pi)^2 - \frac{p_i (\alpha_{i,\zeta_i} - \pi)^2 + 2 E_i(\beta_i - \pi)(\alpha_{i,\zeta_i} - \pi)}{q_i}
\rt] \notag
\\
&\qquad - \sum_{j \in J}\Big( a_{j,0} + \frac{\iota_j}{2}\Big)(\alpha_j - \pi)^2  -\sum_{i\in I}\frac{\iota_i}{2}(\alpha_{i,\zeta_i}-\pi)^2\notag\\
&\qquad +\sum_{s=1}^c U_r(\alpha_{s_1},\dots,\alpha_{s_6},\xi_s)+\Big(\sum_{i=1}^n\frac{\iota_i}{2}\Big)\pi^2.
\end{align*}
By (\ref{lfcexpress2}), Proposition \ref{lfc} and \ref{Rtorsion},
\begin{align*}
&\widehat{f_r^{\mathbf E_I}}(\mathbf{s_I'}, \mathbf A'_{\zeta_I},  \boldsymbol  {0}) \\
&=
    \frac{Y'(\mathbf E_I) r^{|I| + c} }
    { 2^{|I| + c}\pi^{|I|+c} } \lt(\frac{2}{r}\rt)^c \lt(\frac{2\pi}{r}\rt)^{\frac{|I|+c}{2}} (4\pi\sqrt{-1})^{\frac{|I|+c}{2}} 
\\
&\qquad \frac{C^{\mathbf E_I}(\mathbf{z}^{\boldsymbol E_{I}})}{\sqrt{- \lt(\prod_{i\in I}q_i\rt)\det \Hess({G}^{\mathbf E_I} )(\mathbf{z^{\mathbf E_I}})}} e^{\frac{r}{4 \pi} (\Vol (M_{L_{\boldsymbol\theta}})  + \sqrt{-1}\CS(M_{L_{\boldsymbol\theta}}))} \Big( 1 + O\Big(\frac{1}{r}\Big ) \Big)\\
&= 
C'_r\frac{e^{\frac{1}{2}\sum_{k=1}^{n}\mu_k\mathrm H^{(r)}(\gamma_k)}}{\sqrt{\pm\mathbb{T}_{(M\setminus L, \mathrm{\mathbf\Upsilon})}([\rho_{M^{(r)}}])}}e^{\frac{r}{4\pi}\big(\mathrm{Vol}(M^{(r)})+\sqrt{-1}\mathrm{CS}(M^{(r)})\big)}\bigg(1+O\Big(\frac{1}{r}\Big)\bigg), 
\end{align*}
where $C'_r = Y'(\mathbf E_I) 2^{-|I|+c}r^{ \frac{|I|-c}{2}}
(-1)^{- \frac{rc}{2} + \frac{|I|+c}{4} +\sum_{i\in I} \Big(a_{i,0}+\frac{\iota_i}{2}\Big)+\sum_{j\in J}\Big(a_{j,0}+\frac{\iota_j}{2}\Big)}$

By the same argument, we have
\begin{align*}
&\widehat{f_r^{\mathbf E_I}}(\mathbf{s_I''}, \mathbf A''_{\zeta_I},  \boldsymbol  {0}) \\
&= 
C''_r\frac{e^{\frac{1}{2}\sum_{k=1}^{n}\mu_k\mathrm H^{(r)}(\gamma_k)}}{\sqrt{\pm\mathbb{T}_{(M\setminus L, \mathrm{\mathbf\Upsilon})}([\rho_{M^{(r)}}])}}e^{\frac{r}{4\pi}\big(\mathrm{Vol}(M^{(r)})+\sqrt{-1}\mathrm{CS}(M^{(r)})\big)}\bigg(1+O\Big(\frac{1}{r}\Big)\bigg), 
\end{align*}
where $C''_r = Y''(\mathbf E_I) 2^{-|I|+c}r^{ \frac{|I|-c}{2}}
(-1)^{- \frac{rc}{2} + \frac{|I|+c}{4} +\sum_{i\in I} \Big(a_{i,0}+\frac{\iota_i}{2}\Big)+\sum_{j\in J}\Big(a_{j,0}+\frac{\iota_j}{2}\Big)}$ with
\begin{align}\label{defY''EI}
Y''(\mathbf E_I) =
- (-1)^{\sum_{i\in I} \Big(\frac{p_i'}{q_i} + E_i J_i(s_i'')\Big) + |I|}\Bigg(\prod_{i \in I} E_i\Bigg)e^{\frac{r\pi}{4\sqrt{-1}}\sum_{i\in I}\Big(-2A''_{\zeta_i}+ K_i(s_i'')+\frac{p_i'}{q_i} \Big)}.
\end{align}
To prove the proposition, it suffices to study the ratio of $C'_r$ and $C''_r$. Note that from (\ref{defY'EI}) and (\ref{defY''EI}),
\begin{align}\label{pmpart0}
\frac{C'_r}{C''_r}
= \frac{Y'(\mathbf E_I)}{Y''(\mathbf E_I)}
= -(-1)^{-(J_{i_0}(\tilde{s}_{i_0}^+) + J_{i_0}(\tilde{s}_{i_0}^-))}e^{\frac{r\pi}{4\sqrt{-1}}(K_{i_0}(s'_{i_0}) - K_{i_0}(s''_{i_0})+ 4(\tilde{m}_{i_0}' - \tilde{m}_{i_0}''))}.
\end{align}
From Lemma \ref{arith} (2), we know that
$$J_{i_0}(s_{i_0}^+)\equiv  -J_{i_0}(s_{i_0}^-)\quad(\text{mod }2),$$
From Lemma \ref{fcneq01}, we know that $\tilde s_{i_0}^\pm = s_{i_0}^\pm + \frac{q_{i_0}}{2} \pmod{q_{i_0}}$. Thus, we have
$J_{i_0}(\tilde s_{i_0}^+) - J_{i_0}(s_{i_0}^+) \equiv 1 \pmod{2}$, $J_{i_0}(\tilde s_{i_0}^-) - J_{i_0}(s_{i_0}^+) \equiv 1 \pmod{2}$ and
$$
J_{i_0}(\tilde{s}_{i_0}^+) + J_{i_0}(\tilde{s}_{i_0}^-)
= \big( J_{i_0}(\tilde s_{i_0}^+) - J_{i_0}(s_{i_0}^+) \big)
+\big( J_{i_0}(\tilde s_{i_0}^-) - J_{i_0}(s_{i_0}^+) \big)
+\big(J_{i_0}(s_{i_0}^+) + J_{i_0}(s_{i_0}^-) \big)
\equiv 0 \pmod{2}.
$$
This implies that
\begin{align}\label{pmpart1} 
(-1)^{E_i (J_{i_0}(\tilde{s}_{i_0}^+) + J_{i_0}(\tilde{s}_{i_0}^-))} = 1. 
\end{align}
From the definition of $K$ in Lemma \ref{arith} (3), we get
\begin{align}\label{K-K2}
&K_{i_0}(\tilde s_{i_0}^+) - K_{i_0}(\tilde s_{i_0}^-) + 4(\tilde m_{i_0}^+ - \tilde m_{i_0}^-)\notag\\
=& \frac{4C_{i_0,\xi_{i_0}-1}}{q_{i_0}}(\tilde s_{i_0}^+ + \tilde s_{i_0}^- + 1 + K_{i_0,\xi_{i_0}-1})(\tilde s_{i_0}^+-\tilde s_{i_0}^-) + 4(\tilde m_{i_0}^+ - \tilde m_{i_0}^-).
\end{align}
Besides, from the definition of $I$ and (\ref{deftm}),
\begin{align}\label{I+I2}
I_{i_0}(\tilde s_{i_0}^+) + I_{i_0}(\tilde s_{i_0}^-)
= -2 C_{i_0,\xi_{i_0}-1}(\tilde s_{i_0}^+ +\tilde  s_{i_0}^- + 1 + K_{i_0,\xi_{i_0}-1})
= 2q_{i_0}(\tilde m_{i_0}^+ +\tilde  m_{i_0}^- ).
\end{align}
From (\ref{K-K2}) and (\ref{I+I2}), we have
$$
K_{i_0}(\tilde s_{i_0}^+) - K_{i_0}(\tilde s_{i_0}^-) + 4(\tilde m_{i_0}^+ - \tilde m_{i_0}^-)
= 4\big((-\tilde m_{i_0}^+ - \tilde m_{i_0}^-)(\tilde s_{i_0}^+-\tilde s_{i_0}^-)+\tilde m_{i_0}^+ + \tilde m_{i_0}^-\big).
$$
In particular,
\begin{align*}
e^{\frac{r\pi}{4\sqrt{-1}}(K_{i_0}(s'_{i_0}) - K_{i_0}(s''_{i_0})+ 4(\tilde{m}_{i_0}' - \tilde{m}_{i_0}''))}
&= (-1)^{(\tilde m_{i_0}^+ + \tilde m_{i_0}^- )(\tilde s_{i_0}^+-\tilde s_{i_0}^- - 1)}.
\end{align*}
From Lemma \ref{fcneq01}, we know that $\tilde s_{i_0}^\pm = s_{i_0}^\pm + \frac{q_{i_0}}{2} \pmod{q_{i_0}}$. Since $q_{i_0}$ is even, we have
$$
\tilde{s}_{i_0}^+ - \tilde{s}_{i_0}^-
= \big( \tilde s_{i_0}^+ - s_{i_0}^+ \big)
-\big( \tilde s_{i_0}^- - s_{i_0}^- \big)
+\big(s_{i_0}^+ - s_{i_0}^- \big)
\equiv s_{i_0}^+ - s_{i_0}^- \pmod{2}.
$$
From Lemma \ref{arith} (1), we know that $s_{i_0}^+ - s_{i_0}^- \equiv p_{i_0}' \pmod{q_{i_0}}$. Moreover, since $p_{i_0}p'_{i_0} + q_{i_0}q'_{i_0} = 1$ and $q_{i_0}$ is even, $p_{i_0}'$ must be odd. Altogether, we have
$$
\tilde{s}_{i_0}^+ - \tilde{s}_{i_0}^- - 1 
\equiv s_{i_0}^+ - s_{i_0}^- - 1
\equiv 0 \pmod{2}
$$
and
\begin{align}\label{pmpart2}
e^{\frac{r\pi}{4\sqrt{-1}}(K_{i_0}(s'_{i_0}) - K_{i_0}(s''_{i_0})+ 4(\tilde{m}_{i_0}' - \tilde{m}_{i_0}''))}
&= (-1)^{(\tilde m_{i_0}^+ + \tilde m_{i_0}^- )(\tilde s_{i_0}^+-\tilde s_{i_0}^- - 1)} = 1.
\end{align}
From (\ref{pmpart0}), (\ref{pmpart1}) and (\ref{pmpart2}), we get
$C''_r = -C'_r$. This completes the proof.
\end{proof}

\subsection{Estimate of error term and the proof of the main theorems}

The following proposition shows that the error term in Proposition \ref{Fperror} is negligible compared to the leading Fourier coefficient. 
\begin{proposition} \label{errorest} There exists $\delta_0>0$ such that if  
$$ \Vol(M_{L_{\boldsymbol\theta}})> \max_{(\boldsymbol \alpha_{\zeta_I}, \boldsymbol \xi) \in \overline{D_H\setminus D_{\delta_0}}} \im \tilde U(\boldsymbol \alpha_{\zeta_I}, \boldsymbol \xi) ,$$
where $ \tilde U(\boldsymbol \alpha_{\zeta_I}, \boldsymbol \xi) $ is defined in (\ref{deftU}) and $\overline{D_H\setminus D_{\delta_0}}$ is the closure of $D_H\setminus D_{\delta_0}$, then there exists $\epsilon'>0$ such that the error term in Proposition \ref{Fperror} is less than $O(e^{\frac{r}{4 \pi} (\Vol(M_{L_{\boldsymbol\theta}}) - \epsilon')})$.
\end{proposition}\begin{proof}For a fixed $\boldsymbol{\alpha}_J=(\alpha_j)_{j\in J},$ let 
\begin{equation*}
M_{\boldsymbol{\alpha}_J}=\max\Big\{\sum_{s=1}^c2V(\alpha_{s_1},\dots,\alpha_{s_6},\xi_s)\ \Big|\ (\boldsymbol{\alpha}_{\zeta_I},\boldsymbol{\xi})\in\partial \mathrm {D_H}\cup\big(\mathrm {D_A}\setminus \mathrm{D_H}\big)\Big\}
\end{equation*}
where $V$ is as defined in (\ref{defV}). Then by \cite[Sections 3 \& 4]{BDKY}, 
$$M_{\boldsymbol{\alpha}_J}<2cv_8.$$
Besides, we know that $\im \tilde U (\boldsymbol{\alpha}_{\zeta_I},\boldsymbol{\xi}) \leq 2cv_8$ and equality holds if and only if 
$$(\boldsymbol{\alpha}_{\zeta_I},\boldsymbol{\xi}) = ( \pi,\dots,\pi, \frac{7\pi}{4},\dots,\frac{7\pi}{4}).$$ 
As a result, we can choose $\delta_0>0$ sufficiently small so that 
$$
M_{\boldsymbol{\alpha}_J} < \max_{(\boldsymbol \alpha_{\zeta_I}, \boldsymbol \xi) \in \overline{D_H\setminus D_{\delta_0}}} \im \tilde U(\boldsymbol \alpha_{\zeta_I}, \boldsymbol \xi) .
$$
The result follows from the fact that the error terms in Proposition \ref{Fperror} contains those $g_r^{\mathbf E_I}(\mathbf s_I, \mathbf m_{\zeta_I}, \mathbf k)$ with $({\mathbf m_{\zeta_I}},{\mathbf k}) \in \mathrm {D_H}\cup\big(\mathrm {D_A}\setminus \mathrm{D_H}\big)$.
\end{proof}

\begin{lemma}\label{assat0}
There exists $\delta>0$ such that if $|\mathrm{H}(u_k)|<\delta$ for all $k=1,\dots,n$, then we have
$\mathbf {z}^{ \boldsymbol E_{I}} \in D_{\delta_0,\CC}$ and 
$$ \Vol(M_{L_{\boldsymbol\theta}})> \max\Bigg\{\max_{(\boldsymbol \alpha_{\zeta_I}, \boldsymbol \xi) \in \overline{D_H\setminus D_{\delta_0}}} \im \tilde U(\boldsymbol \alpha_{\zeta_I}, \boldsymbol \xi), 2cv_8 - 4\pi \delta_0
\Bigg\} ,$$
where $ \tilde U(\boldsymbol \alpha_{\zeta_I}, \boldsymbol \xi) $ is defined in (\ref{deftU}) and $\overline{D_H\setminus D_{\delta_0}}$ is the closure of $D_H\setminus D_{\delta_0}$.
\end{lemma}
\begin{proof}
Note that by Proposition \ref{prop52}, we have 
\begin{align*}
|\alpha^*_i - \pi|=\frac{|\mathrm H(u_i)|}{2} < \frac{\delta}{2}.
\end{align*}
Moreover, $\{\xi_s\}_{s=1}^c$ depends continuously on $\{\alpha_k\}_{k=1}^n$ with $\xi_s(\pi,\dots,\pi)=\frac{7\pi}{4}$ for all $s=1,\dots,c$. Altogether, by choosing $\delta>0$ sufficiently small, we have $\mathbf {z}^{ \boldsymbol E_{I}} \in D_{\delta_0,\CC}$. Besides, $\Vol(M_{L_{\boldsymbol\theta}})$ depends continuously on $\{\mathrm{H}(u_k)\}_{k=1}^n$ and is equal to $2cv_8$ when $\mathrm{H}(u_k)=0$ for $k=1,\dots,n$. Moreover, 
$$ 2cv_8 > \max\Bigg\{\max_{(\boldsymbol \alpha_{\zeta_I}, \boldsymbol \xi) \in \overline{D_H\setminus D_{\delta_0}}} \im \tilde U(\boldsymbol \alpha_{\zeta_I}, \boldsymbol \xi), 2cv_8 - 4\pi \delta_0
\Bigg\} .$$
By choosing $\delta>0$ sufficiently small, we have 
$$ \Vol(M_{L_{\boldsymbol\theta}})> \max\Bigg\{\max_{(\boldsymbol \alpha_{\zeta_I}, \boldsymbol \xi) \in \overline{D_H\setminus D_{\delta_0}}} \im \tilde U(\boldsymbol \alpha_{\zeta_I}, \boldsymbol \xi), 2cv_8 - 4\pi \delta_0
\Bigg\} .$$
\end{proof}

\begin{lemma}\label{assat1}
There exists $\epsilon>0$ such that whenever $\theta_i, \theta_j \in [0, \epsilon)$ for all $i\in I$ and $j\in J$, we have
$\mathbf {z}^{ \boldsymbol E_{I}} \in D_{\delta_0,\CC}$ and 
$$ \Vol(M_{L_{\boldsymbol\theta}})> \max\Bigg\{\max_{(\boldsymbol \alpha_{\zeta_I}, \boldsymbol \xi) \in \overline{D_H\setminus D_{\delta_0}}} \im \tilde U(\boldsymbol \alpha_{\zeta_I}, \boldsymbol \xi), 2cv_8 - 4\pi \delta_0
\Bigg\} ,$$
where $ \tilde U(\boldsymbol \alpha_{\zeta_I}, \boldsymbol \xi) $ is defined in (\ref{deftU}) and $\overline{D_H\setminus D_{\delta_0}}$ is the closure of $D_H\setminus D_{\delta_0}$.\end{lemma}
\begin{proof}
First, when $\beta_i = \alpha_j = \pi$ for all $i\in I, j \in J$, we have $(\theta_1,\dots,\theta_n) = \mathbf{0} = (0,\dots,0)$, $\mathbf {z}^{ \boldsymbol E_{I}} = (\pi,\dots,\pi, \frac{7\pi}{4}, \dots, \frac{7\pi}{4}) \in D_{\delta_0,\CC}$ and  
$$ \Vol(M_{L_{\mathbf{0}}}) = 2cv_8 > \max\Bigg\{\max_{(\boldsymbol \alpha_{\zeta_I}, \boldsymbol \xi) \in \overline{D_H\setminus D_{\delta_0}}} \im \tilde U(\boldsymbol \alpha_{\zeta_I}, \boldsymbol \xi), 2cv_8 - 4\pi \delta_0
\Bigg\}.$$

By continuity, there exists $\epsilon>0$ such that if $\{ \beta_i\}_{i \in I}$ and $\{ \alpha_j\}_{j \in J}$ are all in $(\pi - \epsilon, \pi + \epsilon)$, then the critical point $\mathbf z^{\mathbf E_I}$ of $G^{\mathbf E_I}$ in Proposition \ref{prop52} lies in $D_{\delta_0, \CC}$, and $\Vol(M_{L_{\boldsymbol\theta}})$ is sufficiently close to $ \Vol(M_{L_{\mathbf{0}}})=2cv_8$ so that
\begin{align*}
\Vol(M_{L_{\boldsymbol\theta}})
 >
 \max\Bigg\{\max_{(\boldsymbol \alpha_{\zeta_I}, \boldsymbol \xi) \in \overline{D_H\setminus D_{\delta_0}}} \im \tilde U(\boldsymbol \alpha_{\zeta_I}, \boldsymbol \xi), 2cv_8 - 4\pi \delta_0
\Bigg\}.
 \end{align*}
\end{proof}

\begin{lemma}\label{assat2}
There exists $\epsilon>0$ and $C>0$ such that whenever $\theta_j \in [0, \epsilon)$ for all $j\in J$, $|p_i|+|q_i|>C$ and $\theta_i \in [0,\pi)$  for all $i \in I$, we have
$\mathbf {z}^{ \boldsymbol E_{I}} \in D_{\delta_0,\CC}$ and 
$$ \Vol(M_{L_{\boldsymbol\theta}})> \max\Bigg\{\max_{(\boldsymbol \alpha_{\zeta_I}, \boldsymbol \xi) \in \overline{D_H\setminus D_{\delta_0}}} \im \tilde U(\boldsymbol \alpha_{\zeta_I}, \boldsymbol \xi), 2cv_8 - 4\pi \delta_0
\Bigg\} ,$$
where $ \tilde U(\boldsymbol \alpha_{\zeta_I}, \boldsymbol \xi) $ is defined in (\ref{deftU}) and $\overline{D_H\setminus D_{\delta_0}}$ is the closure of $D_H\setminus D_{\delta_0}$.\end{lemma}
\begin{proof}
Let $\delta>0$ be the constant in Lemma \ref{assat0}. 
For each $k\in \{1,2,\dots,n\}$, recall that the generalized Dehn filling invariant of the logarithmic holonomy $\mathrm{H}(u_k)$ around $0\in \CC$ is defined by sending $0$ to $\infty\in \RR^2 \cup \{\infty\} = \mathbb{S}^2$ and  sending $\mathrm{H}(u_k)\neq 0$ to the unique pair $(p_k, q_k) \in \mathbb{S}^2$ satisifying 
$$p_k\mathrm{H}(u_k)+q_k\mathrm{H}(v_k) = 2\pi \sqrt{-1}.$$  It is well-known that the generalized Dehn filling invariant gives a local homeomorphism from an open neighborhood of $(0,\dots,0) \in \CC^{n}$ to an open neighborhood around $(\infty,\dots, \infty)\in(\SS^2)^n$ by sending the logarithmic holonomies $(\mathrm{H}(u_1),\dots, \mathrm{H}(u_n))$ to the generalized Dehn filling invariants $( (p_1,q_1), \dots, (p_n, q_n))$ (see e.g. Corollary 15.2.17 and Proposition 15.3.1 in \cite{BM}). In particular, there exists $C>0$ such that whenever $|p_k|+|q_k|>C$ for all $k=1,\dots,n$, we have $|\mathrm{H}(u_k)| < \delta$ for all $k=1,\dots,n$. 

Note that for $i\in I$ with $|p_i|+|q_i|>C$ and $\theta_i \in (0,\pi)$, the equation $p_i \mathrm{H}(u_i) + q_i \mathrm{H}(v_i) = \theta \sqrt{-1}$ implies that
$$\Big(\frac{2\pi p_i}{\theta_i}\Big) \mathrm{H}(u_i) + \Big( \frac{2\pi q_i}{\theta_i} \Big) \mathrm{H}(v_i) = 2\pi \sqrt{-1}.$$
In particular, the generalized Dehn invariant of $\mathrm{H}(u_i) $ is given by $(\frac{2\pi p_i}{\theta_i} , \frac{2\pi q_i}{\theta_i} )$, which satisfies
$$
\Big|\frac{2\pi p_i}{\theta_i}\Big| + \Big| \frac{2\pi q_i}{\theta_i} \Big|
= (|p_i| + |q_i|) \Big(\frac{2\pi }{\theta_i}\Big)
> |p_i| + |q_i| > C.
$$
Besides, for $j\in J$, if the cone angle $\theta_j \in (0,2\pi/C)$, then the equation
$ \mathrm{H}(u_j) = \theta_j \sqrt{-1}$
implies that
$$ \Big( \frac{2\pi}{\theta_j}\Big) \mathrm{H}(u_j) = 2\pi \sqrt{-1}. $$
In particular, the generalized Dehn invariant of $\mathrm{H}(u_j) $ is given by $(\frac{2\pi }{\theta_j} , 0 )$, which satisfies
$$
\Big|\frac{2\pi }{\theta_j}\Big| > C.
$$
As a result, whenever $\theta_j \in [0, \frac{2\pi}{C})$ for all $j\in J$, $|p_i|+|q_i|>C$ and $\theta_i \in [0,\pi)$ for all $i \in I$, 
we have $|\mathrm{H}(u_k)| < \delta$ for all $k=1,\dots,n$. The results follow from Lemma \ref{assat0}.
\end{proof}

\begin{proof}[Proof of Theorem \ref{mainthm0}, \ref{mainthm} and \ref{mainthmlarge}] 
By Lemma \ref{assat0}, \ref{assat1} and \ref{assat2}, the assumptions in Proposition \ref{lfc}, \ref{fail0}, \ref{fail1}, \ref{fail2} and \ref{errorest} are satisfied. Thus, by Proposition \ref{computation}, Proposition \ref{Fperror}, Proposition \ref{0FCasym}, Proposition \ref{fail0}, Proposition \ref{fail1}, Proposition \ref{fail2} and Proposition \ref{errorest}, we have
\begin{align}\label{pfmainthmasym}
    &RT_r(M, L, (\textbf{n}_I, \textbf{m}_J)) \notag\\
    =& Z_r \Big(\sum_{\mathbf E_I}\widehat{f_r}(\mathbf{s}^{\mathbf E_I}, \mathbf{1-2m^{E_I}}, \mathbf{0}) \Big) \lt(1 + O\lt(\frac{1}{r}\rt)\rt)\notag\\
    =& C  \frac{e^{\frac{1}{2}\sum_{k=1}^n\mu_k\mathrm H(\gamma_k)}}{\sqrt{\pm\mathbb{T}_{(M\setminus L,\mathbf m)}([\rho_{M_{L_{\boldsymbol\theta}}}])}}
    e^{\frac{r}{4\pi}\Big( (\Vol(M_{L_{\boldsymbol\theta}}) + \sqrt{-1}\CS(M_{L_{\boldsymbol\theta}})) \Big)} \bigg( 1 + O\Big(\frac{1}{r}\Big) \bigg),
\end{align}
where
\begin{align*}
C&= \frac{ (-1)^{\sum_{i\in I}(\zeta_i + 1 + \sum_{l=1}^{\zeta_i}a_{i,l})} (\sqrt{-1})^{\sum_{i\in I}\frac{\zeta_i -1}{2}}(-1)^{ - \frac{rc}{2} +\sum_{i\in I} \Big(a_{i,0}+\frac{\iota_i}{2}\Big)+\sum_{j\in J}\Big(a_{j,0}+\frac{\iota_j}{2}\Big)} }{\sqrt{-1}^{\sum_{i \in I} \zeta_i - c} } \\
&\qquad \times e^{\frac{\pi\sqrt{-1}}{r} \sum_{i\in I} \sum_{l=1}^{\zeta_i-1} a_{i,l} -\frac{r\pi \sqrt{-1}}{4} ( \sum_{i\in I} (a_{i,0} +a_{i,\zeta_i})  + \sum_{j\in J} a_{j,0}) +\sigma(\tilde{L}_{\text{FSL},I} \cup L')(\frac{3}{r} + \frac{r+1}{4})\sqrt{-1}\pi}\\
&\qquad \times e^{\frac{r\pi}{4\sqrt{-1}}\sum_{i\in I}\Big(4m_i^{+}-2 + K_i(s_i^{+})+\frac{p_i'}{q_i} \Big)}
\end{align*}
is a quantity of norm $1$ independent of the geometric structure on $M.$
\end{proof}

\begin{proof}[Proof of Theorem \ref{Cor} and \ref{Cor2}] 
From (\ref{pfmainthmasym}), we have
\begin{align*}
&\lim_{r\to \infty} \frac{4\pi}{r} \log RT_r(M, L, (\textbf{n}_I, \textbf{m}_J)) \\
=& \Vol(M_{L_{\boldsymbol\theta}}) + \sqrt{-1}\CS(M_{L_{\boldsymbol\theta}}) - 2c\pi^2\sqrt{-1}
+ \pi^2\sqrt{-1}\Big( \sum_{i\in I} (a_{i,0} +a_{i,\zeta_i})  + \sum_{j\in J} a_{j,0}) +\sigma(\tilde{L}_{\text{FSL},I} \cup L') \Big)\\
&-\pi^2\sqrt{-1}\sum_{i\in I}\Big( 4m_i^{+}-2 + K_i(s_i^{+})+\frac{p_i'}{q_i} \Big)\\
=& \Vol(M_{L_{\boldsymbol\theta}}) + \sqrt{-1}\CS(M_{L_{\boldsymbol\theta}}) \pmod{\pi^2\sqrt{-1}\ZZ},
\end{align*}
where in the last equality we apply Lemma \ref{arith} (3).
\end{proof}

\noindent
Tushar Pandey\\
Department of Mathematics\\  Texas A\&M University\\
College Station, TX 77843, USA\\
(tusharp@tamu.edu)\\

\noindent
Ka Ho Wong\\
Department of Mathematics\\  Texas A\&M University\\
College Station, TX 77843, USA\\
(daydreamkaho@tamu.edu)
\\

\end{document}